\documentclass[a4paper,10pt]{amsart}

\usepackage[utf8]{inputenc}
\usepackage{url}

\usepackage[expansion=false]{microtype}
\usepackage{amssymb,amsmath,amsopn,amsxtra,amsthm,amsfonts}
\usepackage{xr}
\usepackage[mathcal]{euscript}
\usepackage{mathalfa}

\usepackage[english, status=final, nomargin, inline]{fixme}
\fxusetheme{color}

\usepackage[pdfusetitle,unicode,hidelinks]{hyperref}

\def\elden#1{{\color{purple}#1}}

\def\vova#1{{\color{olive}#1}}

\usepackage{scalerel,stackengine}
\usepackage{mathtools}

\makeatletter
\newcommand{\raisemath}[1]{\mathpalette{\raisem@th{#1}}}
\newcommand{\raisem@th}[3]{\raisebox{#1}{$#2#3$}}
\makeatother

\newcommand\widecurlywedgeA[1]{%
\raisemath{-2pt}{\smash{\stackrel{\vstretch{0.5}{\hstretch{2.0}{\curlywedge}}}{#1}}}
}
\newcommand\widecurlywedge[1]{%
  \mathchoice
  {\displaystyle{\widecurlywedgeA{#1}}}
  {\textstyle{\widecurlywedgeA{#1}}}
  {\scriptstyle{\widecurlywedgeA{#1}}}
  {\scriptscriptstyle{\widecurlywedgeA{#1}}}
}

\usepackage{enumitem}
\setlist[enumerate,1]{label={(\arabic*)},itemsep=\parskip} 
\setlist[itemize,1]{itemsep=\parskip} 
\newlist{thmlist}{enumerate}{2}
\setlist[thmlist,1]{label={\em(\roman*)},ref={(\roman*)},%
  itemsep=\parskip,leftmargin=*,align=left}
\setlist[thmlist,2]{label={\em(\alph*)},ref={(\alph*)},%
  itemsep=\parskip,leftmargin=*,align=left,topsep=0.1cm}
\newlist{remlist}{enumerate}{2}
\setlist[remlist,1]{label={(\roman*)},ref={(\roman*)},itemsep=\parskip,%
  leftmargin=*,align=left}
\setlist[remlist,2]{label={(\alph*)},ref={(\alph*)},itemsep=\parskip,%
  leftmargin=*,align=left,topsep=0.1cm}

\makeatletter
\let\c@equation\c@subsubsection

\makeatother

\newtheorem{cor}[subsubsection]{Corollary}

\newtheorem{lem}[subsubsection]{Lemma}
\newtheorem{prop}[subsubsection]{Proposition}

\newtheorem{thm}[subsubsection]{Theorem}

\newtheorem*{claim*}{Claim}

\theoremstyle{definition}
\newtheorem{defn}[subsubsection]{Definition}
\newtheorem{quest}[subsubsection]{Question}

\newtheorem{rem}[subsubsection]{Remark}
\newtheorem{constr}[subsubsection]{Construction}
\newtheorem{exam}[subsubsection]{Example}

\renewcommand{\eqref}[1]{(\ref{#1})}

\usepackage{tikz}
\usetikzlibrary{matrix}
\usepackage{tikz-cd}

\usepackage{xspace}

\usepackage{xifthen}

\usepackage{xparse}

\newcommand{\nc}{\newcommand}
\nc{\renc}{\renewcommand}
\nc{\ssec}{\subsection}
\nc{\sssec}{\subsubsection}
\nc{\on}{\operatorname}
\nc{\term}[1]{#1\xspace}

\DeclareMathSymbol{A}{\mathalpha}{operators}{`A}
\DeclareMathSymbol{B}{\mathalpha}{operators}{`B}
\DeclareMathSymbol{C}{\mathalpha}{operators}{`C}
\DeclareMathSymbol{D}{\mathalpha}{operators}{`D}
\DeclareMathSymbol{E}{\mathalpha}{operators}{`E}
\DeclareMathSymbol{F}{\mathalpha}{operators}{`F}
\DeclareMathSymbol{G}{\mathalpha}{operators}{`G}
\DeclareMathSymbol{H}{\mathalpha}{operators}{`H}
\DeclareMathSymbol{I}{\mathalpha}{operators}{`I}
\DeclareMathSymbol{J}{\mathalpha}{operators}{`J}
\DeclareMathSymbol{K}{\mathalpha}{operators}{`K}
\DeclareMathSymbol{L}{\mathalpha}{operators}{`L}
\DeclareMathSymbol{M}{\mathalpha}{operators}{`M}
\DeclareMathSymbol{N}{\mathalpha}{operators}{`N}
\DeclareMathSymbol{O}{\mathalpha}{operators}{`O}
\DeclareMathSymbol{P}{\mathalpha}{operators}{`P}
\DeclareMathSymbol{Q}{\mathalpha}{operators}{`Q}
\DeclareMathSymbol{R}{\mathalpha}{operators}{`R}
\DeclareMathSymbol{S}{\mathalpha}{operators}{`S}
\DeclareMathSymbol{T}{\mathalpha}{operators}{`T}
\DeclareMathSymbol{U}{\mathalpha}{operators}{`U}
\DeclareMathSymbol{V}{\mathalpha}{operators}{`V}
\DeclareMathSymbol{W}{\mathalpha}{operators}{`W}
\DeclareMathSymbol{X}{\mathalpha}{operators}{`X}
\DeclareMathSymbol{Y}{\mathalpha}{operators}{`Y}
\DeclareMathSymbol{Z}{\mathalpha}{operators}{`Z}

\nc{\sA}{\ensuremath{\mathcal{A}}\xspace}
\nc{\sB}{\ensuremath{\mathcal{B}}\xspace}
\nc{\sC}{\ensuremath{\mathcal{C}}\xspace}
\nc{\sD}{\ensuremath{\mathcal{D}}\xspace}
\nc{\sE}{\ensuremath{\mathcal{E}}\xspace}
\nc{\sF}{\ensuremath{\mathcal{F}}\xspace}
\nc{\sG}{\ensuremath{\mathcal{G}}\xspace}
\nc{\sH}{\ensuremath{\mathcal{H}}\xspace}
\nc{\sI}{\ensuremath{\mathcal{I}}\xspace}
\nc{\sJ}{\ensuremath{\mathcal{J}}\xspace}
\nc{\sK}{\ensuremath{\mathcal{K}}\xspace}
\nc{\sL}{\ensuremath{\mathcal{L}}\xspace}
\nc{\sM}{\ensuremath{\mathcal{M}}\xspace}
\nc{\sN}{\ensuremath{\mathcal{N}}\xspace}
\nc{\sO}{\ensuremath{\mathcal{O}}\xspace}
\nc{\sP}{\ensuremath{\mathcal{P}}\xspace}
\nc{\sQ}{\ensuremath{\mathcal{Q}}\xspace}
\nc{\sR}{\ensuremath{\mathcal{R}}\xspace}
\nc{\sS}{\ensuremath{\mathcal{S}}\xspace}
\nc{\sT}{\ensuremath{\mathcal{T}}\xspace}
\nc{\sU}{\ensuremath{\mathcal{U}}\xspace}
\nc{\sV}{\ensuremath{\mathcal{V}}\xspace}
\nc{\sW}{\ensuremath{\mathcal{W}}\xspace}
\nc{\sX}{\ensuremath{\mathcal{X}}\xspace}
\nc{\sY}{\ensuremath{\mathcal{Y}}\xspace}
\nc{\sZ}{\ensuremath{\mathcal{Z}}\xspace}

\nc{\bA}{\ensuremath{\mathbf{A}}\xspace}
\nc{\bB}{\ensuremath{\mathbf{B}}\xspace}
\nc{\bC}{\ensuremath{\mathbf{C}}\xspace}
\nc{\bD}{\ensuremath{\mathbf{D}}\xspace}
\nc{\bE}{\ensuremath{\mathbf{E}}\xspace}
\nc{\bF}{\ensuremath{\mathbf{F}}\xspace}
\nc{\bG}{\ensuremath{\mathbf{G}}\xspace}
\nc{\bH}{\ensuremath{\mathbf{H}}\xspace}
\nc{\bI}{\ensuremath{\mathbf{I}}\xspace}
\nc{\bJ}{\ensuremath{\mathbf{J}}\xspace}
\nc{\bK}{\ensuremath{\mathbf{K}}\xspace}
\nc{\bL}{\ensuremath{\mathbf{L}}\xspace}
\nc{\bM}{\ensuremath{\mathbf{M}}\xspace}
\nc{\bN}{\ensuremath{\mathbf{N}}\xspace}
\nc{\bO}{\ensuremath{\mathbf{O}}\xspace}
\nc{\bP}{\ensuremath{\mathbf{P}}\xspace}
\nc{\bQ}{\ensuremath{\mathbf{Q}}\xspace}
\nc{\bR}{\ensuremath{\mathbf{R}}\xspace}
\nc{\bS}{\ensuremath{\mathbf{S}}\xspace}
\nc{\bT}{\ensuremath{\mathbf{T}}\xspace}
\nc{\bU}{\ensuremath{\mathbf{U}}\xspace}
\nc{\bV}{\ensuremath{\mathbf{V}}\xspace}
\nc{\bW}{\ensuremath{\mathbf{W}}\xspace}
\nc{\bX}{\ensuremath{\mathbf{X}}\xspace}
\nc{\bY}{\ensuremath{\mathbf{Y}}\xspace}
\nc{\bZ}{\ensuremath{\mathbf{Z}}\xspace}

\nc{\dA}{\ensuremath{\mathds{A}}\xspace}
\nc{\dB}{\ensuremath{\mathds{B}}\xspace}
\nc{\dC}{\ensuremath{\mathds{C}}\xspace}
\nc{\dD}{\ensuremath{\mathds{D}}\xspace}
\nc{\dE}{\ensuremath{\mathds{E}}\xspace}
\nc{\dF}{\ensuremath{\mathds{F}}\xspace}
\nc{\dG}{\ensuremath{\mathds{G}}\xspace}
\nc{\dH}{\ensuremath{\mathds{H}}\xspace}
\nc{\dI}{\ensuremath{\mathds{I}}\xspace}
\nc{\dJ}{\ensuremath{\mathds{J}}\xspace}
\nc{\dK}{\ensuremath{\mathds{K}}\xspace}
\nc{\dL}{\ensuremath{\mathds{L}}\xspace}
\nc{\dM}{\ensuremath{\mathds{M}}\xspace}
\nc{\dN}{\ensuremath{\mathds{N}}\xspace}
\nc{\dO}{\ensuremath{\mathds{O}}\xspace}
\nc{\dP}{\ensuremath{\mathds{P}}\xspace}
\nc{\dQ}{\ensuremath{\mathds{Q}}\xspace}
\nc{\dR}{\ensuremath{\mathds{R}}\xspace}
\nc{\dS}{\ensuremath{\mathds{S}}\xspace}
\nc{\dT}{\ensuremath{\mathds{T}}\xspace}
\nc{\dU}{\ensuremath{\mathds{U}}\xspace}
\nc{\dV}{\ensuremath{\mathds{V}}\xspace}
\nc{\dW}{\ensuremath{\mathds{W}}\xspace}
\nc{\dX}{\ensuremath{\mathds{X}}\xspace}
\nc{\dY}{\ensuremath{\mathds{Y}}\xspace}
\nc{\dZ}{\ensuremath{\mathds{Z}}\xspace}

\nc{\bbA}{\ensuremath{\mathbb{A}}\xspace}
\nc{\bbB}{\ensuremath{\mathbb{B}}\xspace}
\nc{\bbC}{\ensuremath{\mathbb{C}}\xspace}
\nc{\bbD}{\ensuremath{\mathbb{D}}\xspace}
\nc{\bbE}{\ensuremath{\mathbb{E}}\xspace}
\nc{\bbF}{\ensuremath{\mathbb{F}}\xspace}
\nc{\bbG}{\ensuremath{\mathbb{G}}\xspace}
\nc{\bbH}{\ensuremath{\mathbb{H}}\xspace}
\nc{\bbI}{\ensuremath{\mathbb{I}}\xspace}
\nc{\bbJ}{\ensuremath{\mathbb{J}}\xspace}
\nc{\bbK}{\ensuremath{\mathbb{K}}\xspace}
\nc{\bbL}{\ensuremath{\mathbb{L}}\xspace}
\nc{\bbM}{\ensuremath{\mathbb{M}}\xspace}
\nc{\bbN}{\ensuremath{\mathbb{N}}\xspace}
\nc{\bbO}{\ensuremath{\mathbb{O}}\xspace}
\nc{\bbP}{\ensuremath{\mathbb{P}}\xspace}
\nc{\bbQ}{\ensuremath{\mathbb{Q}}\xspace}
\nc{\bbR}{\ensuremath{\mathbb{R}}\xspace}
\nc{\bbS}{\ensuremath{\mathbb{S}}\xspace}
\nc{\bbT}{\ensuremath{\mathbb{T}}\xspace}
\nc{\bbU}{\ensuremath{\mathbb{U}}\xspace}
\nc{\bbV}{\ensuremath{\mathbb{V}}\xspace}
\nc{\bbW}{\ensuremath{\mathbb{W}}\xspace}
\nc{\bbX}{\ensuremath{\mathbb{X}}\xspace}
\nc{\bbY}{\ensuremath{\mathbb{Y}}\xspace}
\nc{\bbZ}{\ensuremath{\mathbb{Z}}\xspace}

\nc{\mrm}[1]{\ensuremath{\mathrm{#1}}\xspace}
\nc{\mbf}[1]{\ensuremath{\mathbf{#1}}\xspace}
\nc{\mcal}[1]{\ensuremath{\mathcal{#1}}\xspace}
\nc{\msc}[1]{\ensuremath{\mathscr{#1}}\xspace}

\renc{\bar}[1]{\overline{#1}}

\nc{\sub}{\subset}
\nc{\too}{\longrightarrow}
\nc{\hook}{\hookrightarrow}
\nc*{\hooklongrightarrow}{\ensuremath{\lhook\joinrel\relbar\joinrel\rightarrow}}
\nc{\hooklong}{\hooklongrightarrow}
\nc{\twoheadlongrightarrow}{\relbar\joinrel\twoheadrightarrow}
\nc{\shiso}{\approx}
\nc{\isoto}{\xrightarrow{\sim}}
\nc{\isofrom}{\xleftarrow{\sim}}
\renc{\ge}{\geqslant}
\renc{\le}{\leqslant}
\renc{\geq}{\geqslant}
\renc{\leq}{\leqslant}

\nc{\id}{\mathrm{id}}

\DeclareMathOperator{\Hom}{\mathrm{Hom}}
\nc{\uHom}{\underline{\smash{\Hom}}}
\DeclareMathOperator{\Maps}{\mathrm{Maps}}

\DeclareMathOperator{\End}{\mathrm{End}}
\DeclareMathOperator{\Sym}{\mathrm{Sym}}
\nc{\Pre}{\mathrm{PSh}{}}
\nc{\Shv}{\mathrm{Shv}{}}
\nc{\uEnd}{\underline{\smash{\End}}}

\renc{\lim}{\operatorname*{lim}}
\nc{\colim}{\operatorname*{colim}}
\nc{\Cofib}{\on{Cofib}}
\nc{\Fib}{\on{Fib}}
\nc{\initial}{\varnothing}
\nc{\op}{\mathrm{op}}
\usepackage{graphicx}

\let\bigcoprod=\coprod
\renc{\coprod}{\sqcup}

\usepackage{relsize}
\usepackage[bbgreekl]{mathbbol}
\usepackage{amsfonts}
\DeclareSymbolFontAlphabet{\mathbb}{AMSb} 
\DeclareSymbolFontAlphabet{\mathbbl}{bbold}

\DeclareFontFamily{T1}{cbgreek}{}
\DeclareFontShape{T1}{cbgreek}{m}{n}{<-6>  grmn0500 <6-7> grmn0600 <7-8> grmn0700 <8-9> grmn0800 <9-10> grmn0900 <10-12> grmn1000 <12-17> grmn1200 <17-> grmn1728}{}
\DeclareSymbolFont{quadratics}{T1}{cbgreek}{m}{n}
\DeclareMathSymbol{\qoppa}{\mathord}{quadratics}{19}
\DeclareMathSymbol{\Qoppa}{\mathord}{quadratics}{21}

\nc{\bDelta}{\mbf{\Delta}}
\nc{\DM}{\mbf{DM}}
\nc{\eff}{\mathrm{eff}}
\nc{\veff}{\mathrm{veff}}
\nc{\cyc}{{\mrm{cyc}}}
\nc{\corr}{{\on{corr}}}
\nc{\ft}{\mrm{ft}}
\nc{\flf}{\mrm{flf}}
\nc{\fet}{{\mrm{f\acute et}}}
\nc{\fsyn}{{\mrm{fsyn}}}
\nc{\syn}{{\mrm{syn}}}
\nc{\lci}{{\mrm{lci}}}
\nc{\Perf}{\mbf{Perf}}
\nc{\LSym}{L\mrm{Sym}}
\nc{\perf}{\mrm{perf}}
\nc{\oblv}{\mrm{oblv}}
\nc{\exact}{\on{exact}}

\nc{\F}{{\on{F}}}
\nc{\clopen}{{\mrm{clopen}}}
\nc{\B}{\mrm{B}}
\nc{\D}{\mrm{D}}
\nc{\Fin}{\on{Fin}}
\nc{\fin}{\mrm{fin}}
\nc{\Cut}{\on{Cut}}
\nc{\Cart}{\on{Cart}}
\nc{\pairs}{\mathsf{pairs}}
\nc{\Pairs}{\mathrm{Pair}}
\nc{\Trip}{\mathrm{Trip}}
\nc{\Lab}{\mathrm{Lab}}
\nc{\SL}{\mathrm{SL}}
\nc{\coCart}{\mathrm{coCart}}
\nc{\RKE}{\mathrm{RKE}}
\nc{\strict}{\mathrm{strict}}
\nc{\Emb}{\mathrm{Emb}}
\nc{\Split}{\mathrm{Split}}
\nc{\Set}{\mathrm{Set}}
\nc{\sSets}{\mathrm{sSets}}
\nc{\pb}{\mathrm{pb}}
\nc{\fib}{\mathrm{fib}}
\nc{\diff}{\mrm{diff}}
\nc{\gp}{\mrm{gp}}
\nc{\map}{\mrm{map}}

\nc{\mgp}{\mrm{mot-gp}}
\nc{\FSyn}{\mrm{FSyn}}
\nc{\FEt}{\mrm{FEt}}
\nc{\Spc}{\mrm{Spc}}
\nc{\Ob}{\mrm{Ob}}
\nc{\Spt}{\mrm{Spt}}
\nc{\T}{\bT}
\nc{\suspinf}{\Sigma^\infty}
\nc{\h}{\mrm{h}}
\nc{\uhom}{\underline{\mathrm{Hom}}}
\nc{\umap}{\underline{\mathrm{Maps}}}
\renc{\H}{\bH}
\nc{\Einfty}{{\sE_\infty}}
\nc{\Eone}{{\sE_1}}
\nc{\Stab}{\mrm{Stab}}
\nc{\lax}{{\mrm{lax}}}
\nc{\cocart}{{\mrm{cocart}}}
\nc{\Sch}{\on{Sch}}
\nc{\Fr}{\on{Fr}}
\nc{\A}{\mathbf{A}}
\nc{\N}{\mathbf{N}}
\nc{\Z}{\mathbf{Z}}
\nc{\Q}{\mathbf{Q}}
\nc{\Oo}{\mathcal{O}} 
\nc{\Fscr}{\mathcal{F}}
\nc{\Gscr}{\mathcal{G}}
\nc{\Ll}{\mathcal{L}} 
\nc{\Mm}{\mathcal{M}} 
\nc{\mm}{\mathrm{m}} 
\nc{\K}{\mrm{K}} 
\nc{\W}{\mrm{W}} 
\nc{\red}{{\on{red}}}
\nc{\Voev}{{\on{Voev}}}
\nc{\Corr}{\mrm{Corr}}
\nc{\Span}{\mathbf{Corr}}
\nc{\Gap}{\mrm{Gap}}
\nc{\Corrfr}{\Corr^{\fr}}
\nc{\Corrvfr}{\Corr^{\Vfr}}
\nc{\Spec}{\on{Spec}}
\nc{\Sm}{\on{Sm}}
\nc{\Gm}{\mathbf{G}_{\on{m}}}
\renc{\P}{\bP}
\nc{\nis}{\mathrm{nis}}
\nc{\zar}{\mathrm{zar}}
\nc{\et}{\mathrm{\acute et}}
\nc{\all}{\mathrm{all}}
\nc{\fold}{\mathrm{fold}}
\nc{\Fun}{\mathrm{Fun}}
\nc{\Ho}{\mathrm{Ho}}
\nc{\Segal}{\mathrm{Segal}}
\nc{\Mon}{\mrm{Mon}{}}
\nc{\Ab}{\mrm{Ab}}
\nc{\Sh}{\on{Sh}}
\nc{\M}{\mrm{M}}
\nc{\Lhtp}{L_{\A^1}}
\nc{\Lmot}{L_{\mrm{mot}}}
\nc{\mot}{\mrm{mot}}
\nc{\SH}{\mbf{SH}}
\nc{\RR}{\mbf{R}}
\nc{\CC}{\mbf{C}}
\nc{\Mod}{\mbf{Mod}}
\nc{\QCoh}{\mbf{QCoh}}
\nc{\MonUnit}{\mbf{1}}
\nc{\1}{\mbf{1}}
\nc{\tr}{\on{tr}}
\nc{\cotr}{\mrm{cotr}}
\nc{\vop}{\mrm{vop}}
\nc{\fr}{{\on{fr}}}
\nc{\Ar}{\mrm{Ar}}
\nc{\Vfr}{\on{Vfr}}
\nc{\frdiff}{{\on{frdiff}}}
\nc{\frGys}{\on{frGys}}
\nc{\SHfr}{\SH^{\fr}}
\nc{\SHfrdiff}{\SH^{\frdiff}}
\nc{\SHfrGys}{\SH^{\frGys}}
\nc{\InftyCat}{(\mathrm{\infty,1)\textnormal{-}Cat}}
\nc{\TriCat}{\mathrm{TriCat}}
\nc{\oneCat}{\mathrm{1\textnormal{-}Cat}}
\nc{\Cat}{\mathrm{Cat}}
\nc{\Th}{\on{Th}}

\nc{\CMon}{\mrm{CMon}{}}
\nc{\CAlg}{\mrm{CAlg}{}}
\nc{\MGL}{\mrm{MGL}}
\nc{\Seg}{\mrm{Seg}{}}
\nc{\GW}{\mrm{GW}{}}
\nc{\Tw}{\mrm{Tw}}
\nc{\sslash}{/\mkern-6mu/}
\nc{\PrL}{\mrm{Pr}^\mrm{L}}
\nc{\PrR}{\mrm{Pr}^\mrm{R}}
\nc{\pr}{\mrm{pr}}
\let\phi\varphi

\nc\efr{\mrm{efr}}
\nc\nfr{\mrm{nfr}}
\nc\dfr{\mrm{fr}}
\nc\tfr{\mrm{tfr}}
\nc\Vect{\mrm{Vect}}
\nc\sVect{\mrm{sVect}}
\nc{\fix}{\mrm{fix}}
\nc{\ho}{\mrm{h}}
\nc\Mfd{\mrm{Mfd}}
\nc{\PSh}{\mrm{PSh}}
\nc{\hzmw}{H \tilde\Z{}}
\nc{\Cor}{\mrm{Cor}{}}
\nc{\cormw}{\mrm{\widetilde{Cor}}{}}
\nc{\Chw}{\mrm{\widetilde{CH}}{}}
\nc{\Ex}{\mrm{Ex}}
\nc{\BM}{\mrm{BM}}
\let\setminus\smallsetminus
\nc{\Pic}{\mrm{Pic}}
\nc{\Br}{\mrm{Br}}
\nc{\pur}{\mathfrak p}
\nc{\angles}[1]{\langle #1\rangle}
\nc{\inv}[1]{[\tfrac{1}{#1}]}
\nc{\pinv}{\inv{p}}
\nc{\cinv}{\inv{p}}
\nc{\Sph}{\on{Sph}}
\nc{\KGL}{\mrm{KGL}}
\nc{\KH}{\mrm{KH}}
\nc{\Flag}{\mrm{Flag}}
\nc{\Pro}{\mrm{Pro}}
\nc{\Frac}{\mrm{Frac}}

\nc{\NCMot}{\mbf{NCMot}}

\nc{\arc}{\mrm{arc}}
\nc{\rarc}{\mrm{rarc}}
\nc{\cdarc}{\mrm{cdarc}}
\nc{\vv}{\mrm{v}}
\nc{\rv}{\mrm{rv}}
\nc{\cdv}{\mrm{cdv}}
\nc{\hh}{\mrm{h}}
\nc{\cdh}{\mrm{cdh}}
\nc{\rh}{\mathrm{rh}}
\nc{\Et}{\mathrm{Et}}
\nc{\Nis}{\mathrm{Nis}}
\nc{\Zar}{\mathrm{Zar}}
\nc{\cdp}{\mathrm{cdp}}

\nc{\RZ}{\mathrm{RZ}}
\nc{\qcqs}{\mathrm{qcqs}}
\nc{\aff}{\mathrm{aff}}
\nc{\cl}{\mathrm{cl}}
\nc{\Val}{\mathrm{Val}}
\nc{\GFin}{\mathrm{GFin}{}}
\nc{\Proj}{\mathrm{Proj}}
\nc{\Ind}{\mathrm{Ind}}

\DeclareMathOperator{\Spf}{Spf}
\nc{\mbb}{\mathbb}
\nc{\mc}{\mathcal}
\nc{\ra}{\rightarrow}
\nc{\oop}{\mrm{op}}

\nc{\inftyCat}{\term{$\infty$-category}}
\nc{\inftyCats}{\term{$\infty$-categories}}

\nc{\inftyOneCat}{\term{$(\infty,1)$-category}}
\nc{\inftyOneCats}{\term{$(\infty,1)$-categories}}

\nc{\inftyGrpd}{\term{$\infty$-groupoid}}
\nc{\inftyGrpds}{\term{$\infty$-groupoids}}

\nc{\inftyTop}{\term{$\infty$-topos}}
\nc{\inftyTops}{\term{$\infty$-toposes}}

\nc{\inftyTwoCat}{\term{$(\infty,2)$-category}}
\nc{\inftyTwoCats}{\term{$(\infty,2)$-categories}}

\usepackage{yhmath}
\usepackage{bbm}

\title{On filtered algebraic $K$-theory of stacks I: characteristic zero}

\author[E. Elmanto]{Elden Elmanto}
\address{
University of Toronto\\
40 St. George St. \\
M5S 2E4 Ontario\\
Canada
}
\email{\href{mailto:elden.elmanto@utoronto.ca}{elden.elmanto@utoronto.ca}}

\author[D. Kubrak]{Dmitry Kubrak}
\author[V. Sosnilo]{Vladimir Sosnilo}
\address{
M309\\
Universit{\"a}t Regensburg\\
Universit{\"a}tsstra{\ss}e 31 \\
93053 Regensburg\\
Germany
}
\email{\href{mailto:vsosnilo@gmail.com}{vsosnilo@gmail.com}}

\begin{document}

\begin{abstract} 
Given a compact Lie group $G$ acting on a space $X$, the classical Atiyah-Segal completion theorem identifies topological $K$-theory of the homotopy quotient $X/G$ with an explicit completion of $G$-equivariant topological $K$-theory of $X$. We prove an analog of this result for algebraic $K$-theory over a field of characteristic 0. In our setting $G$ is a reductive group that acts on a derived algebraic space $X$ with the assumption that all stabilizer groups are nice (in the sense of Alper). Our main result identifies the value $R^{\mathrm{dAff}}K([X/G])$ of right Kan extension of the $K$-theory functor from schemes to stacks with the completion of $K$-theory of the category $\mathrm{Perf}([X/G])$ at the augmentation ideal of $K_0(\mathrm{Rep}(G))$. The main novelty of our results is that $X$ is allowed to be singular or even derived. This generality is achieved by employing and improving analogous versions of completion theorem for negative cyclic homology (after Ben-Zvi--Nadler and Chen) and for homotopy $K$-theory (after van den Bergh--Tabuada). We also show that in the singular setting the completion theorem does not necessarily hold without the nice stabilizer assumption. We view our results as a part of the general paradigm of extending the motivic filtration on algebraic $K$-theory of schemes to algebraic $K$-theory of stacks.
\end{abstract}

\maketitle
\setcounter{tocdepth}{1}
\tableofcontents

\section{Introduction}

\renewcommand{\Perf}{\mathrm{Perf}}
\renewcommand{\Mod}{\mathrm{Mod}}
 
 The goal of this paper is to offer computational and structural 
results about the nonconnective algebraic $K$-theory of a large class of algebraic stacks over $\mathbb{Q}$, i.e., those which are of equicharacteristic zero. The meaning of the algebraic 
$K$-theory of a stack can be ambiguous. On the one hand, after the work of \cite{blumberg2013universal}, we know that algebraic $K$-theory is a universal functor
\[
K\colon \Cat^{\perf}_{\infty} \rightarrow \Spt,
\]
enjoying the key property that it respects certain localization sequences. In particular, the $K$-theory of a stack $\sX$ can be defined as $K(\sX):=K(\Perf(\sX))$\footnote{Throughout this paper, by $\Perf(\sX)$ we mean the subcategory of $\mrm{D}_{\mrm{qc}}(\sX)$ spanned by dualizable objects. In contrast to the situation for quasi-compact quasi-separated schemes, subcategories of $\mrm{D}_{\mrm{qc}}(X)$ spanned by compact and, respectively, dualizable objects need not coincide. Furthermore, the $\infty$-category $\mrm{D}_{\mrm{qc}}(\sX)$ is not necessarily compactly generated \cite{one-pos-two-neg}.}. This has the advantage that we can sometimes describe $K(\sX)$ in rather simple terms. For example, if $\sX$ is the classifying stack $B\bbG_m$ over a field $k$ then, identifying $\Perf(B\bbG_m)$ with graded perfect complexes over $k$, yields:
\begin{equation}\label{eq:bgm}
K(B\bbG_m) \simeq \bigoplus_{\bbZ} K(k),
\end{equation}
where each component $K(k)$ corresponds to a single weight. However, for more general $\sX$, we typically do not have such a good control of $\Perf(\sX)$, and it can be very hard to compute $K(\sX)$ explicitly.

 On the other hand, we know that the algebraic $K$-theory of schemes 
enjoys a wealth of structure, chief among which is the \emph{motivic filtration}. This is captured succinctly by the functor landing in filtered $\bbE_\infty$-algebras in spectra: 
\[
\mrm{Fil}_{\mot}^{\star}K: \left(\Sch^{\qcqs}_{\mathbb{Z}}\right)^{\op} \rightarrow \CAlg(\Fun(\bbZ^{\op},\Spt)).
\]
On sufficiently nice schemes the filtration $\mrm{Fil}_{\mot}^{\star}K$ is complete; its associated graded is given by the motivic cohomology $\mbb Z(\star)^{\mot}[2\star]$, which in general is much more accessible for computations. 
It is thus natural to contemplate the following problem:

\begin{quest} Does the motivic filtration on $K$-theory of schemes extend to algebraic stacks?
\end{quest}

We remark, however, that the value on algebraic stacks of any reasonably defined motivic filtration on $K$-theory
 will typically not be complete. Moreover, we do need completeness of the filtration to have many of the desired properties of $\mrm{Fil}_{\mot}^{\star}K$, as the following examples hopefully illustrate.

\begin{exam}\label{exam:bgmK} The formula~\eqref{eq:bgm} can be written multiplicatively as:
\[
K(B\bbG_m) \simeq K(k)[x^{\pm 1}] \qquad x = [\sO(1)] \in K_0(B\bbG_m)
\]
so that, in particular, we have that $K_0(B\bbG_m) \cong \mbb Z[x^{\pm1}]$. A motivic filtration on $K$-theory will then induce a filtration on $K_0$ whose graded pieces are displayed in the $E_{\infty}$-page of the corresponding Atiyah-Hirzebruch spectral sequence. One expects that any reasonable motivic filtration on stacks to have the following property as in the schemes case \cite[Lemma 5.10]{e-morrow}: the Euler class $e(\sO(1)) := 1-[\sO(1)]=1-x$ in $K_0(B\mbb G_m)$ lifts to motivic filtration degree $\geq 1$. Therefore if we want the motivic filtration on $K(B\bbG_m)$ be multiplicative and complete, then it forces $K_0(B\bbG_m)\simeq \mbb Z[x^{\pm 1}] $ to be $(x-1)$-adically complete, which it is clearly \emph{not}. Note that $(x-1)\subset K_0(B\bbG_m)$ is exactly the kernel of the pull-back map $K_0(B\bbG_m) \ra K_0(k)$ for the covering $\Spec k \ra B\mbb G_m$; this is a particular case of the so-called \textbf{Atiyah-Segal ideal}.
\end{exam}

\begin{exam}\label{exam:bgmKQ} This example is supposed to illustrate that some completion is necessary, if we want to have a splitting of $K(\sX)_{\mbb Q}$ into the product of Adams eigenspaces, i.e., a rational motivic filtration. Recall that for (sufficiently nice) schemes the motivic filtration on rationalized  $K$-theory splits functorially \cite[Appendix B]{e-morrow}: namely,
\begin{equation}\label{eq:motivic splitting rationally}
	K(X)_{\mbb Q} \simeq \oplus_{j\ge 0} K(X)_{\mbb Q}^{(j)} \simeq \oplus_{j\ge 0} \mbb Q(j)(X)^{\mot}[2j]
\end{equation}	
where the $j$-th graded piece is canonically identified with the weight $j$ Adams eigenspace.

 Let again $\sX=B\mbb G_m$. The $\ell$-th Adams operation $\psi_\ell$ acts on $K_0(X)_{\mbb Q}\simeq \mbb Q[x,x^{-1}]$ by sending $[\mc O(1)]$ to $[\mc O(1)^{\otimes \ell}]=[\mc O(\ell)]$; in other words we have $\psi_\ell(x)=x^\ell$. If $\ell>1$, it is clear that the only eigenvector of $\phi_\ell$ in $\mbb Q[x,x^{-1}]$ are the scalars (on which the action is trivial); thus we see that for $K(B\mbb G_m)_{\mbb Q}$ there cannot be any splitting  analogous to \eqref{eq:motivic splitting rationally}.

The function in $x$ that would generate ``$K(B\mbb G_m)_{\mbb Q}^{(j)}$" should satisfy the equation $f(x^\ell)=\ell^j \cdot f(x)$ for all $\ell$;  such a function is given by $\log(x^j)$. Note that $\log(x^j)$ makes perfect sense as an element of the $(x-1)$-adic completion $\mbb Q[x,x^{-1}]^{\wedge}_{(x-1)}\simeq \mbb Q[[x-1]]$ (but not before completion). Moreover, $\log(x^j)\in (x-1)^j\subset \mbb Q[[x-1]]$ and it's not hard to check that the map 
$$
\prod_{j\ge 0} \mbb Q\cdot \log(x^j) \rightarrow \mbb Q[[x-1]], \quad (a_j)_{j\ge 0} \mapsto \sum_{j\ge 0} a_j\cdot \log(x^j)
$$
is an isomorphism. This way we see that the splitting in \eqref{eq:motivic splitting rationally} (in the form of direct product) does generalize to $B\mbb G_m$ if we are willing to replace $K(B\mbb G_m)$ by the \textbf{Atiyah-Segal completion} $K(B\mbb G_m)^\wedge_{(x-1)}$.
\end{exam}	

 In this paper, we propose a candidate for the motivic filtration on $K(\sX)$. By construction\footnote{At least if the underlying classical stack $\sX^{\mrm{cl}}$ is of finite type.}, the completion of $\mrm{Fil}_{\mot}^{\star}K(\sX)$ is identified with the right Kan extension of $K$-theory from affine schemes. The main point of the paper, however, is that if $\sX=[X/G]$ this right Kan extension can often be described intrinsically in terms of $K(\sX)$ and the map $K(BG) \ra K(\sX)$ induced by the presentation of $\sX$ as a quotient stack. More precisely, it is the completion of $K(\sX)$ at the Atiyah-Segal ideal $I_G\subset K_0(BG)$ which, in the case of $B\bbG_m$ coincides with the ideal generated by the Euler class of $\sO(-1)$. In effect, we prove an analog of the Atiyah-Segal completion theorem \cite{atiyah-segal}, in a quite general form, for algebraic $K$-theory and other related localizing invariants. 

\subsection{Formulation of main results} Let $k$ be a base commutative ring. We recall that a finitely presented flat $k$-group scheme $G$ is called \textbf{nice} \cite[Definition 1.1]{HallRydh2} (resp. \textbf{linearly reductive} \cite[Definition 12.1]{alper-good})  if it sits in an extension
\[
1 \rightarrow G^0 \rightarrow G \rightarrow H \rightarrow 1,
\]
where $G^0$ is an algebraic torus and $H$ is a finite, locally constant group scheme whose order is invertible in $k$ (resp. the functor of derived invariants is $t$-exact). The notion of a nice group appears in the stacks literature where the notion of an \textbf{ANS stack} (see Definition~\ref{defn:ANS stack}) has been isolated as a key condition that allows for many structural results that one would like to use in practice. By definition, ANS stacks only have nice groups as stabilizers of points.

For a linearly reductive $G$ we consider $K_0(BG)$ which is naturally identified with the (algebraic) representation ring of $G$ whenever $k$ is a field; we recall that the notion of linear reductivity and reductivity agrees for $G$ an algebraic group over a field $k$ of characteristic zero. Now further assume that $k$ is a field of characteristic 0 and let $G$ be a reductive group over $k$. The \textbf{Atiyah-Segal ideal}, denoted by $I_G$, by definition is the kernel of the ``rank" map $K_0(BG) \rightarrow K_0(k)\simeq \mbb Z$ (induced by the covering $\Spec k \ra BG$). Let $X$ be a derived algebraic space over $k$ and assume that we have an action of $G$ on $X$; we have a natural map $\sX:=[X/G]\ra BG$ inducing a $K(BG)$-module structure on $K(\sX)$.  Our main theorem identifies the derived $I_G$-adic completion of $K(BG)$ (in the sense of spectral algebraic geometry \cite[Chapter 7]{SAG}) with the right Kan extension of $K$-theory from affine schemes in some generality.

\begin{thm}\label{thm:main_intro} Fix $G$ a reductive group over a characteristic zero field $k$. Let $\sX \rightarrow BG$ be a representable morphism of derived algebraic stacks, such that the morphism of classical stacks $\sX^{\mrm{cl}} \rightarrow BG$ is of finite type. Assume further that $\sX$ is ANS. Then we have a canonical equivalence:
\begin{equation}\label{eq:main_intro}
K(\sX)^{\wedge}_{I_G} \xrightarrow{\simeq} \lim_{\Spec R \rightarrow \sX} K(R),
\end{equation}
where the limit is taken across all morphisms from derived affine schemes to $\sX$.
\end{thm}

One can also rephrase the statement above as follows: up to completion at the ideal $I_G$, the $K$-theory of the stack $\sX$ is in fact  right Kan extended from derived affine schemes. We will call this type of result ``Atiyah-Segal (AS) completion theorem".

Let us now comment on what sort of algebraic stacks $\sX$ are allowed by Theorem \ref{thm:main_intro}.

\begin{rem}\label{rem:ans-bg} The condition that the morphism $\sX \rightarrow BG$ is representable implies that $\sX$ can be written as $[X/G]$ where $X$ is a derived algebraic space over $k$. The condition that $\sX$ is ANS\footnote{In our convention derived algebraic spaces have affine diagonal.} means that the stabilizers at all points are nice groups. Let us emphasize however that the group $G$ that appears in Theorem~\ref{thm:main_intro} need not be nice. 
	We also can show that the result is not necessarily true if we do not assume that $\sX$ is ANS: see the remark just below.
	
\end{rem}

\begin{rem}[Counterexample to Atiyah-Segal completion] \label{rem:counterexample} In Section~\ref{sec:counterexample}, we show that the AS completion theorem for $K$-theory is \emph{false} in general. The counterexample we offer is in the $\mbb Q$-stack $B\bbG_a\times \Spec {\mathbb{Q}[\varepsilon]/\varepsilon^2}$, viewed as a stack over $BSL_2$ via the identification $B\mbb G_a\simeq [\mbb A^2\setminus \{0\}/SL_2]$. Our calculation is based on the Levy's generalization \cite{levy2022algebraic} of Dundas--Goodwillie--McCarthy theorem to $(-1)$-connective rings which reduces it to an explicit analysis of the negative cyclic homology $HC^-$.
\end{rem}

The right-hand-side of the isomorphism in Theorem~\ref{thm:main_intro} is the value of $R^{\mrm{dAff}}K$, the right Kan extension of the $K$-theory of derived affine schemes to stacks, on the stack $\sX$. This invariant carries a natural filtration given by right Kan extension of the motivic filtration on schemes. We define \textbf{motivic filtration} on $K(\sX)$ via the pullback square:
\begin{equation}\label{eq:motfilt-stk-intro}
\begin{tikzcd}
\mrm{Fil}_{\mot}^{\star}K(\sX) \ar{d} \ar{r} & R^{\mrm{dAff}}\mrm{Fil}_{\mot}^{\star}K(\sX)\ar{d}\\
K(\sX) \ar{r} & R^{\mrm{dAff}}K(\sX).
\end{tikzcd}
\end{equation}
As explained in Example~\ref{exam:bgmK}, the motivic filtration on $K(\sX)$ will often not be complete. On the other hand, at least if $\sX^{\mrm{cl}}$ is of finite type, $R^{\mrm{dAff}}\mrm{Fil}_{\mot}^{\star}K(\sX)$ is complete, since it can be written as a limit of $K$-theory of finite type affine schemes, where motivic filtration is complete. This forces identification of the completion of $K(\sX)$ with respect to motivic filtration and $R^{\mrm{dAff}}\mrm{Fil}_{\mot}^{\star}K(\sX)$. In this terms, Theorem~\ref{thm:main_intro} can also be seen as a way to characterize this completion in a way that is more intrinsic to $\sX$. 

\begin{cor}\label{cor:main} Let $\sX$ be as in Theorem~\ref{thm:main_intro}, then the completion of $K(\sX)$ with respect to the motivic filtration in~\eqref{eq:motfilt-stk-intro} is given by $K(\sX)^{\wedge}_{I_G}$.
\end{cor}

We view this statement as our main contribution towards understanding motivic filtrations on $K$-theory beyond the case of schemes.

The main novelty of Theorem~\ref{thm:main_intro} is that it works generally without assuming any smoothness hypotheses on $\sX$. In this situation algebraic $K$-theory is not $\mbb A^1$-homotopy invariant, so many standard tools can not be applied. As with other results concerning algebraic $K$-theory of non-smooth schemes, we employ trace methods which are centered around the transformation $K \rightarrow TC$, from $K$-theory to \textbf{topological cyclic homology}. This map fits as the top part of the following square:
\begin{equation}\label{eq:square}
\begin{tikzcd}
K \arrow[d]\arrow[r] & TC\arrow[d]\\
KH \arrow[r] & L_{\mrm{cdh}}TC.
\end{tikzcd}
\end{equation}
When restricted to quasicompact and quasiseparated schemes, this square is cartesian thanks to results of Kerz-Strunk-Tamme \cite{KST} and Land-Tamme \cite{land-tamme} (we also refer to \cite[Theorem 3.8]{e-morrow} on how to deduce the general case from the noetherian case). The cartesian-ness of this square encodes two key phenomena:
\begin{enumerate}
\item that Weibel's homotopy $K$-theory $KH$ identifies with $\cdh$-sheafified $K$-theory \cite[Theorem 6.3]{KST};
\item that the fiber of the cyclotomic trace map is a $\cdh$ sheaf \cite[Corollary 3.6]{land-tamme}.
\end{enumerate}
In characteristic zero, $TC(-) \simeq HC^{-}(-/\mbb Q)$, the invariant known as \textbf{negative cyclic homology}. In this setting, these 
two ingredients are somewhat older; (1) was proved by Haesemeyer in his thesis \cite{Haesemeyer} and (2) was proved by 
Corti\~{n}as \cite{cortinas}. Our work includes an extension of the square~\eqref{eq:square} to the case of ANS stacks, at 
least assuming resolution of singularities, which allows to reduce Theorem~\ref{thm:main_intro} for $K$-theory to analogous results for its $\mathbb{A}^1$-invariant approximation (known as $KH$) and for Hochschild homology 
(and related theories, like $HC^-$). In the next two sections of the introduction we will discuss the relevant variants of Theorem~\ref{thm:main_intro}: namely, for $KH$ and other $\mathbb{A}^1$-invariant 
localizing invariants (in Section \ref{ssec:intro:A^1-invariant stuff}), and for localizing invariants born out of Hochshild homology (in Section \ref{ssec:intro:completion theorem for HH}). Let us emphasize that the methods used in these two contexts are very different: the first relies crucially on equivariant resolution of singularities and cdh-descent, while the second uses the geometry of corresponding loop stacks.
%

\subsection{Completion theorems for $\bbA^1$-invariant localizing invariants}\label{ssec:intro:A^1-invariant stuff}

Theorem~\ref{thm:main_intro} fits into the following general framework. Suppose that $E$ is a localizing invariant and that we have a morphism $f\colon \sX \rightarrow BG$ where $\sX$ is a perfect stack; for the formulation of the statements, we do not need to assume anything about $f$, while the perfect stack\footnote{Recall that an algebraic stack is called \textit{perfect} if its derived category of quasi-coherent sheaves $D_{\mrm{qc}}(\sX)$ is compactly generated by perfect complexes: in other words, if $D_{\mrm{qc}}(\sX)\simeq \Ind(\Perf(\sX))$.  In characteristic 0 many stacks are perfect\footnote{More generally, any quotient stack of a quasi-projective scheme by a linear algebraic group is perfect.}, in particular any quasi-compact ANS stack is. } assumption is simply to let us make sense of $E(\sX)$ in terms of perfect complexes over $\sX$. Given a $k$-linear localizing invariant $E$, $E(\sX)$ naturally promotes to a $K(BG)$-module, and we obtain a $K(BG)$-linear map $E(\sX)\ra R^{\mrm{dAff}}E(\sX)$. It is not hard to see (Corollary \ref{cor:Kan extension is I_G-complete}) that the target here is always $I_G$-complete, and thus we obtain an induced map
\begin{equation}\label{eq:asmap-intro}
E(\sX)^\wedge_{I_G} \ra R^{\mrm{dAff}}E(\sX).
\end{equation}
Given this setup, we say that the \textbf{Atiyah-Segal (AS) completion theorem holds} for $E, G, \sX$ if~\eqref{eq:asmap-intro} is an equivalence. While there are no a priory reasons for this to always hold, we note that there are interesting precedents to the veracity of this statement.

%
%

%
%

So far, most AS completion theorems in the literature have involved those $E$'s which are also $\mathbb{A}^1$-invariant. In this paper, using equivariant resolution of singularities and cdh-descent, we bootstrap the results of Tabuada-van den Bergh \cite{tabuada2020motivic} for quotient stacks of a smooth projective variety with a torus action, and get the following general form of AS completion theorem in the $\mbb A^1$-invariant context. 

\begin{thm}\label{thm:main_2}  Fix $G$ a reductive group over a characteristic zero field $k$ and $E$ a truncating localizing 
invariant which is also $\mathbb{A}^1$-invariant. Let $\sX \rightarrow BG$ be representable morphism. Assume that $\sX$ is 
ANS and $\sX^{\mrm{cl}} \rightarrow BG$ is finite type. Then we have an equivalence:
\begin{equation}\label{eq:main_introE}
E(\sX)^{\wedge}_{I_G} \xrightarrow{\simeq} R^{\mrm{dAff}}E(\sX).
\end{equation}
\end{thm}

The conditions of being truncating and $\mathbb{A}^1$-invariant are allied. The former implies that $E$ satisfies cdh descent, as explained in \cite{land-tamme}. On the other hand, for those $E$'s which are $\mathbb{A}^1$-invariant and also left Kan extended from smooth $\mathbb{Z}$-algebras Cisinski's theorem in motivic homotopy theory \cite{Cisinski} implies cdh descent for $E$. A combination of $\mathbb{A}^1$-invariance and cdh descent is necessary to prove Theorem~\ref{thm:main_2} in the stated generality.  Previous incarnations of such theorems were proved in the case $X$ is a smooth projective scheme or, more generally, smooth {\it 
filtrable} scheme (note that in these cases $KH$ is equivalent to $K$). Note, however, that in Theorem \ref{thm:main_2} we get rid of any of these assumptions. 

%
%

\begin{rem}[Remark on geometric classifying spaces] The formulation of previous iterations of the Atiyah-Segal completion theorem (see, for example, \cite{krishna-completion, tabuada2020motivic}), rather uses the geometric classifying space of Morel and Voevodsky \cite[Section 4.2]{mv99}, which imitates the classical Borel construction in algebraic topology. Such a construction was also used by Totaro to define Chow groups of classifying stacks \cite{totaro-motive}. As follows from \cite[Theorem 3.6]{khan-ravi-2}, in the truncating $\mbb A^1$-invariant setup this agrees with right Kan extension from affine schemes.

This model also happens to work in a very specific situation of $BGL_n$ without $\mathbb{A}^1$-invariance thanks 
to the work of Annala-Iwasa (see \cite[Theorem~3.1]{Annala_2022} and \cite[Theorem~5.3]{annala-hoyois-iwasa}). But in general 
such a model relies heavily on $\mathbb{A}^1$-invariance and plainly does not work in more general settings. It is 
therefore crucial, in order to prove Theorem~\ref{thm:main} in the generality of not necessarily smooth stacks, that 
our completion theorems are formulated uniformly both in the case of $\mathbb{A}^1$-invariant theories and Hochschild-type  theories in terms of the \textit{right Kan extension from derived affine schemes}. 

The idea that evaluating on a geometric classifying space is equivalent to just taking the right Kan extension is an idea 
that has been around in the community already for some time. The work of \cite{khan2021generalized}, which we rely on, has used this elegantly to 
provide an extension of the motivic six functor formalism to certain suitable stacks.
\end{rem}

\subsection{Completion theorems for Hochschild homology and related theories}\label{ssec:intro:completion theorem for HH} In an orthogonal direction, it remains to discuss AS completion theorem for the invariants built from Hochschild homology. The starting point here is an observation due to Ben-Zvi--Francis--Nadler \cite{BFN} that for perfect stacks the \textbf{Hochshild homology} $HH(\sX/\mbb Q):=HH(\Perf(\sX)/\mbb Q)$ has a geometric interpretation as the ring of  functions on the \textbf{loop stack} $\sL\sX$, defined as the mapping $\mbb Q$-stack from $S^1$ to $\sX$ (see our Section~\ref{sec:loops} for a short review).
Out of $HH(\sX/\mbb Q)$ together with its natural $S^1$-action, one can then also extract some other invariants: the \textbf{negative cyclic homology} $HC^{-}(\sX/\mbb Q)$ (as homotopy $S^1$-fixed points), as well as \textit{cyclic homology} $HC(\sX/\mbb Q)$  (as homotopy $S^1$-orbits) and \textit{periodic cyclic homology} $HP(\sX/\mbb Q)$ (as the Tate construction).

To prove the completion theorem for Hochschild homology, given $\sX=[X/G]$ we wish to compare the value of right Kan extension $R^\mrm{dAff}HH(\sX/\mbb Q)$ and the Atiyah-Segal completion $HH(\sX/\mbb Q)^\wedge_{I_G}$. It has been already observed in a slightly different form by Nadler--Ben-Zvi \cite{Ben_Zvi_2012}, and later Chen \cite{Chen_2020}, that these two versions of Hochschild homology correspond geometrically to completing $\sL \sX$ along two closed substacks. Namely,
\begin{enumerate}
\item $R^\mrm{dAff}HH(\sX/\mbb Q)$ is given by global functions on the derived completion $\widehat{\sL}\sX$ along the substack of \textbf{constant loops} $\sX \hookrightarrow \sL \sX$. This was observed by Ben-Zvi--Nadler in \cite{Ben_Zvi_2012}; we revisit their results in Appendix~\ref{sec:bzn} to make sure they remain true in the required generality (see Remark \ref{rem:intro:completions}). 
\item The other completion is given by the stack of \textbf{unipotent loops} $$\underline{\sM\mrm{aps}}(B\mathbb{G}_a,\sX) =: \sL^u \sX \rightarrow \sL \sX.$$ As discussed in  \cite{Chen_2020} one can view $\sL^u \sX$ as the completion of $\sL\sX$ at a closed substack that picks up the unipotent elements in the stabilizer group of every point.  The map to $\sL\sX$ is induced by the affinization map $S^1 \rightarrow B\mathbb{G}_a$. As we discuss in Section \ref{ssec:AS-completion of HH and unipotent loops}, one has a natural identification $HH(\sX/\mbb Q)^\wedge_{I_G}\simeq \mc O( \sL^u \sX)$.
\end{enumerate}
This way, the completion theorem for $HH$ reduces to showing that the map ${\sL}^u\sX\ra \widehat{\sL}\sX$ between the two stacky completions is an isomorphism. A pleasant observation is that this is essentially equivalent to stabilizers of all points of $\sX$ being nice or, in other words, the stack $\sX$ should be ANS. This gives the following:

\begin{thm}\label{thm:completion-th-HP-intro}
Fix $G$ a reductive group over a characteristic zero field $k$. Let $\sX \rightarrow BG$ be representable morphism where $\sX$ is derived algebraic stack, such that the morphism $\sX^{\mrm{cl}} \rightarrow BG$ is of finite type. Assume further that $\sX$ is ANS. Then we have canonical equivalences
\begin{align*}
	HH(\sX/\mbb Q)^\wedge_{I_G} &\xrightarrow{\simeq} \lim_{\Spec R \rightarrow \sX} HH(R)\\
	HC^-(\sX/\mbb Q)^\wedge_{I_G} &\xrightarrow{\simeq} \lim_{\Spec R \rightarrow \sX} HC^-(R)\\
HP(\sX/\mbb Q)^\wedge_{I_G} &\xrightarrow{\simeq} \lim_{\Spec R \rightarrow \sX} HP(R)\\
HC(\sX/\mbb Q)^\wedge_{I_G} &\xrightarrow{\simeq}  \lim_{\Spec R \rightarrow \sX} HC(R).
\end{align*}
\end{thm}
In fact, while the case of $HC^-$ formally follows from the case $HH$, the case of $HP$ is a particular case of our completion theorem for $\mbb A^1$-localizing invariants (Theorem \ref{thm:main_2}); both then imply the statement for $HC$.

The work of Chen \cite{Chen_2020} was the first to consider both completions ${\sL}^u\sX$ and $\widehat{\sL}\sX$ in a systematic way. He proved a version of the Atiyah-Segal completion theorem for $HP$ for stacks over an algebraically closed field of characteristic zero (\cite[Theorem B]{Chen_2020}), and Theorem \ref{thm:completion-th-HP-intro} can be considered as a partial generalization of his result. Note that Chen's theorem does not require $\sX$ to be ANS; however, this is due to $\mbb A^1$-invariance of $HP$ which eliminates the difference between the two completions (since, morally, $\mbb A^1$-invariance makes the unipotent cone of any group ``contractible"). In Remark~\ref{rem:no AS in general} we show that the Atiyah-Segal completion theorem for the other Hochschild-type invariants does not hold without the ANS assumption on $\sX$.

Let us also note that the actual definition of unipotent loops that we will consider in the body of the paper is slightly more ad hoc than the one we used above. The relation between the two is explained in Remark~\ref{rem:relation to Chen}.

\begin{rem}[A remark on derived completions and \cite{Ben_Zvi_2012}]\label{rem:intro:completions} One of the main results of Ben-Zvi--Nadler in \cite{Ben_Zvi_2012} is the identification of $\widehat{\sL}\sX$ and the completed shifted cotangent bundle $\widehat{\mbb T}_{\sX}[-1]$; using descent for cotangent complex one can then show that the construction $\sX \mapsto \widehat{\sL}\sX$ is left Kan extended from affine schemes. However, it is never specified in their work what they mean by derived completion, and if one follows their proof the definition is not completely clear, e.g. $\widehat{\mbb T}_{\sX}[-1]$ is understood as a colimit of $n$-th neighbourhoods $\colim_n \mbb T_{\sX}^{\le n}$ of the zero section. In fact, as we discuss in Remark \ref{rem:not a completion in general} this is not quite accurate: namely, this formula will typically give a non-convergent stack (unless $\sX$ is smooth), and thus can't agree with $\widehat{\sL}\sX$ (given the latter is defined via some sort of de Rham stack). Nevertheless, in Proposition \ref{prop:completion at 0 section} we show that the analogous \textit{convergent} colimit does agree with other versions of formal completion (which we review in Section~\ref{sec:reminders-completion}), at least if $H^0(\mbb L_{\sX})$ is locally finitely generated (which is true if $\sX$ is algebraic). 
	
	Another issue is deformation theory, to which the authors appeal in their proof. In a non necessarily finite type situation, to have access to usual deformation theory tools, one needs to use a slightly non-standard definition of de Rham stack (called the ``absolute de Rham stack" in \cite{SAG}) to define the formal completion. There is also a Simpson's formal completion\footnote{The completion $\sY^\wedge_f$ along a map $f\colon \sX \ra \sY$ is then defined as the fiber product $\sY\times_{\sY_\mrm{dR}} \sX_{\mrm{dR}}$.}, with the corresponding de Rham stack given by $\sX_{{\mrm{dR}}}(R):= \sX(\pi_0(R)^{\mrm{red}})$, which is more convenient to use in practice. Note that in our situation, if $\sX$ an algebraic stack that is finite type over a field $k/\mbb Q$, we are interested in $HH(-/\mbb Q)$ and thus are forced to consider $\sX$ as an algebraic stack over $\mbb Q$, where it typically will not be of finite type. In Section~\ref{sec:reminders-completion} we systematically analyze the difference between these two versions of derived completion, showing that they are the same if the map of underlying classical stacks is locally finitely presented. This then applies in particular to completion at constant loops $\sX \ra \sL\sX$, when $\sX$ is an algebraic stack. More generally, we fully revisit the proof of Ben-Zvi--Nadler, making sure all steps work in the required generality (this is in Appendix \ref{sec:bzn}).
	
We would like to warn the reader that these various forms of formal completion (at constant loops and/or zero section) are only compatible if we take algebraic stacks in the sense that they have a \textit{smooth} representable cover by the union of derived affine schemes. For algebraic stacks in the \textit{fpqc} sense, $R^{\mrm{dAff}}HH(\sX/\mbb Q)$ will be given by functions on yet another ``neighborhood completion", details on which we might add in a sequel.

In the upcoming \cite{brav-rozenblyum}, Brav and Rozenblyum also give another perspective on the Ben-Zvi--Nadler theorem. Their argument is slightly different and generalizes the result in a different direction, namely for laft prestacks admitting a deformation theory. In other words, their proof significantly relaxes the geometricity condition on the stack, while remaining in a finite type situation. 
\end{rem}

\subsection{Relation with previous and contemporaneous works} We briefly survey the history of Atiyah-Segal completions in algebraic geometry as well as what is new in this paper in both generality and philosophy; we apologize to the many authors in this very rich field in advance for sins of omission. The subject began with Thomason's construction of equivariant $K$ and $G$-theories \cite{thomason-group-actions} where an equivariant version of Quillen's basic results on $K$-theory were proved. Perhaps inspired by Atiyah--Segal's work \cite{atiyah-segal}, Thomason proved a completion theorem for $G$-theory and $K$-theory in the smooth setting \cite{thomason-topological} after $K(1)$-localization. The usual d\'evissage of reducing from a general reductive group scheme, to a torus and eventually to a diagonalizable group with trivial action was already explained in this work; we remark that this approach heavily relies on $\mathbb{A}^1$-invariant properties to pass from a Borel subgroup to a torus. The relationship between equivariant $G$-theory and algebraic cycles on stacks, a shadow of the story about motivic filtrations, was first explored by Edidin--Graham \cite{edidin-graham, edidin-graham-rr}. The viewpoint there is very natural: up to completion and rationalization the Chow groups of algebraic stacks are discovered as the graded pieces of a split filtration on $G_0$. 

With the advent of motivic homotopy theory and noncommutative motives, the completion theorems take on a richer flavor. In particular, Morel--Voevodsky's construction of their geometric classifying space provides a target for the completed $K$-theory and other invariants. We mention the work of \cite{carlsson-joshua} which deals with $G$-theory and \cite{orbifolds} which proves similarly flavored theorems in the case of Deligne--Mumford stacks/orbifolds. Three papers which are very influential for us are Tabuada--van den Bergh's definitive account of the completion theorem in the torus case for $\mathbb{A}^1$-invariant localizing invariants (which we use) \cite{tabuada2020motivic}, Krishna's $K$-theoretic completion theorem in the smooth, projective setting for $K$-theory \cite{krishna-completion} and Chen's work on equivariant hochschild homology \cite{Chen_2020}. These papers work with cohomological invariants such $K$-theory and Hochschild homology, as opposed to earlier accounts which deals with $G$-theory, and Borel-Moore type invariants. The complexion of the problem changes entirely --- there is a lack of $\mathbb{A}^1$-invariance (if the input stack is not smooth) and there is a lack of localization-type results, both are standard approaches to the completion problem even in topology.

The novelty of our result is that it is $K$-theoretic (as opposed to $G$-theoretic) and is, as far as we know, the first to deal with stacks which are not smooth, reduced or even classical. This accounts for some of limitations (such as ANS) of our results. The new idea that we use is trace methods, combining an $\mathbb{A}^1$-invariant completion theorem with one for negative cyclic homology, and relying on the first and last authors' extension of the Goodwillie theorem in \cite{dgm-elden-vova} to the world of stacks. Like the work of Edidin-Graham we are also motivated by motives --- one can view our results as an integral and higher refinement of their results where Chow groups are replaced by the $(2n,n)$-line of motivic cohomology in the sense of \cite{e-morrow}. 

Lastly, around the same time as this work, Krishna and Nath \cite{krishna-nath} also independently explored the completion theorem. We remark on the overlap and distinctions between their work and ours:

\begin{enumerate}
\item Given a non-smooth algebraic stack $\sX$, \cite[Theorem 1.1]{krishna-nath} proves completion results for the $G$-theory of $\sX$; the generality here asks that $\sX$ is a quotient stack by a reductive group. On the other hand, our Theorem~\ref{thm:main_intro} is $K$-theoretic and only works with added hypotheses on the stabilizers of $\sX$. Indeed, the main foci of the two papers somewhat differs --- theirs follow, and in some sense complete, the tradition of $G$-theoretic completion theorems, while our aim is to explore the less-chartered $K$-theoretic territory. 
\item Our Theorem~\ref{thm:main_2} subsumes \cite[Theorem 1.6, Corollary 1.7]{krishna-nath} in characteristic zero and our Theorem~\ref{thm:completion-th-HP-intro} subsumes \cite[Corollary 1.11, Theorem 1.14]{krishna-nath} in that we do not need conditions on the stabilizers, the smoothness of $X$ or algebraic-closedness of the base field. However, we mention the elegant \cite[Lemma 13.6]{krishna-nath} which controls the completion of the cyclic homology term which should be useful for other future applications.
\item A very interesting aspect of \cite{krishna-nath} is that they also considered completions at other ideals, we do not address this at all in this paper. 
\end{enumerate}

\subsection{Outline} We begin the paper with Section~\ref{sec:examples}. The point of this section is to illustrate via examples one of the main perspectives that we take in this paper: that completion at the Atiyah-Segal ideal restores certain descent properties that one should expect for various invariants in the setting of algebraic stacks. In the preliminaries Section~\ref{sec:prelim}, we review the derived algebraic geometry and the theory of localizing invariants that we will need in this paper. We work in the generality of animated rings over $\mathbb{Z}$, something more than necessary for the main results of this paper (which is over $\mathbb{Q}$) in view of the completion theorems in positive and mixed characteristics that will constitute the sequel to this paper. In Section~\ref{sec:derived} we discuss descent properties for cohomology theories of derived stacks. 
In particular, we introduce $R^{\mrm{dAff}}E$, the right-hand-side of the Atiyah-Segal completion theorem and establish its basic properties.

In Section~\ref{sec:setup} we define the Atiyah-Segal ideal and formulate the completion theorem. A useful result we will 
turn to repeatedly is the ``change-of-groups'' Lemma~\ref{lem:comparing Atiyah-Segal completions}. This ensures that even 
though we are working with stacks which are ANS, and hence its stabilizers are of a fairly restrictive class, we are allowed to complete with respect to $I_G$ where $G$ is a reductive group under the stack. Our first main theorem is established in Section~\ref{sec:askh} where we improve on the Tabuada-Van den Bergh completion theorem using cdh excision. Section~\ref{sec:as-hh} proves the completion theorem for Hochschild invariants --- in particular for $HH, HC^-, HC$ and $HP$, which is our second main theorem. Putting all of this together in Section~\ref{sec:ask}, we prove the Atiyah-Segal completion theorem for ANS stacks in characteristic zero. In Section~\ref{sec:apps}, we discuss applications of our main theorem to building coherent transfers (aka spectral Mackey functors) on completed $K$-theory as well as propose a candidate for motivic cohomology of certain stacks. 

The present paper constitutes three appendices. In Appendix~\ref{sec:bzn} we revisit the Ben-Zvi--Nadler's geometric version of the HKR theorem for geometric stacks. 
It also contains certain constructions in derived algebraic geometry that we could not locate in suitable generality in the literature, including different versions of completions vector bundles and some cotangent complex related descent results.  In Appendix~\ref{sec:deshmukh} we extend a result of Deshmukh, which builds on work of Moret-Bailly, to derived stacks. It is a strong Nisnevich-local structural result for covering of stacks that we will need to control $R^{\mrm{dAff}}E$. Lastly, in Appendix~\ref{app:mot-algpsc} we record a (relatively straightforward) extension of the motivic filtration on schemes to derived algebraic spaces.

\subsection{Vista} This is the beginning of a series of papers on motivic filtrations on algebraic stacks, centered around completion theorems in $K$-theory and its cousins. As already laid out in this paper, our philosophy is that Atiyah-Segal completion theorems provide intrinsic descriptions of completions of genuine invariants along an induced motivic filtration. In the sequel, we plan to establish completion theorems for topological Hochschild and cyclic homology. While we build on the results of Ben-Zvi--Nadler in characteristic zero, the strategy in mixed and positive characteristics is new and more involved. 

Towards this end we are working on a project, joint with Annala, in using the machinery of (derived) Theta-stratifications \cite{DHL-theta}, coupled with the second author's work on $p$-adic Hodge theory \cite{kubrak-prikhodko} to ``induct along the schematic locus'' of a derived algebraic stack. We hope to use this strategy to establish a completion theorem for a nice enough collection of algebraic stacks in mixed and positive characteristics. 

\subsection{Acknowledgements} We thank Amalendu Krishna for discussions about his work on the completion theorem, Jacob Lurie for explaining to us his version of derived completion and numerous other useful conversations related to this paper, Charanya Ravi for helpful conversations about the independence of the AS completion on the group and the affine bundle invariance, Nick Rozenblyum for clarifying various aspects of the HKR filtration for stacks and explaining to us his own lucid perspsective on it, David Rydh for his usual steadfast guidance on all stack-theoretic issues, and Matthew Morrow for conversations related to motivic filtrations beyond schemes.

EE was supported by an Erik Ellentuck fellowship during his time at the Institute for Advanced Study (IAS). Somewhat ironically, some of this work was done during the special semester on ``$p$-adic geometry''; we apologize to (but also thank) the organizers, Jacob Lurie and Bhargav Bhatt that the work produced is in characteristic zero (for now). We thank Vita the cat for constant distractions and entertainment while EE and DK were writing this paper at the IAS. DK is grateful to IAS and IH\'ES for the great working conditions and a fruitful academic environment during the time this project was carried out. VS acknowledges support from the grant SFB 1085: "Higher Invariants."

\section{Examples and phenomenology}\label{sec:examples} We begin with an informal explanation of some examples of the phenomena that this paper aims to capture. To simplify matters all schemes and stacks in this section are classical. Let $E: \Cat^{\perf}_{\infty} \rightarrow \Spt$ be a localizing invariant; we review this concept in more detail in Section~\ref{sec:localizing}. The value of $E$ on a qcqs scheme $X$ is defined to be $E(X) := E(\Perf(X))$ and the functor $X \mapsto E(X)$ enjoys some reasonable descent properties. At the very least, it satisfies Nisnevich descent, as already is implicit in the work of Thomason-Trobaugh \cite{TT} (see, for example, \cite[Proposition A.15]{CMNN}). Often, one can actually do better: the functors constructed out of trace considerations like $THH, HH(-/k)$, $TC$ and their cousins in fact satisfy flat descent \cite[Corollary 3.4]{BMS2}, and so do telescopically-localized invariants \cite[Theorem 5.1]{CMNN}. 

We point out two further features of the functor $X \mapsto E(X)$. First, by the definition of a localizing invariant, $E(X)$ is amenable to calculations from derived-categorical techniques. The point is that any localizing invariant splits semiorthogonal decomposition; for example Beilinson's decomposition of $\Perf(\bbP_X^n)$ \cite{beilinson-decomp}
\[
\Perf(\bbP_X^n) \simeq \langle \sO, \sO(1), \cdots, \sO(n) \rangle,
\]
leads to the projective bundle formula: $E(\bbP^n_X) \simeq \bigoplus_{i=0}^n E(X)$. Second, the functoriality of the assignment $X \mapsto \Perf(X)$ endows $E(X)$ with additional structure. For example, if $f: X \rightarrow Y$ is a morphism then the pullback functor $f^*: \Perf(Y) \rightarrow \Perf(X)$ admits a right adjoint $f_*$ whenever $f$ is proper and has finite tor-amplitude \cite[Tag 0B6G]{stacks}. This endows the functor $X \mapsto E(X)$ with the structure of transfers (more precisely, it organizes into a spectral Mackey functor as in \cite{BarwickMackey}), which are often useful for computations and are central to proving descent for rationalized/telescopically-localized localizing invariants, as was done in \cite{CMNN}. Under these considerations it is reasonable, for $\sX$ an algebraic stack, to set $E(\sX):= E(\Perf(\sX))$. Just as in the case of schemes, such a functor splits semiorthogonal decomposition of stacks and admits transfers along proper morphisms of stacks that have finite tor-amplitude.

However, this definition has a major drawback, related to \emph{descent}. We illustrate this in the next two examples. For simplicity, we work with localizing invariants linear over a field $k$ of characteristic zero and consider the case of $E = HH(-/k)$, Hochschild homology. To formulate descent, we will also use the language of $E$(-universal) descent which we review in Definition~\ref{def:eff-desc}. 

\begin{exam}[\'Etale coverings of Deligne-Mumford stacks are not of $HH(-/k)$-descent.]\label{exam:mun}  The morphism $\Spec k \rightarrow B\mu_n$ is an \'etale covering (recall the characteristic zero hypothesis). If $\Spec k \rightarrow B\mu_n$ were of $HH$-descent, then we would have equivalences
	\begin{equation}\label{eq:hh-mun}
	HH(B\mu_n/k) \simeq \lim_{[m] \in \Delta} HH(\mu_n^{\times m}/k) \simeq \lim_{[n] \in \Delta} \mc O(\mu_n^{\times m}) \simeq \mc O(B\mu_n).
	\end{equation}
	Via Cartier duality, the category $\Perf(B\mu_n)$ is canonically equivalent to $\bbZ/n$-graded objects in $\Perf(k)$, in particular
	\[
	\Perf(B\mu_n) \simeq \Perf(k)^{\oplus \bbZ/n}
	\]
	as small stable $\infty$-categories. Therefore, we get that $HH(B\mu_n/k) \simeq k[\mbb Z/n]$. On the other hand, $\mc O(B\mu_n)\simeq k \neq k[\mbb Z/n]$.
\end{exam}

\begin{exam}[Nisnevich surjections of Artin stacks are not of $HH(-/k)$-descent.]\label{exam:bgm} Since any line bundle is Zariski-locally trivial, the standard covering map $\Spec k \rightarrow B\bbG_m$ is a Zariski, and, in particular, Nisnevich surjection. If the morphism $\Spec k \rightarrow B\bbG_m$ were of $HH$-descent, then we would have an equivalence:
	\begin{equation}\label{eq:hh-gm}
HH(B\bbG_m/k) \simeq \lim_{[n] \in \Delta} HH(\bbG_m^{\times n}/k).
\end{equation}
The HKR isomorphism furnishes an equivalence $HH(\bbG_m^{\times n}/k) \simeq R\Gamma(\bbG_m^{\times n}; \prod_i \Sym^i(\bbL_{\bbG_m^{\times n}}[1]))$. By descent for cotangent complex (see Proposition~\ref{prop:etale-descen-cot}) we then conclude that
\[
\lim_{[n] \in \Delta} HH(\bbG_m^{\times n}/k) \simeq \prod^{\infty}_{i=0} R\Gamma(B\bbG_m,   \prod_i \Sym^i\left(\bbL_{B\bbG_m}[1])   \right).
\]
Now, the cotangent complex of a classifying stack $BG$ is identified as $\bbL_{BG} \simeq \mathfrak{g}^{\vee}[-1]$ (see, for example, \cite[Example A.3.8]{kubrak-prikhodko}) where $\mathfrak{g}$ is the Lie algebra of $G$. Therefore, the right-hand-side simplifies to $\prod_i R\Gamma(B\bbG_m,   \prod_i \Sym^i(  \mc O_{B\mbb G_m}))\simeq \prod_ik$. On the other hand, since $HH(-/k)$ is a \emph{finitary} localizing invariant and $\Perf(B\mbb G_m)\simeq \oplus_{i\in \mbb Z} \Perf(k)$, we have that $HH(B\bbG_m/k) \simeq \bigoplus_i k$. We see that they cannot be the same.

\end{exam}

Examples~\ref{exam:mun} and~\ref{exam:bgm} suggest an alternative extension of $HH(-/k)$ to algebraic stacks, one which 
enjoys better descent properties and ensures that~\eqref{eq:hh-mun} and~\eqref{eq:hh-gm} holds. Consider $E$ as a functor on 
affine schemes and for an algebraic stack $\sX$, we set
\[
R^{\mrm{Aff}}E(\sX) := \lim_{\Spec R \rightarrow \sX} E(\Spec R).
\]
This procedure is nothing but a right Kan extension and will be discussed more extensively in Section~\ref{subsec:rke}. As 
explained Lemma~\ref{eq:colimofrep}, $R^{\mrm{Aff}}E$ inherits descent properties that $E$ has in the following sense: if $E$ 
is an $\tau$-sheaf on affine schemes, then any $\tau$-surjection of stacks is of universal $R^{\mrm{Aff}}E$-descent. For 
example if $E = HH(-/k)$, then
\[
R^{\mrm{Aff}}HH(B\mu_n/k) \simeq \mc O(B\mu_n) \simeq k \qquad R^{\mrm{Aff}}HH(B\bbG_m/k) \simeq \prod_i R\Gamma(B\bbG_m,   \prod_i \Sym^i(  \mc O_{B\mbb G_m}))\simeq \prod_ik.
\]
Additionally, as already explained in the introduction, $R^{\mrm{Aff}}E$ inherits filtrations that $E$ has on affine schemes. Of particular interest, $R^{\mrm{Aff}}K$ of a stack comes readily equipped with motivic filtrations.

The drawback of this extension is that it is not clear at all that $R^{\mrm{Aff}}E$ splits semiorthogonal decompositions and admits coherent transfers. The Atiyah-Segal completion theorem can be read as the rather surprising assertion that, in fact, 
$R^{\mrm{Aff}}E$ does. Theorem~\ref{thm:main_intro} presents an alternative formula $R^{\mrm{Aff}}E(\sX)$, in certain cases, 
which does not require the data of all affine schemes mapping into $\sX$. This alternate formula instead asks for the data of 
a $K(BG)$-module structure on $E(\Perf(\sX))$. 
We illustrate this theorem in light of Examples~\ref{exam:mun} and~\ref{exam:bgm}. 

\begin{exam}[The Atiyah-Segal completion for $HH(-/k)$] We compute $R^{\mrm{Aff}}HH(-/k)$ of $B\bbG_m$ and $B\mu_n$ via the 
Atiyah-Segal completion theorem. Both computations could be performed easily via descent and, in fact, already implicitly 
carried out in  Examples~\ref{exam:mun} and~\ref{exam:bgm}. We hope that their computation via the completion theorem would 
be illustrative for the reader.

Using the lax-monoidal structure on $HH$, the above computations of $HH(B\bbG_m/k)$ and $HH(B\mu_n/k)$ enhance to equivalences of derived commutative rings:
\[
HH(B\bbG_m/k) \simeq k[\mbb Z]\simeq k[t,t^{-1}] \qquad HH(B\mu_n/k) \simeq k[\mbb Z/n]\simeq k[t]/(t^n-1)
\]
(both rings  being concentrated in cohomological degree 0). The Atiyah-Segal ideal in $HH$ in both cases is generated by the element $t-1$; we refer to Example~\ref{exam:as-ideal} and Proposition~\ref{prop:ig-uni} for details. 

The $(t-1)$-adic completion of $k[t]/(t^n-1)$ picks out the component $k[t]/(t-1) \cong k$ via the isomorphism $k[t]/(t^n-1) \cong k[t]/(t-1) \times k[t]/(t^{n-1} + t^{n-2} + \cdots + 1)$, and therefore, by the Atiyah-Segal completion theorem, we have that:
\[
R^{\mrm{Aff}}HH(B\mu_m) \simeq (k[t]/(t^n-1))^\wedge_{(t-1)}\simeq k[t]/(t-1) = k.
\]

For $B\mbb G_m$,  the $(t-1)$-adic completion $k[t,t^{-1}]^\wedge_{(t-1)}$ is simply given by power series in $t-1$, and thus we get
\[
R^{\mrm{Aff}}HH(B\bbG_m)\simeq k[t,t^{-1}]^\wedge_{(t-1)} \simeq k[[t-1]].
\]



\end{exam}

\begin{exam}[Rational $K$-theory of stacks] Another interesting invariant which enjoys \'etale descent for schemes is rationalized algebraic $K$-theory. As already pointed out in \cite[Remark A.16]{CMNN}, it does note enjoy descent along the simplest \'etale covering of stacks $\Spec \mathbb{C} \rightarrow BC_2$. Our main results also apply to establish \'etale descent for Atiyah-Segal completed rational algebraic $K$-theory. 
\end{exam}

\begin{exam}[Right Kan extended $K$-theory] In the $HH$ situations, the right Kan extended theories are easily computed because $HH$ enjoys \'etale descent on schemes.  For $K$-theory, which infamously does not, the right Kan extended theories more easily computed via the Atiyah-Segal completion and semiorthogonal decompositions of categories. For example, from our main theorem we easily read off:
\[
R^{\mrm{Aff}}K(B\bbG_m) \simeq (K(k)[t,t^{-1}])^{\wedge}_{(t-1)}\simeq K(k)[[t-1]].
\]
\end{exam}

%

\section{Preliminaries} \label{sec:prelim}

\subsection{Generalities on prestacks and stacks}In this section we make a review of notions of derived algebraic geometry that we will use. We also set up notations and conventions for the paper which, to warn the reader, might differ slightly from the standard ones.

While most of the results of this paper concern derived algebraic stacks over characteristic zero fields, we still aim to set up notations and foundations for future use (particularly the sequel where will investigate the Atiyah-Segal completion theorems in positive and mixed characteristics). The paper also contains some general results that don't require characteristic 0 assumption (like the motivic filtration on K-theory of derived algebraic spaces in Appendix \ref{app:mot-algpsc}). This is why below we will set up everything in the context of $\mbb Z$-linear derived algebraic geometry. 

 We sometimes refer to \cite{GRderalg1, GRderalg2} which assumes that the base is a field of characteristic zero throughout, however, for the results that we cite it is never important.

\subsubsection{Derived prestacks} Let $\CAlg^{\mrm{cl}}$ denote the 1-category of classical commutative rings.
Let $\CAlg^{\mrm{an}}:=\mrm{Ani}(\CAlg^{\mrm{cl}})$ denote the $\infty$-category of animated commutative rings, given by ``animation" of $\CAlg^{\mrm{cl}}$. Recall that by definition this is the $\infty$-category of product preserving functors from polynomial $\mbb Z$-algebras, $\PSh_{\Sigma}(\mrm{Poly}_\mbb Z) \simeq \CAlg^{\mrm{an}}$. Given an animated ring $R$ we denote by $\CAlg_{R}^{\mrm{an}}$ the under category $(\CAlg^{\mrm{an}})_{R/}$. In case $R$ is classical, we have $\CAlg_{R}^{\mrm{an}}:=\mrm{Ani}(\CAlg^{\mrm{cl}}_{R})$. If $R$ is a $\mbb Q$-algebra, then $\CAlg_{R}^{\mrm{an}}$ is naturally identified with the category $\CAlg(\D(R)_{\ge 0})$ of $\mbb E_\infty$-algebras in connective $R$-modules.

\begin{defn}We say that $A\in \CAlg^{\mrm{an}}$ is
	\begin{enumerate}
		\item \textbf{$n$-truncated} if $\pi_i(A)=0$ for $i>n$;
		\item  \textbf{truncated} if it is $n$-truncated for some $n\ge 0$.
	\end{enumerate}	
	We denote by  $$ \CAlg^{\mrm{an},[0,\infty)}\subset \CAlg^{\mrm{an}}$$ the full subcategory spanned by truncated animated commutative rings. 
\end{defn}

\begin{defn} Let $R\in \CAlg^{\mrm{an}}$. Recall the following notions:
	\begin{enumerate}
		\item A \textbf{derived prestack} over $R$ is a functor $\mc X\colon \CAlg^{\mrm{an}}_R\ra \Spc$; we denote by $\mrm{dPStk}_R$ the $\infty$-category of derived prestacks over $R$. 
		\item A prestack $\sX$ is called \textbf{convergent} if 
		$$
		\sX(A)\xrightarrow{\sim}\lim_n \sX(\tau_{\le n} A)
		$$
		for any $A\in \CAlg^{\mrm{an}}_R$.
	\end{enumerate}	
	We will denote the full subcategory spanned by convergent derived prestacks by
	$$
	\mrm{dPStk}^{\mrm{conv}}_R\subset \mrm{dPStk}_R.
	$$
\end{defn}	

\begin{exam}
		We denote by $\Spec A$ the prestack represented by an animated algebra $A\in \CAlg^{\mrm{an}}_R$. We call such prestacks \textbf{derived affine schemes} and denote by $\mrm{dAff}_R\subset \mrm{dPStk}_R$ the corresponding full subcategory. By Yoneda's lemma $\mrm{dAff}_R\simeq (\CAlg^{\mrm{an}}_R)^{\mrm{op}}$.
\end{exam}	

\begin{exam}\label{ex:constant prestack}
	Let $S\in \mrm{Spc}$ be a space. A functor $\underline{S}$ sending any $A\in \CAlg^{\mrm{an}}_R$ to $S\in \mrm{Spc}$ is the \textbf{constant prestack} associated to $S$. Functor $S\mapsto \underline{S}$ commutes with all limits and colimits. In particular, we have $\underline{*}\simeq \Spec R$ is the final object and 
	$$
	\underline{S}\simeq \colim_S (\Spec R),
	$$ 
	where the diagram in the colimit is constant with value $\Spec R$.
\end{exam}	

\begin{rem}\label{rem:convergent prestacks}
	Restricting to $\CAlg^{\mrm{an},[0,\infty)}\subset \CAlg^{\mrm{an}}$ induces an equivalence
	$$
	\mrm{dPStk}^{\mrm{conv}}_R\simeq \Fun(\CAlg^{\mrm{an},[0,\infty)}_R,\Spc)
	$$
	with the inverse functor given by right Kan extension. The embedding 
	$$
	\mrm{dPStk}^{\mrm{conv}}_R\ra \mrm{dPStk}_R
	$$
	commutes with limits, but typically not colimits. E.g. given an animated algebra $A\in \CAlg^{\mrm{an}}_R$ which is not necessarily truncated we have
	$$
	\Spec A \simeq \colim_n \Spec \tau_{\le n}A
	$$
	in $\mrm{dPStk}^{\mrm{conv}}_R$, but typically not in $\mrm{dPStk}_R$ since the identity map $A\ra A$ doesn't need to factor through any of the truncations $\tau_{\le n}A$.
\end{rem}

\begin{exam}
	Any affine derived scheme $\mc X:=\Spec A$ or a constant prestack $\underline{S}$ is convergent.
\end{exam}	

Sometimes, we will also need to consider classical prestacks. 
\begin{defn}
	Let $R$ be a classical ring. A \textbf{classical prestack} over $R$ is a functor $\mc X\colon \CAlg^{\mrm{cl}}_R\ra \Spc$. We denote by $\mrm{PStk}_R$ the category of classical prestacks over $R$.
\end{defn}	

Let $R\in \CAlg_{R}^{\mrm{an}}$. The fully faithful embedding $\CAlg^{\mrm{cl}}_{\pi_0(R)}\subset  \CAlg^{\mrm{an}}_R $ defines a restriction functor 
$$
\mrm{dPStk}_R\ra \mrm{PStk}_{\pi_0(R)}, \quad \sX \mapsto \sX^{\mrm{cl}}.
$$
We will call $\sX^{\mrm{cl}}$ the underlying classical stack of $\sX$. We also have the left Kan extension along the above inclusion, which gives a functor 
$$
L^\mrm{cl}\colon \mrm{PStk}_{\pi_0(R)} \ra \mrm{dPStk}_R
$$
in the other direction. 
\begin{defn}
	The \textbf{classical truncation} $\tau^{\mrm{cl}}(\sX)$ of $\sX$ is defined as 
	$$
	\tau^{\mrm{cl}}(\sX):=L^{\mrm{cl}}(\sX^{\mrm{cl}}).
	$$
	By adjunction, we have a natural map $\tau^{\mrm{cl}}(\sX) \ra \sX$.
\end{defn}

\begin{exam}\label{exam:cl} If $\sX = \Spec A$ is affine, then $\tau^{\mrm{cl}}(\sX) =\Spec \pi_0(A)$. The map $\tau^{\mrm{cl}}(\sX) \ra \sX$ is induced by the natural map $A \rightarrow \pi_0(A)$. 
\end{exam}

\subsubsection{Quasi-coherent sheaves}  Given a derived prestack $\sX\in \mrm{dPStk}_R$ one defines the $\infty$-category $\mrm{D}_{\mrm{qc}}(\sX)$ of \textbf{quasi-coherent sheaves on $\sX$} via right Kan extension from derived affine schemes; explicitly this is given  as
$$\mrm{D}_{\mrm{qc}}(\sX):= \lim_{\Spec A \rightarrow \sX} \mrm{D}(A)$$
where $\mrm{D}(A)$ is the unbounded derived $\infty$-category of $A$-modules. This is a stable, presentable $\infty$-category equipped with a symmetric monoidal structure coming from the limit of symmetric monoidal structures on $\mrm{D}(A)$.  By construction, if $\sX=\Spec A$ one has $\mrm{D}_{\mrm{qc}}(\sX)\simeq \mrm{D}(A)$. A map of prestacks $f\colon \sX\ra \sY$ induces the \textbf{pull-back functor} 
$$
f^*\colon \mrm{D}_{\mrm{qc}}(\sY) \ra \mrm{D}_{\mrm{qc}}(\sX).
$$

The functor $f^*$ commutes with colimits and thus has a right adjoint which we denote by 
$$f_*\colon \mrm{D}_{\mrm{qc}}(\sX) \ra \mrm{D}_{\mrm{qc}}(\sY).$$
In the case of the final object $\sY=\Spec R$ in $\mrm{dPStk}_R$ and the projection $p\colon \sX \ra \Spec R$ the pushforward $p_*$ is identified with the \textbf{global sections functor} $$\Gamma(\sX, -)\colon \mrm{D}_{\mrm{qc}}(\sX) \ra \mrm{D}(R).$$
\begin{exam}
	The association $(\Spec A \ra \sX) \mapsto A\in \mrm{D}(A)$ defines an object $\mc O_{\sX}\in\mrm{D}_{\mrm{qc}}(\sX)$ which is called the \textbf{structure sheaf}. We will often denote $\Gamma(\sX, \sO_{\sX})$ as $\sO(\sX)$. It can be explicitly computed as 
	$$
	\Gamma(\sX, \sO_{\sX})\simeq \lim_{\Spec A\ra \sX} A.
	$$
\end{exam}	

An object $\sF\in \mrm{D}_{\mrm{qc}}(\sX)$ is called a \textbf{perfect complex} if it is dualizable in the sense of the 
symmetric monoidal structure. We denote by $\Perf(\sX)\subset \mrm{D}_{\mrm{qc}}(\sX)$ the full subcategory spanned by perfect sheaves. We also have 
$$
\Perf(\sX)\simeq \lim_{\Spec A \rightarrow \sX} \Perf(A);
$$
see, for example, the discussion at the beginning of \cite[Appendix A]{Halpern_Leistner_2023}. If $\sX\colon \mrm{dAff}^{\mrm{op}}_R \ra \Spc$ is an accessible functor, then $\Perf(\sX)$ is a small $\infty$-category (see e.g. \cite[Proposition A.2]{Hesselholt-Pstragowski}). In practice,  we will only consider the category $\Perf(\sX)$ when $\sX$ is accessible.
\exam{Since $A\in \Perf(A)$ for any $A$ we get that the structure sheaf $\mc O_{\sX}\in\mrm{D}_{\mrm{qc}}(\sX)$ is a perfect sheaf on $\sX$.}
\subsubsection{Stacks} 
By a \textbf{derived stack} over $R$ we mean a prestack which is a sheaf in the \'etale topology\footnote{We remark that this is different from the convention in some sources where a stack is supposed to be a sheaf in the flat topology; see, for example, \cite[Tag 026N]{stacks}.} on $\CAlg^{\mrm{an}}_R$ (as in \cite[Definition 2.2.2.12]{toen-vezzosi} or similar to \cite[Chapter 2]{GRderalg1}). We denote by $\mrm{dStk}_R\subset \mrm{dPStk}_R$ the full subcategory spanned by derived stacks. The fully faithful embedding $\mrm{dStk}_R\ra \mrm{dPStk}_R$ has a left adjoint
$$L_{\text{\'et}}\colon \mrm{dPStk}_R \ra \mrm{dStk}_R$$
given by \'etale sheafification. The functors $\mrm{D}_{\mrm{qc}}$, $\Perf$, $\sO$ are right Kan extensions of functors on derived affine schemes that are sheaves in \'etale topology. As such, they factor through $L_{\text{\'et}}$: namely, given a prestack $\sX$, the map
\[
\sF(L_{\et}\sX) \xrightarrow{} \sF(\sX)
\]
(induced by canonical map $\sX \rightarrow L_{\et}\sX$) is an equivalence for $\sF$ being any of the functors above.

We will reserve the word \textbf{surjection} for an \'etale-effective epimorphism (=\'etale surjection). Namely, a map of derived stacks $f\colon \sX \ra \sY$ will be called a surjection if the induced map $\pi_0(\sX) \ra \pi_0(\sY)$ is a surjection of \'etale sheaves of sets. Here by $\pi_0(\sX)\colon \CAlg^{\mrm{an}}_R\ra \mrm{Set}$ we mean the \'etale sheafification of the functor which sends $R$ to $\pi_0(\sX(R))$.
\subsubsection{Algebraic and geometric stacks} We now carefully state our conventions on algebraicity of stacks; we follow quite closely \cite[Chap. 2, Sec. 4]{GRderalg1}, except we will also make systematic use of derived algebraic spaces which 
are omitted in the reference. For a Grothendieck topology $\tau$ on $\mrm{dAff}_R$, we will denote by $\mrm{dStk}_{R}^\tau\subset  \mrm{dPStk}_R$ the full subcategory spanned by $\tau$-sheaves. We will also make use of the 
sheafification functor 
\[
L_{\tau}\colon  \mrm{dPStk}_R \rightarrow \mrm{dStk}^{\tau}_R,
\]
whenever the latter exists, in particular for the Nisnevich and \'etale topologies which feature prominently in this paper.

If $\mbf P$ is a property of a map of (derived) affine schemes, then we say that a morphism of (derived) prestacks $\sX \rightarrow \sY$ is \textbf{affine and is $\mbf P$} if for any map $\Spec A \rightarrow \sY$, the pullback $\Spec A \times_{\sY} \sX \ra \Spec A$ is affine and is $\mbf P$. Here $\mbf P$ could be empty, in which case we get a definition of affine map of derived prestacks.

The derived stacks we consider this paper typically are examples of $n$-Artin stacks (defined as in \cite[Chap. 2, Sec. 4]{GRderalg1} or \cite[Sec 1.3.3]{toen-vezzosi} for the original reference). We will not recall the general definition of $n$-Artin stack, since all the important examples in the paper are at most 2-Artin.  Instead, below we make a short glossary for different types of ``algebraic" stacks that the reader will meet in the paper. We warn that our terminological conventions might diverge slightly from those in some other sources.

\begin{defn} \label{def:alg-geom} We fix the following terminology throughout the paper:
	\begin{enumerate}
		\item A stack $\mc X$ over $R$ is called \textbf{derived 0-Artin} if it is isomorphic to $L_{\text{\'et}}(\sqcup_{i\in I}U_i)$ where $U_i$ are derived affine $R$-schemes. A map of prestacks $\sX\ra \sY$ is called \textbf{0-representable} if for any map $\Spec A\ra \sY$ the base change $\Spec A \times_{\sY} \sX \ra \Spec A$ is a derived 0-Artin stack.
			\item A stack $\mc X$ is called a \textbf{derived $R$-scheme} if it is obtained as a quotient 
			$$
			X_1 \rightrightarrows X_0
			$$
			by a correspondence where $X_i$ are derived 0-Artin and the maps are Zariski open embeddings. We denote by $\mrm{dSch}_R$ the $\infty$-category of derived $R$-schemes.
			\item A stack $\mc X$ over $R$ is called a \textbf{derived algebraic space} if the diagonal $\Delta_{\sX}\colon \sX \ra \sX \times \sX$ is affine, there is an \'etale 0-representable cover $U\twoheadrightarrow \sX$ where $U$ is 0-Artin, and $\mc X^{\mrm{cl}}$ is a classical algebraic space\footnote{If $\sX$ is classical, this notion will be stronger than being a quasi-separated algebraic space in the sense of \cite[Tag 025X, Tag 02X5]{stacks} where the diagonal is required to  be representable in quasi-compact schemes.}  over $\pi_0(R)$. We denote by $\mrm{dAlgSpc}_R$ the $\infty$-category of derived algebraic spaces over $R$.


\item We will say that a morphism of derived prestacks $f\colon \sX \ra \sY$ is \textbf{representable} if it is represantable in derived algebraic spaces: namely, for any point $\Spec A \ra \sY$ the fiber $\Spec A \times_{\sY} \sX \ra \Spec A$ is a derived algebraic space over $A$. If $\mbf P$ is a property of a morphism of derived algebraic spaces, then we say that a map of prestacks $f\colon \sX \rightarrow \sY$ \textbf{is $\mbf P$} if it is representable and $\Spec A \times_{\sY} \sX \ra \Spec A$ is $\mbf P$ for any $\Spec A \ra \sY$. 
	\item A derived stack $\mc X$ over $R$ is called \textbf{algebraic} if the diagonal $\Delta_{\sX}\colon \sX \ra \sX \times \sX$ is representable and there exists a smooth map $U\twoheadrightarrow \mc X$ from a $0$-Artin stack which is surjective. We denote by $\mrm{dAlgStk}_R$ the $\infty$-category of derived algebraic stacks.

		\item A derived stack $\mc X$ over $R$ is called \textbf{geometric} if it is algebraic and diagonal $\Delta_{\sX}\colon \sX \ra \sX \times \sX$ is affine. We denote by $\mrm{dGeoStk}_R$ the $\infty$-category of derived geometric stacks.
		\item A derived stack $\mc X$ over $R$ is called \textbf{Deligne--Mumford} if it is algebraic and $\sX^{\mrm{cl}}$ is a classical Deligne-Mumford stack over $\pi_0(R)$.  
	\end{enumerate}	

\begin{rem} We reminder the reader of certain properties $\mbf P$ of morphisms of derived algebraic spaces which will appear throughout the paper. Let $f\colon \sX \rightarrow \sY$ be a morphism of derived algebraic spaces. 
\begin{enumerate}
\item For $\mbf P = \{ \text{proper}, \text{closed immersions} \}$, then we say that $f$ is $\mbf P$ if the map of classical algebraic spaces $f^{\mrm{cl}}x \sX^{\mrm{cl}} \rightarrow \sY^{\mrm{cl}}$ is $\mbf P$ in the usual sense. 
\item For $\mbf P = \{ \text{\'etale}, \text{smooth}, \text{open immersion}, \text{flat} \}$,  then $f$ is $\mbf{P}$ if for any \'etale 0-representable cover $U\twoheadrightarrow \sX$ (so that $U = L_{\et}(\sqcup_{i\in I} \Spec A_i)$) the composite $U \rightarrow \sY$, which is $0$-representable morphism, is $\mbf P$. 
\end{enumerate}
\end{rem}

By definition, we have the sequence of inclusions
\[
\mrm{dAlgSpc}_R \subset \mrm{dGeoStk}_R \subset \mrm{dAlgStk}_R \subset \mrm{dStk}_R.
\]
A derived algebraic stack $\sX$ is called \textbf{quasi-compact (qc)} if there is a smooth surjective map $\Spec A \twoheadrightarrow \sX$ from a derived affine scheme. It is called \textbf{quasi-separated (qs)} if diagonal $\Delta_{\sX}\colon \sX \ra \sX \times \sX$ is quasi-compact, meaning that the pull-back of $\Delta_{\sX}$ to any point $\Spec A\ra \sX\times \sX$ is a quasi-compact derived algebraic space.
\end{defn}	

We denote by 
\[
\mrm{dAlgStk}^{\mrm{qcqs}}_R \subset \mrm{dAlgStk}_R,
\]
the inclusion of the $\infty$-category of qcqs derived algebraic stacks into all derived algebraic stacks; we us a similar notation for $\mrm{dGeoStk}^{\mrm{qcqs}}_R$.

\begin{rem}
	By the argument in \cite[Chapter 2, Proposition 4.4.9]{GRderalg1} any algebraic (or, more generally, any $n$-Artin) stack is convergent. 
\end{rem}

\begin{exam}[Classifying stack] Let for simplicity $R=k$ be a field of characteristic 0.
	Let $G$ be a finite type affine group scheme (aka algebraic group) over $k$; note that $G$ is automatically smooth. For any $A\in \CAlg^{\mrm{an}}_k$ the space $G(A)$ has a natural structure of a group-like $\mbb E_1$-monoid and we can consider its classifying space $B(G(A)):= \{*\}_{hG(A)}$. The \textbf{classifying stack $BG$} is defined as \'etale sheafification of the presheaf $A\mapsto B(G(A))$. Furthermore, $BG$ is a geometric stack: the map $\Spec k \ra BG$ (induced by $\{*\}\ra B(G(A))$ on $A$-points) gives a smooth cover, and the diagonal is affine.
\end{exam}	

\begin{exam}[Quotient stack]\label{exam:quot_stack}
	More generally, given a derived algebraic space $X$ over $k$ endowed with an action of $G$ we can form \textbf{the quotient stack} $[X/G]$, defined as the \'etale sheafification of the prestack $R\mapsto X(R)_{hG(R)}$. The map $X\ra [X/G]$ is a smooth affine cover and one can check that $[X/G]$ is a geometric stack (since in our convention $X$ has affine diagonal).
\end{exam}

\subsubsection{ANS and perfect stacks} We now recall conditions on stacks which are ultimately about their stabilizer groups. These conditions are rather standard in stacky extensions of motivic and $K$-theoretic objects; see for example \cite{hoyois-sixops}, \cite{khan2021generalized}, \cite{BKRS}, \cite{dgm-elden-vova}. We will be somewhat terse and refer the reader to \cite[Section 2.2.16]{dgm-elden-vova} for more details.

\begin{defn}\label{defn:reductive}
	An affine group scheme $G$ of finite type over a classical ring $R$ is called \textbf{linearly reductive} if the global sections functor
	\[
	\Gamma(BG,-)\colon  \mrm{D}_{\mrm{qc}}(BG) \to \mrm{D}(R)
	\]
	is $t$-exact for the standard $t$-structure.
\end{defn}

\begin{exam}\label{exam:linred}
	The group $\mathbb{G}_m$ is linearly reductive over any base ring $R$. 
	If $R=k$ is a field of characteristic $0$ then $G$ is linearly reductive if and only if it is reductive: namely, its unipotent radical is trivial. In particular, in this case the general linear group $GL_n$ and all finite group schemes are linearly reductive. The group $\mathbb{G}_a$ is an example of a non-linearly reductive group.
\end{exam}

\begin{defn}\label{defn:nicegroup}
	A group scheme $G$ is said to be {\bf nice} if it fits into a short exact sequence 
	\[
	1 \to G_1 \to G \to G_2 \to 1
	\]
	such that $G_2$ is finite locally free  and $G_1$ is an algebraic torus (in other words the base change to the algebraic closure $(G_1)_{\bar{k}}$ is isomorphic to $\mathbb{G}_m^n$ for some $n$). 
\end{defn}
\begin{rem}
	Nice group schemes are linearly reductive (see \cite[Remark~2.2]{alper2019tale}). 
\end{rem}

The following class of stacks is the most important one for our considerations.

\begin{defn}\label{defn:ANS stack}
	An algebraic
	derived stack $\sX$ is called {\bf ANS}\footnote{ANS = algebraic with nice stabilizers} if the diagonal of 
	$\sX$ is affine and the stabilizer group of every point $\Spec K \to \sX$ is nice.
\end{defn}


	One of the structural results on ANS stacks that will be useful for us is the following theorem of Alper--Hall--Halpern-Leistner--Rydh (see \cite[Theorem~1.9]{alper2022artin}; also \cite[Theorem~2.12]{khan2021generalized}):
	\begin{thm}\label{thm:niscover}
		Let $\sX$ be a quasi-compact ANS derived algebraic stack over a field $k$. 
		Then there is an affine Nisnevich cover $\{[\Spec A_i/G_i] \to \sX\}_{i\in I}$, where $R_i$ are some connective $k$-algebras and 
		$G_i$ are nice group $k$-schemes. 
	\end{thm}

%
\begin{rem}\label{rem:ANS passes through representable maps}
	Note that by definition an ANS stack $\sX$ is a derived geometric stack. Also, by \cite[Lemma A.1.7]{BKRS} if $f\colon \sX' \ra \sX$ is representable, then $\sX'$ is ANS if $\sX$ is. 
\end{rem}	

\begin{exam}
Over a field of characteristic zero all derived Deligne-Mumford stacks with affine diagonal are ANS. 
\end{exam}
\begin{exam}
	Given a nice group scheme $G$ acting on a derived algebraic space $X$, the quotient stack $[X/G]$ is ANS. 
\end{exam}

Lastly, we will often impose the following sheaf-theoretic restriction on our stacks. It ensures that derived $\infty$-category of quasicoherent sheaves admits an ample supply of dualizable objects. 

\begin{defn}[{\cite{BFN}, \cite[Definition A.3.1]{BKRS}}]\label{def:perfect-stacks}
A derived algebraic stack $\sX$ is called \textit{perfect} if the natural functor
$$
\mrm{Ind}(\Perf(\sX)) \rightarrow \mrm{D}_{\mrm{qc}}(\sX)
$$
is an equivalence of categories.  
\end{defn}

\begin{exam}\label{ex:perfect stacks}
In characteristic zero, many quotient stacks are perfect.  Indeed, by \cite[Corollary 3.22]{BFN} the quotient $[X/G]$ of a derived scheme $X$ quasi-projective over a field in characteristic 0 by 
an algebraic group $G$ is perfect. Also, by \cite[Theorem A.3.2]{BKRS} a quasi-compact ANS stack is 
perfect.
\end{exam}	

\subsection{Localizing invariants}\label{sec:localizing} In this section, we recall some preliminaries about localizing invariants that we will need in this paper, following \cite{blumberg2013universal} and the relative version as developed in \cite{HSS}. Throughout, let us fix
\[
\sE \in \CAlg(\Cat_{\infty}^{\perf}),
\]
a symmetric monoidal, idempotent complete stable $\infty$-category (whose tensor product is exact in each variable) which is \textbf{rigid}: all of its objects are dualizable. We set the $\infty$-category of $\sE$-modules in $\Cat^{\perf}$ as:
\[
\Cat^{\perf}(\sE) := \Mod_{\sE}(\Cat^{\perf}).
\]
We have the forgetful functor $U_{\sE}: \Cat^{\perf}(\sE) \rightarrow \Cat^{\perf}$. 

\begin{exam}\label{exam:only-exam}Examples of $\sE$ include $\Perf(\sX)$ where $\sX$ is a perfect stack in the sense of Definition~\ref{def:perfect-stacks}. The example that we will mostly care about is $\sX = BG$ where $G$ is an affine algebraic group in characteristic zero.
\end{exam}

We will use the terminology of \cite{GW-II}.

\begin{defn}\label{defn:localizing_invariant}
A pair of composable morphisms  
\[
\sA \xrightarrow{i} \sB \xrightarrow{p} \sC
\]
in $\Cat^\perf$ is called a {\bf Karoubi sequence} (see \cite[Definition A.3.5, Proposition A.3.7]{GW-II}), if the composite is null, $i$ is fully faithful and the quotient map (in stable categories)
$\sB/\sA \to \sA$ is an equivalence in $\Cat^\perf$ (in other words, equvialence after idempotent completion). 

A pair of composable morphisms
\[
\sA \xrightarrow{i} \sB \xrightarrow{p} \sC
\]
in $\Cat^{\perf}(\sE)$ is said to be a {\bf Karoubi sequence} (see \cite[Definition 5.3]{HSS}) if $U_{\sE}$ of it is a Karoubi sequence in the above sense. Let $\sV$ be a stable $\infty$-category. An {\bf ($\sV$-valued) localizing invariant (of $\sE$-linear categories)}
is a functor
\[
F: \Cat^{\perf}(\sE) \to \sV
\]
which sends $0$ to $0$ and  Karoubi sequences to fiber sequences. Such an $F$ is \textbf{finitary} if it preserves filtered colimits. 
\end{defn}

To formulate the completion theorem, we will need the following structure on $K$-theory of $\sE$-linear categories.

\begin{constr}\label{constr:k-lax} According to the formula of \cite[Theorem 5.25]{HSS}, generalizing the main theorem of \cite{blumberg2013universal}, (non-connective) $K$-theory promotes to a lax monoidal functor
\[
K\colon  \Cat^{\perf}(\sE) \to \Spt. 
\]
Since $\sE$ is the unit object of $ \Cat^{\perf}(\sE)$, $K(\sE)$ is canonically an $\bbE_{\infty}$-algebra and the functor $K$ factors uniquely through
\begin{equation}\label{eq:k-factors}
\begin{tikzcd} 
& \Mod_{K(\sE)}(\Spt) \ar{d}\\
\Cat^{\perf}(\sE) \ar[dashed]{ur}{K} \ar{r}{K} & \Spt. 
\end{tikzcd}
\end{equation}
\end{constr}
Therefore, if $\sC \in \Cat^{\perf}(\sE)$, then there is a canonical $K(\sE)$-module structure on $K(\sC)$. We will use this structure implicitly without further comment.

\begin{rem}[A remark on localizing invariants as ``$K$-modules'']\label{rem:unit} In fact, more is true. Let $\Fun_{\mrm{loc}}^{\mathrm{fin}}(\Cat^{\perf}(\sE), \Spt)$ be the 
$\infty$-category of finitary localizing invariants of $\sE$-linear categories. Under Day convolution, we have a symmetric 
monoidal $\infty$-category $\Fun_{\mrm{loc}}^{\mathrm{fin}}(\Cat^{\perf}(\sE), \Spt)^{\otimes}$ and, by 
\cite[Theorem 5.25]{HSS}, $K$-theory is the unit object and therefore any localizing invariant admits a coherent action of 
$K$-theory as a functor out of $\Cat^{\perf}(\sE)$. We will not need this general structure as other multiplicative 
localizing invariants involved in this paper will come with preferred multiplicative maps from $K$-theory.
\end{rem} 


%
%

%
%
%

There are other properties of localizing invariants that will be useful for us. Let
\[
E\colon \Cat^{\perf}(\sE) \rightarrow \sV
\]
be a localizing invariant, valued in $\sV$ a stable, presentable $\infty$-category.
\begin{enumerate}
\item $E$ is said to be {\bf truncating} if for any connective algebra $A$, and an $\mc E$-module stucture on $\Perf(A)$ the map $E(\mrm{Perf}(A)) \to E(\mrm{Perf}(\pi_0 A))$ is an equivalence;
\item $E$ is said to be {\bf $\bbA^1$-invariant} if for any $\sC \in \Cat^{\perf}(\sE)$, the canonical map\footnote{Here, we denote $\sC[t]:=\sC\otimes \Perf(\mbb S[t])$.} 
$E(\sC) \rightarrow E(\sC[t])$ is an equivalence and $E$ factors as $F \circ E'$ where $E'$ is a {\it finitary} localizing invariant\footnote{This technical assumption appears as (C3) in \cite{tabuada2020motivic}. All known localizing invariants satisfy that.}.
\end{enumerate}

We now discuss some examples which will feature prominently throughout the paper. 
\begin{exam}[$K$-theoretic examples]\label{rem:examples} 
Here we list some examples of localizing invariants of $K$-theoretic nature that will be relevant for us. 
\begin{enumerate}
\item The primary example is the nonconnective algebraic K-theory $K$;
it is neither truncating, nor $\mathbb{A}^1$-invariant, and is the central character of the present paper. 
\item Let $\sE = \Perf(R)$ where $R$ is a base commutative ring spectrum. We define Weibel's homotopy-invariant $K$-theory $KH$ as
\[
KH(\sC) = \lim\limits_{n \in \Delta} KH(\sC \otimes_R \Perf_{R[t]^{\otimes n}}).
\]
As name suggests, this is $\mathbb{A}^1$-invariant and is the initial $\mathbb{A}^1$-invariant localizing invariant under $K$-theory. In fact the same construction may be performed with any localizing invariant to produce an $\mathbb{A}^1$-invariant approximation thereof.
\item Thomason's $K(1)$-local algebraic $K$-theory, $L_{K(1)}K$ is defined by applying the $K(1)$-localization functor (recall that we have picked an implicit prime here) to $K$-theory. Results of Land-Meier-Mathew-Tamme \cite[Corollaries 4.23 and 4.25]{lmmt} shows, in particular, that $L_{K(1)}K$ is both $\mathbb{A}^1$-invariant and truncating whenever $\sE = \Perf(R)$, where $R$ is a classical commutative ring.
\item Let $\sE = \Perf(\bbC)$. Blanc's semitopological $K$-theory $K^{\mathrm{st}}$ and topological $K$-theory $K^{\mrm{Top}}$ \cite{blanc} are also examples of localizing invariants constructed out of algebraic $K$-theory. A theorem of Konovalov \cite{konavalov} proves that $K^{\mathrm{st}}$ is truncating, while its $\mathbb{A}^1$-invariance is established in \cite[Lemma 3.1]{antieau-heller}. This leads to the same properties for topological $K$-theory. 
\end{enumerate}
\end{exam}
\begin{exam}[Hochschild-theoretic examples]\label{rem:exam2} Another class of localizing invariants that we will study are those constructed out of Hochschild homology.
\begin{enumerate}
\item Fixing a base commutative ring spectrum $R$ and a small $R$-linear stable $\infty$-category $\sC$ the trace map 
\[
\mrm{D}(R) \to \Ind(\sC^{\mrm{op}}) \otimes_R \Ind(\sC) \to \mrm{D}(R)
\]
is determined by its value on the unit $R \in \mrm{D}(R)$. The corresponding object
\[
HH(\sC/R) \in \mrm{D}(R),
\] 
is called the {\bf Hochshild homology} of $\sC$ and determines a $\mrm{D}(R)$-valued localizing invariants of $\Perf(R)$-linear categories. If $R = \mathbb{S}$, the sphere spectrum $HH(\sC/\bbS) = THH(\sC)$ is often referred to as \textbf{topological Hochschild homology}. For the purposes of this paper, we are mostly interested in the case that $R$ is a base discrete commutative ring. Hochschild homology is a localizing invariant, which is neither  $\mathbb{A}^1$-invariant 
nor truncating. There is a natural $S^1$-action on $HH(\sC/R)$. In fact it has a further cyclotomic structure \cite{nikolaus-scholze}; but since we are not really considering the topological variant of it too seriously we will not go into it further. In section~\ref{sec:loops}, we will review another perspective of Hochschild homology in terms of loop stacks. This will be crucial in understanding completions of Hochschild homology at various ideals. 

\item The $S^1$-action on $HH(\sC/R)$ gives rise to other invariants:
\[
HC^-(\sC/R) =: HH(\sC/R)^{hS^1} \qquad HH(\sC/R)_{hS^1} =: HC(\sC/R) \qquad HH(\sC/R)^{tS^1}=: HP(\sC/R).
\]
These invariants fit into a fibre sequence:
\[
HC(\sC/R)[1] \to HC^-(\sC/R) \to HP(\sC/R).
\]
\item Let $k$ be a field of characteristic zero. Then the $HP(\sC/k)$ is $\mathbb{A}^1$-invariant and truncating. The latter basically follows from the former and is a theorem of Goodwillie \cite{goodwillie-hp}. For a modern account, we refer the reader to \cite{ramzi-hp}. While these accounts only prove $\mathbb{A}^1$-invariance on $k$-dga's the arguments of \emph{loc. cit.} can be used to promote the results to arbitrary $k$-linear categories. 
\end{enumerate}
\end{exam}

\section{Descent for cohomology theories on derived stacks} \label{sec:derived}

In this section, we discuss descent for cohomology theories for algebraic stacks. We start with a cohomology theory (more specifically, a localizing invariant) $E\colon \mrm{dAff}^{\op}_R \rightarrow \Spt$ on derived affine schemes. Invoking the Knutson-Lurie Theorem~\ref{thm:cover}, which states that every qc derived algebraic space is covered by derived affine schemes, we can canonically extend $E$ to $\mrm{dAlgSpc}_R^{\mrm{qc}}$ using Nisnevich descent. At this point, to further extend $E$ to derived algebraic stacks, we have two natural options: for a Grothendieck topology $\tau$ on $\mrm{dAff}_R$ for which $E$ is a $\tau$-sheaf (e.g. Nisnevich), one could impose descent with respect to suitably defined geometric \textit{$\tau$-coverings} or with respect to all \textit{$\tau$-surjections}. We discuss this difference in detail in Section~\ref{def:nis-covers}. In particular, we will see that the functor $\sX \mapsto E(\Perf(\sX))$ has descent with respect to Nisnevich covers when restricted to perfect stacks (Corollary \ref{cor:Nisnevich descent for perfect stacks}), while extending $E$ via right Kan extension has descent with respect to all Nisnevich surjections (Theorem~\ref{thm:universal}). From this point of view, the Atiyah-Segal completion theorem bridges the gap between these two different forms of descent. 

Lastly, in Section \ref{sec:cdh excision} we discuss cdh excision for localizing invariants of stacks.

\ssec{Universal descent}
To formulate descent results on stacks, we will use the language of universal $E$-descent as in \cite{LiuZheng2}.

\newcommand{\triplearrows}{\begin{smallmatrix} \to \\ \to \\ 
		\to \end{smallmatrix} }

\begin{defn}\label{def:eff-desc} Let $\sC$ be an $\infty$-category with pullbacks, $\sD$ an $\infty$-category with all limits, and suppose that we are given a functor
	\[
	E\colon \sC^{\op} \rightarrow \sD;
	\]
	we say that a morphism $f\colon \sX \rightarrow \sY$ in $\sC$ is \textbf{of (resp. universal) $E$-descent} if the canonical map
	\[
	E(\sY) \rightarrow \lim_{\Delta} \left( E(\sX) \rightrightarrows E(\sX \times_{\sY}
	\sX) \triplearrows \dots \right)
	\]
	is an equivalence (resp. remains so after every pullback of $f$ in $\sC$). We also say that $E$ \textbf{satisfies (resp. universal) descent} with respect to a collection of morphisms $\sE$ in $\sC$ if any $f \in \sE$ is of (resp. universal) $E$-descent.
\end{defn}

\begin{exam} Let $\tau$ be a Grothendieck topology on an $\infty$-category $\sC$ with pullbacks and coproducts. Let $\{ X_i \rightarrow X \}$ be collection of morphisms generating a $\tau$-sieve. Then, by definition, the map $\bigcoprod X_i \rightarrow X$ is of universal descent for any $\tau$-sheaf. 
\end{exam}

The following records standard stability properties of morphisms of $E$-descent \cite[Lemma 3.1.2]{LiuZheng2}:

\begin{lem}\label{lem:desc-basics}  Let $\sC$ be an $\infty$-category with pullbacks, $\sD$ an $\infty$-category with limits,  and suppose that we are given a functor
	\[
	E\colon \sC^{\op} \rightarrow \sD.
	\]
	\begin{enumerate}
		\item Every retraction is of universal $E$-descent.
		\item If (composable) morphisms $f, g$ are of universal $E$-descent, then so is the composition $f\circ g$.
		\item If the composition $f \circ g$ is of universal $E$-descent, then $f$ is.
		\item Morphisms of universal $E$-descent are stable under base change.
	\end{enumerate}
\end{lem}

\subsection{Right Kan extensions of cohomology theories}\label{subsec:rke} In this section, we will focus on answering Question~\ref{quest:main} for the extension of $E$ to algebraic stacks given by right Kan extension. Most of the material here is covered and extended further in the work of Khan-Ravi \cite{khan2021generalized} and Chowdhury \cite{chowdhury} (see Remark \ref{rem:Chowdhury and Khan-Ravi}). We will only develop some rudiments of this formalism needed for the Atiyah-Segal completion theorem.

Let $R\in \CAlg^{\mrm{an}}$ be a base animated commutative ring. For any functor $E\colon \mrm{dAff}_R^{\mrm{op}} \ra \mc \sD$ to a complete $\infty$-category $\mc D$ we can consider its right Kan extension 
$$
R^{\mrm{dAff}}E\colon \mrm{dPStk}_R^\op \ra \mc D
$$
with respect to the embedding $\CAlg_R^{\mrm{cn}}\simeq \mrm{dAff}_R^{\op}\subset \mrm{dPStk}_R^\op$ 
(cf. \cite[Construction~12.1~and~Remark~12.3]{khan2021generalized}). In most examples of our interest, the functor on affine schemes will be of the form $$E \circ \Perf\colon \mrm{dAff}_R^{\op} \ra \Spt,$$ where $E\colon \Cat^{\perf}_R \rightarrow \Spt$ is an $R$-linear localizing invariant. In this case we will still simply write $R^{\mrm{dAff}}E$ for the corresponding right Kan extension.

Note that since $\mrm{dPStk}_R$ is freely generated by $\mrm{dAff}_R\subset \mrm{dPStk}_R$ under colimits, $R^{\mathrm{dAff}}E$ sends colimits in $\mrm{dPStk}_R$ to limits in $\mc C$. Since any $\sX\in \mrm{dPStk}_R$ is reconstructed as $$\colim_{\{\Spec A\ra \sX\}} \Spec A \in \mrm{dPStk}_R,$$ we can explicitly compute $R^{\mathrm{dAff}}E(\sX)$ as
\begin{equation}\label{eq:colimofrep}
R^{\mathrm{dAff}}E(\sX)\simeq \lim_{\Spec A\ra \sX} E(\Spec A).
\end{equation}
The next proposition is a purely formal consequence of this observation.

\begin{prop}\label{prop:spc} Let $\tau$ be a Grothendieck topology on $\mrm{dAff}_R$ and let $\sD$ be a presentable $\infty$-category. Suppose that
\[
E\colon \mrm{dAff}_R^{\op}  \rightarrow \sD,
\]
is a $\tau$-sheaf. Then any $\tau$-surjection $\pi\colon \sX \rightarrow \sY$ in $\mrm{PStk}_R$ is of universal $R^{\mrm{dAff}}E$-descent.
\end{prop}

\begin{proof} First, let us assume that $\sD = \Spc$ so that  $E\colon  \mrm{dAff}_R^{\op} \rightarrow \Spc$ is itself a prestack;~\eqref{eq:colimofrep} then states that:
\[
R^{\mrm{dAff}}E(\sX) \simeq  \Maps_{\mrm{PStk}_R}(\sX, E). 
\]
Since $\tau$-sheafification is a localization, we get that $\Maps_{\mrm{PStk}_R}(\sX, E) \simeq \Maps_{\mrm{PStk}_R}(L_{\tau}\sX, E)$. By definition, if $\pi$ is a $\tau$-surjection, then $|\sX_\bullet|\ra \sY$ is an $L_{\tau}$-equivalence. We thus conclude that
\[
\Maps_{\mrm{PStk}_R}(\sX, E) \simeq \lim_{[\bullet]\in \Delta}(R^{\mathrm{dAff}}E(\sX_\bullet)).
\]
Since any presentable $\infty$-category is a localization of a category of presheaves valued in spaces \cite[Proposition 5.5.1.1]{HTT}, the result follows. 
\end{proof}


The previous proposition states that $R^{\mrm{dAff}}E$ satisfies universal descent for $\tau$-surjections, whenever $E$ is a $\tau$-sheaf on $\mrm{dAff}_R$. We now will also characterize certain functors which are of the form $R^{\mrm{dAff}}E$ via the following general lemma. Let us recall some terminology. Let $\sC$ be an $\infty$-category with pullbacks and $\sC' \subset \sC$ be a full subcategory, stable under pullbacks. We say that a morphism $f\colon X \rightarrow Y$ in $\sC$ is \textbf{representable in $\sC'$} if for all $Z \in \sC'$ and all morphisms $Z \rightarrow Y$, the pullback $Z \times_Y X$ belongs to $\sC'$.

\begin{lem}\label{lem:rke}\cite[Proposition 4.1.1]{LiuZheng2} Let $\sC$ be an $\infty$-category with pullbacks and $\sD$ a presentable $\infty$-category. Let $\sE \subset \sC^{\Delta^1}$ a collection of morphisms stable under pullback and composition. Let $\iota \colon \sC' \subset \sC$ be a full subcategory closed under pullbacks. Assume that for any $X \in \sC$ there exists a morphism $f\colon  Y \rightarrow X$ such that 
\begin{enumerate}
\item  $Y \in \sC'$, 
\item $f$ is representable in $\sC'$ and 
\item $f$ is in $\sE$. 
\end{enumerate}
Consider a presheaf
\[
\sF\colon   \sC ^{\op} \rightarrow \sD,
\]
such that any arrow in $\sE$ is of universal $\sF$-descent. Then the canonical map 
$$\sF \rightarrow R^{\sC/\sC'} (\sF|_{\sC'})$$ is an equivalence.

%

\end{lem} 

\begin{rem} Let us also formulate Lemma~\ref{lem:rke} as an equivalence of categories, as is done \cite[Proposition 4.1.1]{LiuZheng2}. With the notation of Lemma~\ref{lem:rke} we can consider the full subcategories
\[
\Fun^{\sE}(\sC^{\op}, \sD) \subset \Fun(\sC^{\op}, \sD) \qquad \Fun^{\sE \cap \sC'}( \left( \sC' \right)^{\op}, \sD) \subset \Fun(\left( \sC' \right)^{\op}, \sD)
\]
spanned by functors of universal $\sE$ and $\sE\cap \sC'$-descent correspondingly.
The right Kan extension functor
\[
\iota_{\ast}\colon\Fun(\left( \sC' \right)^{\op}, \sD) \rightarrow\Fun(\sC^{\op}, \sD)
\]
implements the right adjoint to restriction along $\iota$. Lemma~\ref{lem:rke} then asserts an equivalence over this functor, in other words a commutative diagram:
\[
\begin{tikzcd}
 \Fun^{\sE \cap \sC'}( \left( \sC' \right)^{\op}, \sD) \ar{r}{\simeq} \ar[hook]{d} & \Fun^{\sE}(\sC^{\op}, \sD) \ar[hook]{d} \\
 \Fun(\left( \sC' \right)^{\op}, \sD) \ar{r}{\iota_{\ast}} & \Fun(\sC^{\op}, \sD).
\end{tikzcd}
\]
\end{rem}

The next theorem characterizes $R^{\mrm{dAff}}E$ as the universal extension of $E$ to a presheaf on a large class of stacks which respects Nisnevich surjections. Let $\mrm{dAlgStk}_R^{\qcqs,\mrm{sd}}\subset \mrm{dAlgStk}_R^{\qcqs}$ denote the full subcategory spanned by stacks with separated diagonal.

\begin{thm}\label{thm:universal} Let $E$ be a $k$-linear localizing invariant. Then there exists a unique functor
\[
\widetilde{E}\colon \left(\mrm{dAlgStk}_R^{\qcqs,\mrm{sd}}\right)^{\op} \rightarrow \Spt
\]
characterized by the following properties:
\begin{enumerate}
\item there is an equivalence of presheaves on $\mrm{dAff}_R$: $E \xrightarrow{\simeq} \widetilde{E}|_{\mrm{dAff}_R^{\op}}$,
\item any Nisnevich surjection is of universal $\widetilde{E}$-descent.
\end{enumerate}
\end{thm}

\begin{proof} The existence part is given by setting $\widetilde{E}:= R^{\mrm{dAff}}E$ and invoking Proposition~\ref{prop:spc} which asserts that $\Nis$-surjections are of universal descent for such a functor. We now prove the uniqueness part. Let $\widetilde{E}$ be such a functor, it then admits a canonical map $\eta\colon \widetilde{E} \rightarrow R^{\mrm{dAff}}(\widetilde{E}|_{\mrm{dAff}})$. By part (1) of the hypothesis, we have an equivalence $(\widetilde{E}|_{\mrm{dAff}}) \simeq E$. Hence we just need to show that the map $\eta$ is an equivalence. We claim that $\eta'$ is an equivalence by Lemma~\ref{lem:rke}. Indeed, set $\sC' \subset \sC$ to be the inclusion $\mrm{dAff_R^{\qcqs}} \subset \mrm{dAlgStk}_R^{\qcqs,\mrm{sd}}$ and $\sE$ be the collection of Nisnevich surjections. The hypotheses are then verified by the derived Deshmukh's Theorem~\ref{thm:deshmukh}.

%

\end{proof}

\begin{rem}\label{rem:geostk} We expect Theorem~\ref{thm:deshmukh} to also hold without the separated diagonal assumptions. In the classical case this has been verified in a new version of \cite{deshmukh} where an improvement of the original result was obtained by Deshmukh and Hall. This uses a somewhat nontrivial ingredient --- the theory of Hilbert stacks. To prove a derived version of this more general result would probably require us to delve into the theory of the derived Hilbert stack which is surely interesting. Since, in this paper, we are mostly concerned with ANS stacks which are derived geometric stacks we leave it to the interested reader to extend Theorem~\ref{thm:universal} to that generality. 
\end{rem}

\begin{rem}[Chowdhury and Khan-Ravi's generalized cohomology of stacks] \label{rem:Chowdhury and Khan-Ravi} We sketch briefly what is done in \cite{chowdhury} and \cite{khan2021generalized} and how it relates to the present discussion. In both work, Chowdhury and Khan-Ravi begins with a six functor formalism defined on (derived) schemes; the most basic data of such is a functor $\mrm{D}^{\ast}\colon \Sch^{\op} \rightarrow \Cat_{\infty}$. These categories ``house" cohomology theories which are representable in these categories. In \cite{chowdhury} the extension to stacks is performed by taking the right Kan extension to stacks which one denotes by $\mrm{D}^{\ast}_{\mrm{Kan}}$, while \cite{khan2021generalized} takes a smaller ``lisse" extension: for an algebraic stack $\sX$ they set $\mrm{D}^{\ast}_{\vartriangleleft}(\sX) := \lim_{(U, u)} \mrm{D}^{\ast}(U)$ where $u\colon U \rightarrow \sX$ is a smooth surjective morphism and $U$ a scheme. These two extensions, typically agree for a large class of stacks. If $\mrm{D}^{\ast}$ satisfies \'etale descent then this is even true for any $n$-Artin stacks; if $\mrm{D}^{\ast}$ is just a Nisnevich sheaf the best known generality is due to Deshmukh \cite[Corollary 3.12]{deshmukh}: namely, the two extensions for quasi-separated algebraic stacks with a separated diagonal. 
	
	The focus of \cite{chowdhury} and \cite{khan2021generalized} is on developing a full six functor formalism for these extended theories. In particular one gets a notion of homology, Borel-Moore homology and compactly supported cohomology on stacks by taking global sections of some objects. In this paper, we are interested in extensions of cohomology theories from derived affine schemes to derived stacks and, in this particular section, we essentially perform a ``decategorified" version of \cite{chowdhury} and \cite{khan2021generalized}. One big difference is that our cohomology theories are not necessarily represented in the versions of $\mrm{D}^{\star}$ considered in these papers; for example they are not assumed to have $\bbA^1$-invariance, though in all examples they are Nisnevich sheaves on affine derived schemes. 
\end{rem}

%

%

\subsection{Nisnevich covers}\label{def:nis-covers} We now explain a relatively large class of morphisms which are of universal descent for any $E$ a localizing invariant. Nisnevich covers on algebraic stacks has been defined in several places in the literature: \cite[Section 6.1]{KO}, \cite[Section 3]{addendum}, \cite[Section 2C]{hoyois-krishna}, \cite[Section 2.1]{BKRS} to name a few. We define it in the context of derived stacks here, via a trivial modification. 

 To define Nisnevich covers, we first need to define the notion of being \textbf{completely decomposed} for algebraic stacks. If $\sX$ is a derived algebraic stack, and $x\colon \Spec K \rightarrow \sX$ is a morphism over $k$ where $K$ is a field over $k$, then it factors through $\tau^{\mrm{cl}}\sX$. In which case we get a map $\Spec k \rightarrow \sX^{\mrm{cl}}$ and it makes sense to speak of its \textbf{residual gerbe} $\eta_x \rightarrow \Spec K$ as in \cite[Tag 06ML]{stacks}; we refer to \cite[Appendix B]{rydh-devissage} and \cite[Tag 06RD]{stacks} for the existence of residual gerbes on qs algebraic stacks. Roughly speaking $\eta_x \rightarrow \Spec K$ is of the form $BG$ where $G$ is the stabilizer of the point at $x$ and the morphism $\eta_x \rightarrow \sX^{\mrm{cl}}$ is a closed immersion. 

\begin{defn}\label{def:cd} A morphism $f\colon \sX \rightarrow \sY$ is said to be \textbf{completely decomposed} if for any field point $y\colon  \Spec K \rightarrow \sY$, there is a $K$-point $x\colon \Spec K \ra \sX$ over $y$ such that the induced map on residual gerbes $\eta_y \rightarrow \eta_x$ is an isomorphism. 
\end{defn}

\begin{defn}\label{def:nis} Let $\sX \in \mrm{dAlgStk}^{\qcqs}_R$ and let $\{ p_i\colon \sX_i \rightarrow \sX \}_{i \in I}$ be a collection of representable \'etale morphisms where $\sX_i$ is qcqs. Then we say that that the above collection is a \textbf{Nisnevich cover} if $\coprod_{i \in I} \sX_i \rightarrow \sX$ is completely decomposed.
\end{defn}

\begin{exam}\label{exam:algspc} If $f\colon \sX \rightarrow \sY$ is a morphism of derived algebraic spaces, then the completely decomposed condition just means that any field point to $\sY$ lifts to $\sX$. 
\end{exam}

\begin{rem}[Nisnevich covers vs Nisnevich surjections] \label{rem:nis-covers} Let $f\colon \sY \rightarrow \sX$ be a Nisnevich cover of qcqs derived algebraic stacks. Then $f\colon \sY \rightarrow \sX$ is a Nisnevich surjection. Indeed, it is enough to show that for any $\Spec A \rightarrow \sX$, the map $\sY \times_{\sX} \Spec A \rightarrow \Spec A$ has a Nisnevich-local section. Using Knutson--Lurie Theorem~\ref{thm:cover} we can find a Nisnevich cover $\Spec S \rightarrow \sY \times_{\sX} \Spec A$; since the composite map $\Spec S \ra \Spec A$ is also a Nisnevich cover this provides a desired Nisnevich-local section.

The converse is not true: not all Nisnevich surjections are Nisnevich coverings. Let $X \in \mrm{dSch}_R$ and suppose that $G$ is a smooth $X$-group scheme such that any $G$-torsor over affine $\Spec A\ra X$ is Nisnevich-locally trivial. Examples include $G = \mrm{GL}_n, \mrm{SL}_n, \mrm{Sp}_{4n}$. Then the morphism of stacks $X \rightarrow [X/G]$ is a Nisnevich-effective epimorphism. However it is not, in general, a Nisnevich cover. Indeed, the map $\Spec k \rightarrow B\bbG_m$ cannot be a Nisnevich cover since the morphism on the residual gerbes over the point $\Spec k \rightarrow B\bbG_m$ is given by $B\bbG_m \rightarrow \Spec k$. 
\end{rem}

One of the main features of the Nisnevich topology is the equivalence between Nisnevich descent and excision for Nisnevich squares.

\begin{defn} A \textbf{Nisnevich square} in $\mrm{dAlgStk}^{\mrm{qcqs}}_R$ is a Cartesian square of derived algebraic stacks
		$$
	\begin{tikzcd}
		\sW\arrow[d] \arrow[r]& \sV\arrow[d,"p"] \\
		\sU\arrow[r,"j"]& \sX
	\end{tikzcd}	
	$$
	where $j$ is a Zariski open immersion\footnote{Since we are working in $\mrm{Stk}^{\mrm{qcqs}}$, it is implied that $j$ is a qc open immersion; there are, however, plenty of non qc open of a stack.} and $p$ is a representable \'etale morphism that induces an isomorphism away from $\sU$. 
%
\end{defn}	

\begin{prop}\label{prop:cd-descent} Let $\sD$ be an $\infty$-category with limits. Let
\[
\sF\colon \left( \mrm{dAlgStk}^{\mrm{qcqs}}_R \right)^{\op} \rightarrow \sD,
\]
be a presheaf. Then the following are equivalent:
\begin{enumerate}
\item $\sF(\emptyset) \simeq \ast$ and converts Nisnevich squares to cartesian squares;
\item any Nisnevich cover is of universal $\sF$-descent.
\end{enumerate}
\end{prop}
\begin{proof} The case of classical algebraic stacks is discussed in \cite[Section 2C]{hoyois-krishna} with the main content being \cite[Proposition 3.3]{rydh-devissage}. The derived case follows by the same argument.
\end{proof}

\begin{prop}[{\cite[Corollary 4.1.2]{BKRS}}]
Let 
		$$
	\begin{tikzcd}
		\sW\arrow[d] \arrow[r]& \sV\arrow[d,"p"] \\
		\sU\arrow[r,"j"]& \sX
	\end{tikzcd}	
	$$
	be a Nisnevich square of derived algebraic stacks, where $\sX$ and $\sV$ are perfect. Then for any localizing invariant $E$ the induced square 
			$$
	\begin{tikzcd}
	E(\sX)	\arrow[d,"j^*"] \arrow[r,"p^*"]& E(\sV)\arrow[d] \\
		E(\sU)\arrow[r]& E(\sW)
	\end{tikzcd}	
	$$
	is Cartesian.
	\end{prop}
From Proposition \ref{prop:cd-descent} we get the following.
\begin{cor}\label{cor:Nisnevich descent for perfect stacks}
	For any localizing invariant $E$ the restriction 
	$$
	E\colon \left(\mrm{dAlgStk}^{\mrm{perf}}_R\right)^{\mrm{op}} \ra \Spc , \qquad E\mapsto E(\Perf(\sX))
	$$ 
	to the subcategory $ \mrm{dAlgStk}^{\mrm{perf}}_R\subset  \mrm{dAlgStk}_R$ of perfect stacks satisfies descent for Nisnevich covers.
\end{cor}	

\begin{rem}
	Since any qc ANS stack is perfect we also get that $E$ satisfies descent for Nisnevich covers when restricted to qc ANS stacks. Note that if we have a Nisnevich cover $\coprod_{i \in I} \sX_i \rightarrow \sX$ and $\sX$ is ANS, then any of the $\sX_i$ is also ANS by Remark \ref{rem:ANS passes through representable maps}.
\end{rem}	

\sssec{Cohomology theories on derived algebraic spaces} Zariski descent implies that any cohomology theory on qcqs derived schemes is controlled by its values on derived affine schemes. Nisnevich descent allows us to control cohomology theories on qcqs derived algebraic spaces by its values on derived affine schemes. In other words, for many purposes, defining a cohomology theory on affine schemes is as good as defining one in qcqs derived algebraic spaces. This principle is formalized by:

\begin{thm}[Knutson-Lurie]\label{thm:cover} Let $X$ be qcqs algebraic space. Then there exists a finite collection of \'etale morphisms $\{ \Spec A_i \rightarrow X \}_{i = 1, \cdots n}$ constitute a Nisnevich covering of $X$.
\end{thm}

\begin{proof} In the case of a classical algebraic space, this result is due to Knutson \cite[Chapter II, Theorem 6.4]{Knutson1971}. The derived extension of this result  due to Lurie in \cite[Example 3.7.1.5]{SAG}.
\end{proof}
%

As a straightforward application of Lemma~\ref{lem:rke} we note the following lemma which says that any cohomology theory on $\mrm{dAff}_R$ which satisfies Nisnevich descent extends uniquely to one which has descent with respect to Nisnevich covers on qcqs derived algebraic spaces. In the notation of Lemma~\ref{lem:rke}, we set $\sC = \mrm{dAlgSpc}_R^{\mrm{qcqs}}$. We set to be $\sE = \mrm{Nis}$, the collection of arrows of the form $\coprod Y_i \rightarrow Y$, where $\{ Y_i \rightarrow Y \}$ is a Nisnevich cover. We let $\sC' := \mrm{dAff}_R$ and $\sE$ restricts to $\mrm{dAff}_R$ as the collection of Nisnevich covers of derived affine schemes. If $\sD$ is any presentable $\infty$-category, we then obtain an equivalence 
\begin{equation}\label{eq:nis-general}
\Fun^{\mrm{Nis}}((\mrm{dAlgSpc}_R^{\mrm{qcqs}})^{\op}, \sD) \simeq \Fun^{\mrm{Nis}}(\mrm{dAff}_R^{\op}, \sD),
\end{equation}
implemented by right Kan extension.

Let $E$ now be an $R$-linear localizing invariant. Let $\sX$ be a derived algebraic stack. The Atiyah-Segal completion theorem concerns the difference between $E$ evaluated on $\Perf(\sX)$ and on the right Kan extension of $E$ from derived affine schemes evaluated on $\sX$. The next lemma says that no such ambiguity occurs for derived algebraic spaces. 


%
%

%

\begin{prop}\label{prop:localizing-nis} Let $E\colon \Cat^{\perf}_R \rightarrow \Spt$ be a localizing invariant. Then the functor $E \circ \Perf\colon  (\mrm{dAlgSpc}_R^{\qcqs})^{\op} \rightarrow \Spt$ is right Kan extended from $\mrm{dAff}$. In particular, the morphism 
\[
E \circ \Perf \rightarrow R^{\mrm{dAff}}E
\]
is an equivalence.
\end{prop}

\begin{proof} We need only to prove that if $\{ Y_i \rightarrow Y \}$ is a Nisnevich cover, then $E \circ \Perf$ is of universal descent for $\coprod Y_i \rightarrow Y$. This is recorded in \cite[Proposition A.13]{CMNN}.

\end{proof}

%
%
%
%

\subsection{Cdh excision}\label{sec:cdh excision} Lastly, we discuss cdh-excision on stacks. We avoid going into the cdh topology on stacks, but rather view it as a property of a cohomology theory respecting certain naturally occurring pullback squares of derived stacks. To our knowledge, this idea was first considered in \cite[Section 6]{hoyois-krishna}, and later in \cite[Section 2.1]{BKRS}. 
%

\begin{defn}\label{defn:cdh} A \textbf{proper cdh square (or abstract blow-up square)} in $\mrm{dAlgStk}^{\mrm{qcqs}}$ is a commutative square of derived algebraic stacks	
	$$
	\begin{tikzcd}
		\sZ'\arrow[d,"g"] \arrow[r,"i'"]& \sX'\arrow[d,"f"] \\
		\sZ\arrow[r,"i"]& \sX
	\end{tikzcd}	
	$$
	with the following properties:
	\begin{itemize}
		\item the corresponding square of classical stacks is Cartesian: $(\mc Z')^{\mrm{cl}}\simeq (\mc X' \times_{\mc X} \mc Z)^{\mrm{cl}}$;
		\item $f$ is representable, $f^{\mrm{cl}}$ is finitely presented and proper, $i^{\mrm{cl}}$ is finitely presented, closed embedding whose complement is quasi-compact;
		\item $f$ induces an isomorphism on the complements\footnote{We refer the reader to \cite[Section 3.3.6]{SAG} for the discussion on the underlying topological space of a derived algebraic stack; we note that the open complement of a closed substack inherits the structure of a derived algebraic stack from $\sX$.} $f_{|\sX'\setminus \sZ'}\colon (\sX'\setminus \sZ')\simeq (\sX \setminus \sZ)$. 
	\end{itemize}
	
	A functor
	\[
	E\colon \left(\mrm{dAlgStk}_R^{\mrm{qcqs}}\right)^{\op} \rightarrow \sD
	\]
	is called \textbf{cdh-excisive} if $E$ converts abstract blow-up squares to pullback squares. 
\end{defn}	

An important example of an abstract blow-up square is given by
	$$
	\begin{tikzcd}
	\varnothing	 \arrow[d,"g"] \arrow[r,"i'"]& \varnothing \arrow[d,"f"] \\
	\sX^\mrm{cl}	 \arrow[r,"i"]& \sX.
	\end{tikzcd}	
	$$
Consequently, if $E$ takes values in a stable $\infty$-category and converts the empty stack to $0$, $E$ converts the closed immersion $\tau^{\mrm{cl}}\sX \hookrightarrow \sX$ to an equivalence. The following lemma will be useful.

\begin{lem}\label{lem:truncating} Let $\sD$ be an $\infty$-category with all limits and let
\[
E\colon  \mrm{dAff}_R^{\op} \rightarrow \sD
\]
be a functor such that for any $\Spec A \in \mrm{dAff}_R$, the map $E(\Spec A) \rightarrow E(\Spec \pi_0(A))$ is an equivalence. Then the map
\[
R^{\mrm{dAff}}E(\sX) \rightarrow R^{\mrm{dAff}}(\tau^{\mrm{cl}}\sX),
\]
is an equivalence. 
\end{lem}

\begin{proof} As recalled from Remark~\ref{exam:cl}, $\tau^{\mrm{cl}}\sX$ is obtained as the left Kan extension of the restriction of $\sX$ to classical affine schemes. We claim that the canonical map in prestacks:
\[
\tau^{\mrm{cl}}\sX \simeq \colim_{ \mrm{Aff}_{/\sX}} \Spec A \rightarrow \colim_{  \mrm{dAff}_{/\sX}} \Spec \pi_0(A),
\]
is an equivalence. Restricted to classical affine schemes, the map above is an equivalence. Hence it suffices to prove that the target is also left Kan extended from classical affine schemes. But now, each term is a classical affine scheme and this subcategory is closed under colimits. The result then follows from the formula for $R^{\mrm{dAff}}E$.

\end{proof}

%

\begin{prop}\label{lem:localizing-invt} Let $E$ be a $R$-linear, truncating localizing invariant. Then
\[
R^{\mrm{dAff}}E\colon (\mrm{dAlgStk}_R^{\qcqs,\mrm{sd}})^{\mrm{op}} \rightarrow \Spt
\]
converts an abstract blow-up square to a cartesian square.
\end{prop}

\begin{proof}  Fix an abstract blow-up square which we denote as	\[
\square = 	\begin{tikzcd}
		\sZ'\arrow[d,"g"] \arrow[r,"i'"]& \sX'\arrow[d,"f"] \\
		\sZ\arrow[r,"i"]& \sX
	\end{tikzcd}		
	\]
Our goal is to prove that $R^{\mrm{dAff}}E(\square)$ is cartesian. According to Theorem~\ref{thm:deshmukh} we we may find a Nisnevich surjection $\pi\colon \Spec R \rightarrow \sX$. Denote by $\square_{\pi}$ the base change of $\square$ along $\pi$. Forming the \v{C}ech nerve of $\square$ along the map $\pi$, we have an augmented cosimplicial diagram of squares $\square \rightarrow \square_{\pi}^{\times \bullet}$. By Lemma~\ref{prop:spc}, we get that $R^{\mrm{dAff}}E(\square) \simeq \lim_{\Delta} R^{\mrm{dAff}}E(\square_{\pi}^{\times \bullet})$, hence it suffices to prove that for each $n \in \bbN$, the square $R^{\mrm{dAff}}E(\square_{\pi}^{\times n})$ is cartesian. Now, each such square is an abstract blow-up square of derived algebraic spaces. By Lemma~\ref{prop:localizing-nis}, $R^{\mrm{dAff}}E \simeq E \circ \Perf$ and hence it suffices to prove that the latter preserves abstract blow-up squares. This is then a special case of \cite[Corollary 5.2.6]{dgm-elden-vova}.

\end{proof}

\section{Formulating the Atiyah-Segal completion theorem}\label{sec:setup}  As noted above, not all natural extensions of $E$ satisfy descent with respect to $\tau$-surjections even if $E$ does on derived affine schemes. Moreover, $R^{\mrm{dAff}}E$ is an essentially unique extension of $E$ that has satisfies this form of
$\tau$-descent. In this section, we define another extension of $E$ that is more intrinsic to $E$, at least when evaluated on quotient stacks $[X/G]$ with a fixed group $G$. This extension is a completion of $E$ at a certain ideal, called the Atiyah-Segal ideal in honor of their analogous work in topology in the context of topological $K$-theory \cite{atiyah-characters, atiyah-segal}. 

\subsection{Preliminaries on completions}\label{sec:sag-completion}
We begin by briefly recalling the machinery of completions of structured ring spectra as in \cite[Chapter 7.3]{SAG}. To the authors' knowledge, the relevance of this machinery to completion theorems for $K$-theory was first noted in \cite{tabuada2020motivic}. Let $R$ be a commutative ring spectrum and $I \subset \pi_0(R)$ an ideal. According to \cite[Proposition 7.2.4.4]{SAG} there is a semiorthogonal decomposition
\[
D(R) \simeq \langle D(R)^{I-\mrm{nil}} , D(R)^{I-\mrm{loc}} \rangle.
\]
Here, $ D(R)^{I-\mrm{nil}}$ is the subcategory of \textbf{$I$-nilpotent} modules: those $M$ such that $M[x^{-1}] \simeq 0$ for all $x \in I$ and \textbf{$I$-local} objects are exactly those which are right orthogonal to $I$-nilpotent objects. On the other hand, the subcategory \textbf{$I$-complete modules}, denoted by $D(R)^{I-\mrm{cpl}}$, is defined as the right orthogonal to $I$-local objects \cite[Definition 7.3.1.1]{SAG}: in other words an $R$-module $M$ is $I$-complete if and only if for all $I$-local module $N$ we have that $\Maps(N, M) \simeq 0$. 

Under mild hypotheses, the category of complete and nilpotent modules coincide; we formulate this equivalence. The inclusion $i_*:D(R)^{I-\mrm{cpl}} \hookrightarrow D(R)$ (resp. $D(R)^{I-\mrm{nil}} \hookrightarrow  D(R)$) admits a left adjoint $i^*\colon D(R) \rightarrow D(R)^{I-\mrm{cpl}}$ \cite[Notation 7.3.1.5]{SAG} (resp. a right adjoint $i^!\colon   D(R) \rightarrow  D(R)^{I-\mrm{nil}}$ \cite[Notation 7.2.4.6]{SAG}). The unit map
\[
M \rightarrow i_*i^*M =: M^{\wedge}_I
\]
witnesses the target as the \textbf{$I$-adic completion of $M$}, while the counit map
\[
\Gamma_I(M) := i_*i^!M \rightarrow M
\]
witnesses the domain as the $I$-local cohomology of $M$. We now further assume that $I$ is \emph{finitely generated} (which is the case for the ideals that appear throughout this paper), then each adjoint pair induces an equivalence $ D(R)^{I-\mrm{nil}} \simeq  D(R)^{I-\mrm{cpl}}$ \cite[Proposition 7.3.1.7]{SAG}. Identifying this category as $ D(R)^{I-\mrm{cpl}}$, the inclusion
\[
i_*:  D(R)^{I-\mrm{cpl}} \hookrightarrow  D(R)
\]
thus admits a left adjoint $i^*$ and a right adjoint $i^!$ and we denote the endofunctor $i_*i^*\colon  D(R) \rightarrow  D(R)$ by $(-)^{\wedge}_I$. Furthermore, we note that the endofunctor $(-)^{\wedge}_I$ is symmetric monoidal since $i^*$ is a strong symmetric monoidal functor \cite[Corollary 7.3.5.2]{SAG}.

\begin{rem}[Formula for the $I$-completion]\label{rem:formula for derived completion on module} Let $(r_1,\ldots, r_n)$ be a set of generators for $I$ and let $M\in  D(R)$. For $r\in R$ denote by $[R/r]$ the cofiber $R \xrightarrow{r} R$. The module
	$
	M\otimes_R [R/r_1^{i_1}] \otimes_R \ldots \otimes_R [R/r_n^{i_n}]
	$
	is $I$-complete for any $n$-tuple of powers $(i_1,\ldots i_n)\in \mbb N^n$ (see e.g. \cite[Corollary 7.3.3.3]{SAG}). The natural maps $
	M\ra M\otimes_R [R/r_1^{i_1}] \otimes_R \ldots \otimes_R [R/r_n^{i_n}]
	$
	then give a map 
	$$
	M^\wedge_I \ra \lim_{(i_1,\ldots, i_n)} M\otimes_R [R/r_1^{i_1}] \otimes_R \ldots \otimes_R [R/r_n^{i_n}],
	$$
	which in fact is a quasi-isomorphism (this follows by induction on $n$ from \cite[Corollary 7.3.2.1]{SAG}, with the induction step provided by \cite[Proposition 7.3.3.2]{SAG}).
\end{rem}	

\begin{rem}[Change of rings]
	Let $R\ra R'$ be a homomorphism of commutative ring spectra. Let $M'$ be an $R'$-module. Let $I\subset \pi_0(R)$ be a finitely generated ideal and let $I'\subset \pi_0(R')$ be the ideal generated by the image of $I$. Then by \cite[Corollary 7.3.3.6]{SAG} the $I'$-adic completion $(M')^\wedge_{I'}$ exhibits the $I$-adic completion of $M'$ as an $R$-module. In such situation we will thus often write $(M')^\wedge_{I}$ instead of $(M')^\wedge_{I'}$ to simplify the notation.
\end{rem}	

\begin{rem}[A warning about various completions] \label{rem:completions} Starting with Definition~\ref{defn:de Rham prestack and completion}, we will see a proliferation of other completion functors. While these can be compared, the context differs slightly. We will mostly apply the theory $I$-adic completion to the $K$-theory spectra of various stacks over $BG$ with the ideal $I$ given by the augmentation ideal of $G$ (as recalled in the next section). In particular, we only perform $I$-adic completions at $K(BG)$-module spectra. In contrast, completions via de Rham prestacks are performed on the geometric level of maps of \textit{derived} (as opposed to spectral) prestacks. There is still a certain compatibility between the two notions: we refer the reader to Remark~\ref{rem:compare}, and also Lemma~\ref{lem:affine completion vs derived} for a comparison statement that will be useful for our main results.
\end{rem}

\subsection{Atiyah-Segal ideal}\label{sec: AS-completion}
We now set the base ring to be a field, which we will denote by $k$. Let $G$ be a linearly reductive group over $k$.  Let $\mathrm{Rep}_k^{\mathrm{fg}}(G)$ denote the abelian category of finite-dimensional $G$-representations on $k$-vector spaces. We have a natural functor\footnote{Namely, $\mathrm{Rep}_k^{\mathrm{fg}}(G)$ is identified with the heart $\mathrm{Perf}(BG)^\heartsuit$ of the natural $t$-structure on $\mathrm{Perf}(BG)$.} $\mathrm{Rep}_k^{\mathrm{fg}}(G)\rightarrow \mathrm{Perf}(BG)$, and the category $\mathrm{Perf}(BG)$ splits as the direct sum
\begin{equation}\label{eq:description_perfbg}
\Perf(BG)\simeq \bigoplus_{V_\lambda\in \mathrm{Irr}(G)}\Perf(k)\otimes V_\lambda
\end{equation}
with $V_\lambda$ given by isomorphism classes of irreducible $G$-representations. This way the ring $K_0(BG)$ is naturally identified with the representation ring $\mc R_k(G)$ of (virtual) finite-dimensional  $G$-representations valued in $k$-vector spaces; if no confusion arises we will write $\mathrm{Rep}^{\mathrm{fg}}(G) = \mathrm{Rep}_k^{\mathrm{fg}}(G)$ and $\mc R(G) = \mc R_k(G)$. The map $\Spec k\rightarrow BG$ induces a ring homomorphism $K_0(BG)\rightarrow K_0(k)$ which is naturally identified with the map $\dim\colon \mc R(G) \ra \mbb Z$ sending the class $[V]\in \mc R(G)$ to $\dim V\in \mbb Z$ for any finite-dimensional $G$-representation $V$. 

\begin{defn}
	The {\bf Atiyah-Segal ideal} (or {\bf AS-ideal} for short) $I_G\subset K_0(BG)$ is defined as the kernel of this map. Via the isomorphism $K_0(BG)\simeq \mc R(G)$ it is identified with  the augmentation ideal $\ker(\dim)\subset \mc R(G)$.  
	\end{defn}

\begin{rem}
	Note that $\mc R(G)$ is a finitely generated ring over $\mbb Z$ and that, consequently, $I_G$ is a finitely generated $\mc R(G)$-ideal. Indeed, for the first assertion note that $G$ fits into a short exact sequence 
\[
1\ra G_0 \ra G \ra G/G_0\ra 1
\]
with $G_0$ being the connected component of the unit $e\in G(k)$ and $G/G_0$ finite \'etale. Let $\overline{k}$ denote the algebraic closure of $k$ and put $G_{\overline{k}}:= G \times_k \overline{k}$. By \cite{Tits1971}, irreducible representations of $G_0$ are encoded by $\Gamma:=\mrm{Gal}(\overline{k}/k)$-orbits in the character lattice $X^*(T_{\overline{k}})$ of a maximal torus $T_{\overline{k}}\subset (G_0)_{\overline k}$, and as generators of $\mc R(G_0)$ one can take the classes $[V_{\Gamma\cdot \omega_i}]$ of irreducible representations corresponding to Galois orbits of the fundamental weights. To get from $G_0$ to $G$ one can take all direct summands of the inductions $\mrm{Ind}_{G_0}^G(V_{\Gamma\cdot \omega_i})$: they form a finite set of tensor generators of $\mathrm{Rep}^{\mathrm{fg}}(G)$, and thus their classes in $\mc R(G)$ generate it as a ring.
\end{rem}

To get a sense of how the representation ring and the Atiyah-Segal ideal looks like, we offer the following examples. 

\begin{exam}\label{exam:as-ideal} \begin{enumerate}
		\item Let $G=T$ be a split algebraic torus. Let $X^*(T):=\Hom_{\mathrm{gr-Sch}}(T,\mbb G_m)$ denote the character lattice. The representation ring of $\mc R(T)$ is given by the group ring $\mbb Z[X^*(T)]$. In particular, $K_0(BT)$ is naturally isomorphic to the group ring $\mbb Z[X^*(T)]$.  The Atiyah-Segal ideal $I_T\subset \mbb Z[X^*(T)]$ is generated by elements $[\chi] -1$ for $\chi\in X^*(T)$. 
	\item Let $G$ be a connected split reductive group. Let $T\subset G$ be a maximal torus. The representation ring $\mc R(G)$ is identified with the subring $\mbb Z[X^*(T)]^W\subset \mbb Z[X^*(T)]\simeq \mc R(T)$ under the natural restriction map $\mc R(G)\rightarrow \mc R(T)$. Here $W$ denotes the Weyl group of $G$. This way Atiyah-Siegal ideal $I_G\subset \mc R(G)$ is described as the intersection $I_T\cap \mbb Z[X^*(T)]^W\subset \mbb Z[X^*(T)]^W$. 
	\item Let $G=GL_n$. This is a particular case of (2), where we can identify $W\simeq\Sigma_n$ and $\mbb Z[X^*(T)]\simeq \mbb Z[t_1^{\pm 1},\ldots, t_n^{\pm 1}]$, with $\Sigma_n$ acting by permutations of $t_i$. It is well-known that the $\Sigma_n$-invariants $\mbb Z[t_1^{\pm 1},\ldots, t_n^{\pm 1}]^{\Sigma_n}$ are given by the subring $\mbb Z[e_1,\ldots,e_{n-1},e_n, e_n^{-1}]\subset \mbb Z[t_1^{\pm 1},\ldots, t_n^{\pm 1}]$ where $e_i\in \mbb Z[t_1,\ldots, t_n]$ are elementary symmetric polynomials in $t_i$. The homomorphism $\dim\colon \mc R(GL_n)\ra \mbb Z$ is induced by the map sending $t_i$ to 1, or, consequently, $e_i$ to $\binom{n}{i}$. We see that the AS-ideal $I_{GL_n}$ is generated by $(e_i-\binom{n}{i})_{1\le i\le n}$.
	
	\end{enumerate}
\end{exam}	


\ssec{The Atiyah-Segal completion and change of groups} 

Let $\sX\ra BG$ be a derived prestack over $BG$. Let $E$ be a localizing invariant; then $E(\sX):=E(\Perf(\sX))$ is naturally a module over $K(BG)$.
\begin{defn}\label{defn:as_completion}
	The {\bf Atiyah-Segal completion} of $E(\sX)$ is the $I_G$-completion ${E(\sX)}^\wedge_{I_G}$ of $E(\sX)$ as a $K(BG)$-module (in the sense of Section \ref{sec:sag-completion}). 
\end{defn}

Assume we have an embedding $H\hookrightarrow G$ of reductive groups over $k$. This induces a map of classifying stacks $BH\ra BG$ and a functor $\mrm{dPStk}/BH \ra \mrm{dPStk}/BG$ given by post-composition. 
In this case, together with the Atiyah-Segal ideal $I_H\subset K_0(BH)$ we have the image $I_G\cdot K_0(BH)$ of  $I_G\subset K_0(BG)$ under the induced map $K_0(BG)\ra K_0(BH)$. It will be very useful to know that the completions of $E(\sX)$ for a prestack $\sX \ra BH \ra BG$ with respect to these two ideals coincide. 
\lem{\label{lem:comparing Atiyah-Segal completions} In the situation above, there exists $n\gg 0$ such that
	$$
	I_H^n\subset I_G\cdot K_0(BH)\subset I_H
	$$
Consequently, for any derived prestack $\sX\ra BH$ and finitary localizing invariant $E$ there is an equivalence of completions
$$
E(\sX)^\wedge_{I_H}\simeq E(\sX)^{\wedge}_{I_G\cdot K_0(BH)}
$$
as $K(BH)$-modules.}

\begin{proof}
	Identifying $K_0(BH)\simeq \mc R(H)$ and $K_0(BG)\simeq \mc R(G)$ this follows e.g. from 
	\cite[Corollary 6.1]{Edidin-Graham}. The second part follows since by the above the radicals of $I_H$ and $I_G\cdot K_0(BH)$ coincide.
\end{proof}	

\rem{We note that by the general properties of derived completions, the underlying $K(BG)$-module of $E(\sX)^{\wedge}_{I_G\cdot K_0(BH)}$ is equivalent to the Atiyah-Segal completion $E(\sX)^{\wedge}_{I_G}$ (with $E(\sX)$ considered as a $K(BG)$-module).} In particular, we get an equivalence
$$
E(\sX)^\wedge_{I_H}\simeq E(\sX)^\wedge_{I_G}
$$
of $K(BG)$-modules.

\begin{lem}\label{lem:affine schemes are AS-complete}For any map $S=\Spec R \ra BG$ and a finitary localizing invariant $E$, the $K(BG)$-module $E(S)$ is $I_G$-complete.
\end{lem}	
\begin{proof}
	Pick some embedding $G\hookrightarrow GL_n$. Since any vector bundle is Zariski-locally trivial, there is a Zariski cover $S'\ra S$ such that the composition
	$$
	S'\ra S \ra BG \ra BGL_n
	$$
	factors through the cover $\Spec k\ra BGL_n$. Let $S'_\bullet$ denote the \v Cech nerve of $S'\ra S$. The image of $I_{GL_n}$ in $E_0(S'_\bullet)$ is 0, and so by Lemma \ref{lem:comparing Atiyah-Segal completions} it follows that image of $I_G^n$ in $E_0(S'_\bullet)$ is 0 for some $n$. Consequently, $E(S'_\bullet)$ is $I_G$-complete. Recall that
	$E\colon \mrm{dAff}^{\mrm{op}} \ra \mrm{Sp}$  is a Nisnevich, and in particular Zariski, sheaf and so
	$$
	E(S)\simeq \lim_{[\bullet]\in \Delta} E(S'_\bullet).
	$$
	Thus $E(S)$ is a limit of $I_G$-complete $K(BG)$-modules, and as such is $I_G$-complete.
\end{proof}


As a corollary, we obtain the right Kan extended theory is always complete. 

\begin{cor}\label{cor:Kan extension is I_G-complete}
	For any derived prestack $\sX$ over $BG$ and finitary localizing invariant $E$, the $K(BG)$-module 
	$
	R^{\mrm{dAff}}E(\sX)
	$ is $I_G$-complete. 
\end{cor}	
\begin{proof}
We have $R^{\mathrm{dAff}}E(\sX)\simeq \lim_{\Spec R\ra \sX} E(\Perf(R))$, and by Lemma \ref{lem:affine schemes are AS-complete} all terms in the limit are $I_G$-complete. As a limit of $I_G$-complete modules, $R^{\mathrm{dAff}}E(\sX)$ is $I_G$-complete.
\end{proof}

By Corollary \ref{cor:Kan extension is I_G-complete} the target of the map $E(\sX)\ra R^{\mrm{dAff}}E(\sX)$ is $I_G$-complete and thus it factors through the $I_G$-completion:
\begin{equation}\label{eq:asmap}
E(\sX)^\wedge_{I_G} \ra R^{\mrm{dAff}}E(\sX).
\end{equation}
\begin{defn}[Atiyah-Segal completion] Let $E$ be a $k$-linear localization invariant, $\sX$ is a derived prestack over $k$ and $\sX \rightarrow BG$ be a morphism. We say that the \textbf{AS completion theorem holds} for $E, \sX, G$ if the map~\eqref{eq:asmap} is an equivalence.
\end{defn}

The rest of the paper will be dedicated in proving that, under certain conditions on $E$, $\sX$ and $G$, the AS completion theorem holds. We will also offer negative results in Section~\ref{sec:counterexample} which involves $E = K$ on certain non-reduced stack.

\subsection{AS ideal in $K_0(BG)_{\mbb Q}$}

We finish off this section with a discussion of the AS ideal in rationalized $K$-theory; this will be relevant to extend AS completion results to singular stacks. For this section, we \emph{assume that $k$ is a field of characteristic zero}. Let $G$ denote a reductive group scheme over $k$. We have a short exact sequence 
$$
1\ra G^0 \ra G\ra \pi_0(G)\ra 1
$$
where $G_0$ is the connected component of $e\in G(k)$ and  $\pi_0(G)$ is the ``group of connected components"; since we are in characteristic zero $\pi_0(G)$ is \'etale. The fiber square 
$$
\begin{tikzcd}
	BG^0 \arrow[r]\arrow[d]&BG \arrow[d]\\
	\mrm{pt} \arrow[r]& B\pi_0(G)
\end{tikzcd}
$$
induces a map
$$
r\colon \sR(G)\otimes_{\sR(\pi_0(G))} \mbb Z \ra \sR(G^0)
$$
(where the tensor product is classical).

\begin{lem}
	The induced map
	$$
	r_{\mbb Q}\colon \sR(G)\otimes_{\sR(\pi_0(G))} \mbb Q \ra \sR(G^0)_{\mbb Q}
	$$
	admits a natural retraction. Moreover,  the ideal generated by $r_{\mbb Q}(I_G\otimes 1)$ has the same radical as $I_{G^0}$.
	If, moreover, $\pi_0(G)$ is a constant group scheme over $k$ then the left hand side naturally identifies with the (classical) invariants 
	$$
	\sR(G^0)_{\mbb Q}^{\pi_0(G)} \subset \sR(G^0)_{\mbb Q}.
	$$
\end{lem}	

The following Galois descent statement for $\mrm{D}({\mbb Q})$-valued localizing invariants will be useful.

\begin{lem}
	Let $E\colon  \Cat^{\perf}_{k} \rightarrow \mrm{D}({\mbb Q})$ be a localizing invariant. Let $\ell/k$ be a Galois extension with the Galois group $\Gamma_{\ell/k}$. Then  
	$$
	\pi_i(E(\mc C))\simeq \pi_i(E(\sC\otimes_{\Perf({k})} \Perf({\ell})))^{\Gamma_{\ell/k}}.
	$$
\end{lem}	
\begin{proof}
	The pull-back $f^*\colon \Perf(k)\ra \Perf(\ell)$ has a natural right adjoint functor $f_*\colon \Perf({\ell})\ra \Perf({k})$. They induce functors $g^*\colon \mc C \ra \sC\otimes_{\Perf({k})} \Perf({\ell})$ and $g_* \colon \sC\otimes_{\Perf({k})} \Perf({\ell}) \ra \mc C$. The composition $g_*g^*$ is given by tensoring $-\otimes_k \ell$ which induces multiplication by $\dim({\ell/k})$ on $E(\sC)$ which is an isomorphism. This splits $E(\sC)$ off the \v Cech nerve $E(\mc C\otimes_{\Perf({k})} \Perf({\ell})^{\otimes \bullet})$ which evaluates the $\Gamma_{\ell/k}$-fixed points of $E(\sC\otimes_{\Perf({k})} \Perf({\ell}))$ considered as a spectrum with $\Gamma_{\ell/k}$-action. The result for homotopy groups then follows since rationally taking $\Gamma$-invariants is $t$-exact.
\end{proof}	

\begin{exam}
	Let $\sX$ be a perfect stack over $k$. Then (e.g. by \cite[Theorem 1.2(1)]{BFN}) $\Perf(\sX)\otimes_{\Perf(k)} \Perf(\ell)\simeq \Perf(\sX_{\ell})$ with $\sX_{\ell}:=\sX\times_k \ell$. Thus we get that for any $\mrm{D}({\mbb Q})$-valued localizing invariant $E\colon  \Cat^{\perf}_{k} \rightarrow \mrm{D}({\mbb Q})$ we have 
	$$
	\pi_i(E(\sX))\simeq \pi_i(E(\sX_{\ell}))^{\Gamma_{\ell/k}}.
	$$
	In particular, taking $E=K_{\mbb Q}$ and $\sX=BG$ for $G$ reductive we get an isomorphism of rings
	$$
	K_0(BG)_{\mbb Q} \simeq K_0(BG_{\ell})_{\mbb Q}^{\Gamma_{\ell/k}}. 
	$$
One also gets that $I_{G,\mbb Q}:=I_G\otimes_{\mbb Z} \mbb Q\simeq I_{G_{\ell},\mbb Q}^{\Gamma_{\ell/k}}$.
\end{exam}	

Let $I_{G,\mbb Q}:=I_G\otimes_{\mbb Z} \mbb Q \subset K_0(BG)_{\mbb Q}$; it is clear that $I_{G,\mbb Q}$ maps inside $I_{G_{\ell},\mbb Q}$ under the map $K_0(BG)_{\mbb Q} \ra K_0(BG_{\ell})_{\mbb Q}$. Even though it doesn't necessarily generate it, it is true that their radicals coincide:
\begin{cor}\label{cor:AS-ideal for base change}
	For any reductive $k$-group $G$ and $n\gg 0$ we have 
	$$
	{(I_{G_{\ell},\mbb Q})}^n \subset I_G\cdot K_0(BG_{\ell})_{\mbb Q}
	$$
\end{cor}
\begin{proof} 
	This follows from Hilbert's theorem in invariant theory. 
\end{proof}	

\section{Atiyah-Segal completion for $KH$ and other $\mathbb{A}^1$-invariant localizing invariants}\label{sec:askh}

In this section, we extend the results of \cite[Theorem 1.5]{tabuada2020motivic} to a larger class of stacks, at least when the base field is assumed to be characteristic zero; this restriction comes from our use of resolution of singularities. As we already mentioned in the introduction, Theorems~\ref{thm:main}~and~\ref{thm:main_2} have 
been proved for $X$ a smooth projective variety; more generally it is proved in \cite[Corollary 1.3]{tabuada2020motivic} for stacks of the form $[X/G]$ where $X$ is a smooth, separated $k$-scheme, $G$ a reductive group with a $k$-split maximal torus $T$ such that the action of $T$ on $X$ is filtrable. 


%

\subsection{Reminders on equivariant resolution of singularities} To proceed, we recall some results on functorial, equivariant resolution of singularities, mostly following \cite{kollar-resolution}.  

\begin{defn}\cite[Section 3.1]{kollar-resolution} \label{def:resolution}
Let $Y$ be a reduced, separated finite type $k$-scheme. Then a morphism of $k$-varieties
\[
X \stackrel{\pi}\to Y
\]
is said to be a {\bf resolution of singularities} if 
\begin{enumerate}
\item
$\pi$ is projective birational morphism,
\item 
$\pi$ is an isomorphism over 
the smooth locus of $Y$;
\item 
$X$ is smooth over $k$.
\end{enumerate}
We say that $\pi$ is a \textbf{strong resolution of singularities} if moreover the preimage of the singular locus 
$\pi^{-1}(\mrm{Sing} Y)$ is a simple normal crossings divisor.
\end{defn}

The existence of a strong resolution  of singularities which is ``functorial enough'' to be accommodate equivariance is recorded below. 

\begin{thm}\label{thm:resolution}
Let $k$ be a field of characteristic zero and let $X$ be a reduced separated scheme of finite type over $k$, endowed with the action of an algebraic group $G$. 
Then there exists a $G$-equivariant map $\tilde{X} \to X$ which is a strong resolution of 
singularities. 
\end{thm}
\begin{proof}
The existence of $\tilde{X}$ follows from \cite[Theorem~3.36]{kollar-resolution} (see also \cite[Corollary~3.43]{kollar-resolution}). Now \cite[Proposition~3.9.1]{kollar-resolution} yields that the action of $G$ 
on $X$ lifts to the action on its functorial resolution $\tilde{X}$. 
\end{proof}

We can interpret Theorem~\ref{thm:resolution} in terms of abstract blow-up squares of stacks as in Definition~\ref{defn:cdh} as follows. If $X$ is a reduced, separated scheme of finite type over $k$, equipped with a $G$-action and $Z \hookrightarrow X$ is a $G$-invariant subscheme of $X$ containing the singular locus of $\tilde{X}$, then there exists an abstract blow-up square

\begin{equation}\label{eq:resolv-sq}
	\begin{tikzcd}
		\left[ W/G \right] \arrow[d,"g"] \arrow[r,"i'"]& \left[ \tilde{X}/G \right]\arrow[d,"f"] \\
		\left[ Z/G\right] \arrow[r,"i"]&\left[ X/G \right]
	\end{tikzcd}	
\end{equation}

The following result of the the first and third authors \cite[Corollary 5.2.6]{dgm-elden-vova} will play a key role for us:

\begin{prop}\label{prop:cdh} Let $E$ be a truncating invariant and consider the square as in~\eqref{eq:resolv-sq}. Further assume that $G$ is reductive. Then the square
\begin{equation}
\begin{tikzcd}
E\left( \left[ X/G \right]\right) \arrow[d,"g"] \arrow[r,"i'"]& E\left(\left[ \tilde{X}/G \right] \right)\arrow[d,"f"] \\
E\left( \left[ Z/G\right]\right) \arrow[r,"i"]& E\left( \left[ W/G \right]\right),
\end{tikzcd}
\end{equation}
is cartesian. 
\end{prop}

\begin{lem}\label{lem:localization}
Let $E$ be an $\mathbb{A}^1$-invariant and truncating\footnote{Although many homotopy invariant theories 
are in practice nilinvariant and truncating, it is not expected to be true in general.} 
localizing invariant, let $f:U \to X$ be a $T$-equivariant open immersion of smooth algebraic spaces such that 
the complement $Z = X - U$ is a normal crossings divisor, where $T$ is a split torus. 
Then there exists a fiber sequence 
\[
E([Z/T]) \to E([X/T]) \stackrel{f^*}\to E([U/T]) 
\]
\end{lem}
\begin{proof}
First of all, observe that since $Z$ is a divisor in $X$, the pushforward along the inclusion 
defines a map $\Perf_{[Z/T]} \to \Perf_{[X/T]}$ and induces the first map in the fiber sequence. 

By definition, $Z$ is a union of $n$ smooth divisors $\bigcup_{i=1}^n D_i$ 
that have smooth scheme-theoretic intersections of expected codimensions. 
Note that since $T$ is connected, $T\cdot D_i = D_i$ (it contains $D_i$ and has as many irreducible components $D_i$). 
%
We prove the desired claim using induction on $n$. 
If $n=1$, then the result is the usual localization sequence \cite[Theorem~6.1]{tabuada2020motivic}. 
Now set $Z_1 = D_1$ and $Z_2 = \bigcup_{i=1}^{n-1} D_i$. Denote by $Z_3$ 
the intersection $Z_1 \cap Z_2$. 

From base change formula for quasi-coherent sheaves (and transversality of intersections) we have a commutative diagram
\[
\begin{tikzcd}
E([Z_1/T]) \arrow[r, "="]\arrow[d] & E([Z_1/T]) \arrow[r]\arrow[d] & 0 \arrow[d]\\
E([Z/T]) \arrow[r]\arrow[d] & E([X/T]) \arrow[r]\arrow[d] & E([X\setminus Z/T]) \arrow[d, "="]\\
E([Z \setminus Z_1/T]) \arrow[r] & E([X\setminus Z_1/T]) \arrow[r] & E([X\setminus Z/T]).
\end{tikzcd}
\]
By induction, we know that the bottom row is a fiber sequence, as well as the middle column. To 
show that the middle row is a fiber sequence, it suffices to show that the left column is a fiber sequence. 
As above, we will use induction on the amount $n$ of smooth components of $Z$. If $n=1$ then 
$Z=Z_1$ and the claim is clear. In general, using 
blow-up excision \ref{prop:cdh} for the square 
\[
\begin{tikzcd}
{[Z_3/T]} \arrow[r]\arrow[d] & {[Z_1/T]}\arrow[d]\\
{[Z_2/T]} \arrow[r] & {[Z/T]}
\end{tikzcd}
\]
we may rewrite the sequence as a pullback of sequences:
\[
\begin{tikzcd}
E([Z_1/T]) \arrow[r,"="]\arrow[d] &E([Z_1/T]) \arrow[r]\arrow[d] &E(\emptyset)\arrow[d]\\
E([Z_3/T]) \arrow[r,"="] &E([Z_3/T]) \arrow[r] &E(\emptyset)\\
E([Z_3/T]) \arrow[r]\arrow[u] &E([Z_2/T]) \arrow[r]\arrow[u] &E([Z_2-Z_1/T])\arrow[u].
\end{tikzcd}
\]
The top and the middle sequences are fiber sequences since $E(\emptyset) = 0$. The bottom sequence is a fiber sequence 
by induction.
\end{proof}

In the case of schemes, any Zariski sheaf which is $\mathbb{A}^1$-invariant is automatically invariant for torsors under a vector bundle. This is a bit more subtle in case of stacks. 

\begin{lem}\label{lem:affine_bundle_invariance}
Let $E$ be an $\mathbb{A}^1$-invariant and truncating localizing invariant. 
Let $\sX$ be an ANS derived algebraic stack of finite type and let $p:\sV \to \sX$ be an affine bundle (i.e. a torsor over a 
vector bundle). 
Then the pullback map 
\[
E(\sX) \to E(\sV)
\]
is an equivalence.
\end{lem}
\begin{proof}
By Theorem~\ref{thm:niscover} any $\sX$ admits an affine Nisnevich cover by stacks of the form $[\Spec R/G]$. So it suffices to prove the claim for quotient stacks $[X/G]$ where $X$ is 
separated of finite type and $G$ is nice. 
By Theorem~\ref{thm:resolution} any such quotient stack $[X/G]$ admits an abstract blow-up square 
\[
\begin{tikzcd}
	\left[ W/G \right] \arrow[d,"g"] \arrow[r,"i'"]& \left[ \tilde{X}/G \right]\arrow[d,"f"] \\
	\left[ Z/G\right] \arrow[r,"i"]&\left[ X/G \right]
\end{tikzcd}
\]
with $Z$ and $W$ of smaller dimension and $\tilde{X}$ smooth. By \cite[Remark 3.13, Lemma 2.17]{hoyois-sixops}, the result is true for smooth quotients; Hoyois' observation asserts that any $\mathbb{A}^1$-invariant Nisnevich sheaf on $\Sm_{BG}$ is invariant for affine bundles. So the general claim follows by induction on dimension of $X$ and Proposition~\ref{prop:cdh}. 
\end{proof}

\subsection{Proof of Theorem~\ref{thm:main_2}} We are now ready to prove Theorem~\ref{thm:main_2}. For the reader's convenience we recall how the theorem is stated:

\begin{thm}\label{thm:main_2_text}  Fix $G$ a reductive group over a characteristic zero field $k$ and $E$ a truncating localizing invariant which is also $\mathbb{A}^1$-invariant. Let $\sX \rightarrow BG$ be representable morphism where $\sX$.  Assume that $\sX$ is ANS and $\sX^{\mrm{cl}} \rightarrow BG$ is finite type. Then we have an equivalence:
\begin{equation}\label{eq:main}
E(\sX)^{\wedge}_{I_G} \xrightarrow{\simeq} R^{\mrm{dAff}}E(\sX).
\end{equation}
\end{thm}

\begin{proof}  To start, we note that having a representable morphism $\sX \rightarrow BG$ whose restriction to the classical locus of the domain is finite type implies that $\sX = [X/G]$ where $X$ is a derived algebraic space, whose classical locus is finite type over $k$. We now prove the results in a few steps. 

\begin{enumerate}

\item[(Step 1)] {\bf First, we reduce to the case that $X$ is classical and reduced}. Because $E$ is truncating, it is nilinvariant on quotient stacks where the group is embeddable  \cite[Proposition 5.1.10]{dgm-elden-vova}. The same comment applies for the right Kan extended theory by Lemma~\ref{lem:localizing-invt} since it is cdh excisive on algebraic stacks. Hence we may assume $X$ is classical and reduced. 

\item[(Step 2)] {\bf We first reduce our claim to the case that $G = T$ a $k$-split torus.} 
First we note that for any linear algebraic group $G$, there is an embedding $G \to \mrm{GL}_n$. In particular, 
we have an isomorphism of stacks $[X/G] \simeq [X \times_G \mrm{GL}_n / \mrm{GL}_n]$. Denote $X \times_G \mrm{GL}_n$ by $X'$.

For a Borel subgroup $B \subset \mrm{GL}_n$ consider the induced map 
$p: [X'/B] \rightarrow [X'/GL_n]$, we claim that the the induction and restriction functors induce a retraction diagram
\[
\Perf_{[X'/\mrm{GL}_n]} \xrightarrow{p^*} \Perf_{[X'/B]} \xrightarrow{p_*}  \Perf_{[X'/\mrm{GL}_n]}.
\]
Indeed, by the projection formula it suffices to prove that the canonical map 
\[
\sO_{[X'/GL_n]} \rightarrow p_{*}\sO_{[X'/B]}
\]
is an equivalence. It suffices to check that this map is an equivalence after base change along 
$f:X' \rightarrow [X'/G]$. To check this, consider the Cartesian square
\[
\begin{tikzcd}
  X' \times G/B \ar{r}{p'} \ar[swap]{d}{f'} & X' \ar{d}{f}\\
\left[X'/B\right] \ar{r}{p} &  \left[X'/G\right].
\end{tikzcd}
\]
By base change, we get an equivalence $f^*p_{*}\sO_{[X'/B]} \simeq p'_*f^{'*}\sO_{[X'/B]} \simeq \sO_{X'} \otimes_k p_{G/B*}\sO_{G/B}$, where $p_{G/B}: G/B \rightarrow \Spec k$ is the canonical map. Now Kempf's theorem \cite[Proposition 4.6(a)]{jantzen} asserts that $p_{G/B*}\sO_{G/B} \simeq \sO_k$, thus we conclude that $f^*p_{*}\sO_{[X'/B]} \simeq \sO_{X'}$.

We claim that this gives us a diagram of retractions:
\begin{equation}\label{eq:retract}
\begin{tikzcd}
E([X'/\mrm{GL}_n])^{\wedge}_{I_{\mrm{GL}_n}} \arrow[d]\arrow[r]& E([X'/B])^{\wedge}_{I_{\mrm{GL}_n}}\arrow[d]\arrow[r] & E([X'/\mrm{GL}_n])^{\wedge}_{I_{\mrm{GL}_n}}\arrow[d]\\
R^{\mrm{dAff}}E([X'/\mrm{GL}_n]) \arrow[r]&R^{\mrm{dAff}}E([X'/B]) \arrow[r] & R^{\mrm{dAff}}E([X'/\mrm{GL}_n]).
\end{tikzcd}
\end{equation}
That the top row is a retraction is immediate from the previous paragraph. For the bottom row: choose a faithful representation of $GL_n$ (concretely, we may choose the space of $n \times n$-matrices $M_n$), $\sE$. For each $\nu \geq 1$, we set $U_{\nu}$ to be the open of $\bbV(\sE)^{\oplus \nu}$ where the diagonal action of $GL_n$ is free; note that the quotient stack $[U_{\nu}/GL_n]$ is schematic and so we simply write it as $U_{\nu}/GL_n$. We can also consider $U_{\nu}$ as a $B$-scheme with a free $B$-action and the same comment applies to the quotient stack $[U_{\nu}/B]$. We consider the diagram, functorial in $\nu$:
\begin{equation}\label{eq:approx}
\begin{tikzcd}
E([X'/\mrm{GL}_n] \times_{BGL_n} U_{\nu}/GL_n) \arrow[r]& E([X'/B] \times_{BB} U_{\nu}/B) \arrow[r] & E([X'/\mrm{GL}_n] \times_{BGL_n} U_{\nu}/GL_n).
\end{tikzcd}
\end{equation}
The composite is an equivalence from the previous paragraph, hence the limit of composites is an equivalence. But now a result of Khan and Ravi \cite[Theorem 3.6]{khan-ravi-2}, coupled with \cite[Theorem 12.27]{khan2021generalized} to ensure that their ``lisse-extension'' coincides with our right Kan extension shows that the limit of~\eqref{eq:approx} is the bottom row of~\eqref{eq:retract}.

Now, observe that the map $[X'/T] \to [X'/B]$ is an affine bundle and $[X'/B]$ is ANS of finite type since it is 
affine over $[X'/GL_n] = [X/G]$. By Lemma~\ref{lem:affine_bundle_invariance} 
we have an equivalence $E([X'/B]) \simeq E([X'/T])$ and by Lemma~\ref{lem:comparing Atiyah-Segal completions} we have the horizontal equivalences in the commutative diagram
\[
\begin{tikzcd}
E([X/B])^{\wedge}_{I_{\mrm{GL}_n}} \arrow[d]\arrow[r,"\simeq"]& E([X/T])^{\wedge}_{I_{\mrm{GL}_n}}\arrow[d]\\
R^{\mrm{dAff}}E([X/B]) \arrow[r,"\simeq"] &R^{\mrm{dAff}}E([X/T]).
\end{tikzcd}
\]
Therefore, combining this observation with Lemma~\ref{lem:comparing Atiyah-Segal completions} we reduce to showing the claim for $G=T$ a split torus. 

\item[(Step 3)] {\bf Assume that the completion theorem holds for all smooth algebraic spaces of dimension $\leq d$ then it holds for any quasi-separated finite type algebraic space of dimension $d$.}

Let $X$ be a quasi-separated algebraic space of dimension $d$. 
Consider the comparison map as a map on the category of \'etale morphisms over $X$
\[
E([-/T])^{\wedge}_{I_T} \rightarrow R^{\mrm{dAff}}E([-/T]): \left(\Et^{BT}_{X}\right)^{\op} \rightarrow \Spt.
\]
By Nisnevich descent and Theorem~\ref{thm:cover} (this uses the quasi-separated assumption on $X$), 
we may assume that $X$ is separated.

By Theorem~\ref{thm:resolution} a separated algebraic space $X$ admits an equivariant resolution of singularities map 
$\tilde{X} \to X$. This yields an abstract blow-up square 
\[
\begin{tikzcd}
{[\tilde{Z}/T]} \arrow[r]\arrow[d] & {[\tilde{X}/T]}\arrow[d]\\
{[Z/T]}\arrow[r] & {[X/T]},
\end{tikzcd}
\]
where $Z$ is the subscheme of singularities and $\tilde{Z}$ is the exceptional divisor. 
Using the abstract blow-up excision for $E$ (\ref{prop:cdh}) and the fact that completion commutes 
with finite limits we reduce to proving the completion theorem for $[\tilde{Z}/T]$, $[Z/T]$ and $[\tilde{X}/T]$. The 
dimensions of $Z$ and $\tilde{Z}$ are less than $d$ and $\tilde{X}$ is smooth, so in all of these cases the completion 
theorem is known by assumption. 

\item[(Step 4)] {\bf Assume that the completion theorem holds for all quasi-separated, finite type algebraic spaces of 
dimension $\leq d-1$ and that it holds for all smooth projective schemes of dimension $d$, then the completion theorem holds 
for any quasi-separated algebraic space of dimension $d$.}

By Step 3, it suffices to show the claim for $X$ a smooth algebraic space of dimension $d$. Consider the comparison map as a 
map on the category of $T$-equivariant \'etale morphisms over $X$
\[
E([-/T])^{\wedge}_{I_T} \rightarrow R^{\mrm{dAff}}E([-/T]): \left(\Et^{BT}_{X}\right)^{\op} \rightarrow \Spt.
\]
By Nisnevich descent and Theorem~\ref{thm:cover}, 
it suffices to prove the result for $X = \Spec A$, a smooth affine scheme with a $T$-action.

Consider a finite dimensional subspace $V \subset A$ that is invariant under the action 
of $T$ and generates $A$ as a $k$-algebra. This exists because any finite dimensional subspace of a representation is contained in a finite 
dimensional representation of $T$. Therefore, we have a $T$-equivariant closed immersion 
\[
\iota: X \to \mathbf{P}(V)
\]
The closure of the image of $\iota$ is a $T$-equivariant 
projective variety. Resolving its singularities Theorem~\ref{thm:resolution} we obtain an equivariant embedding 
$X\to \bar{X}$ where $\bar{X}$ is a smooth projective variety. The complement of $Z=\bar{X} \setminus X$ is a simple normal 
crossings divisor. 
Using Lemma~\ref{lem:localization} we reduce the completion theorem to the case of $Z$ which has smaller dimension and the 
case of $\bar{X}$ which is a smooth projective, finite type $k$-scheme. 
Both cases are handled by our assumptions.

\item[(Step 5)] {\bf Verify the cases for $d=0$ and for $X$ a smooth projective $k$-scheme.} This then follows from the main theorem of \cite{tabuada2020motivic}; in particular we use \cite[Theroem 1.5]{tabuada2020motivic} and appeal to \cite{bb-groups, hesselink} for the fact that any smooth projective $k$-scheme with a $T$-action are $T$-filtrable.
\end{enumerate}
\end{proof}

\begin{rem}\label{rem:invariants} As pointed out in Examples~\ref{rem:examples} and~\ref{rem:exam2}, the above result holds 
for $KH, L_{K(1)}K$ and $HP$ (for $\mathbb{Q}$-linear categories). Since our arguments here bootstrap from the main results 
of \cite{tabuada2020motivic}, they equally well  apply to the examples discussed in the introduction there. 
\end{rem}

\begin{rem}[The ANS hypothesis]\label{rem:ans} Our proof reduces the completion theorem for stacks of the form $[X/G]$ to 
stacks of the form $[X/T]$. What allows this to happen is Lemma~\ref{lem:comparing Atiyah-Segal completions} on how 
completions compare along immersions of groups as well as affine bundle invariance. The latter is known only to hold for ANS 
stacks by Lemma~\ref{lem:affine_bundle_invariance}. This hypothesis will appear again when we patch together Theorem~\ref{thm:main_2_text} with the Hochschild homology results of the next section. The authors do not have a counterexample for a more general affine bundle invariance statement and this is the only bottleneck to generalizing the AS completion theorem for $\mathbb{A}^1$-invariant, localizing invariants. 
\end{rem}
\section{Atiyah-Segal completion theorem for Hochschild invariants}\label{sec:as-hh}

The goal of this section is to show that in case $E$ is either $HH$, $HC^-$, $HC$, or $HP$ the comparison map 
\[
E([X/G])^\wedge_{I_G} \ra R^{\mrm{dAff}}E([X/G])
\]
is an equivalence if $\sX=[X/G]$ is ANS and satisfies some further (very mild) assumptions. Our results generalize those of Chen's \cite{Chen_2020}, where the completion theorem is formulated 
in the language of {\it loop stacks} and {\it formal completions} thereof. 

Let $\sX$ be a perfect geometric stack over $\mbb Q$, then $HH(\Perf(\sX))$ is naturally identified with derived functions on the loop stack $\sL_{\mbb Q} \sX:= \underline{\sM\mrm{aps}}(\underline{S^1},\sX)$, defined as the mapping stack from $S^1$ to $\sX$. As already explained in \cite{Chen_2020}, there are at least two natural completions of $\sL_{\mbb Q} \sX$ that one can look at. On the one hand, we can take completion along the map $\sX \rightarrow \mathcal \sL_{\mbb Q}\sX$ embedding $\sX$ into $\sL \sX$ as constant loops. On the other hand, we can consider the unipotent loops $\sL^u_{\mbb Q} \sX:= \underline{\sM\mrm{aps}}(B\mathbb{G}_a,\sX)$ which itself looks as a completion of $\sL_{\mbb Q}\sX$ at a (in general different) closed substack of $\sL_{\mbb Q} \sX$. Functions on each of the completions produce variants of completed Hochschild homology. In case that $\sX = [X/G]$, there is yet a third option: one can complete $HH(\sX)$ along the Atiyah-Segal ideal. 

For Hochschild homology, the main results of this section clarifies these various completions.
\begin{enumerate}
\item in large generality, functions the completion along constant loops always computes the right Kan extended theory; Proposition~\ref{prop:Ran as functions on formal loops}.
\item in the case that $\sX = [X/G]$, the Atiyah-Segal completion agrees with functions on the unipotent completion; Proposition~\ref{prop:ig-uni}, Remark~\ref{rem:relation to Chen}.
\end{enumerate}

The Atiyah-Segal completion theorem for HH then holds when the completions along unipotent and constant loop agree, as happens for example if $G$ is nice (Theorem~\ref{thm:completion-th-}) or, under some nice assumptions, if $[X/G]$ is ANS (see Theorem~\ref{thm:general completion thm for HH}). We note (see Remark~\ref{rem:no AS in general}) that these completions \emph{do not agree in general}, if we do not assume that $G$ is nice. 

\subsection{Reminders and complements on formal completions}\label{sec:reminders-completion}

In this section we recall the notion of formal completion along a map of derived prestacks. In applications, the prestacks we consider will \emph{not} necessarily be of locally of finite presentation over the base; in this context various definitions of de Rham prestack (and its relative version which by definition is the formal completion) all differ. The original version of de Rham stack (which we will call ``thin de Rham stack") was introduced by Simpson and is probably better known. However, without finite presentation assumptions, as we are forced to in this paper, it is better to work with Lurie's \textit{absolute de Rham stack} (\cite[Definition 18.2.1.1]{SAG}).


\begin{defn}\label{defn:de Rham prestack and completion}
	Let $\sX\in \mrm{dPStk}$ be a derived prestack. 
	\begin{enumerate}
		
		\item The {\bf thin de Rham prestack} $\sX_{\widecurlywedge{\mrm{dR}}}\colon \CAlg^{\mrm{an}}_{\mbb Z} \ra \Spc$  is defined via formula
		$$
		\sX_{\widecurlywedge{\mrm{dR}}}(R):= \sX(\pi_0(R)^\mrm{red}).
		$$
		This is the standard Simpson's definition of the de Rham prestack.
		
	The (Lurie's) {\bf de Rham prestack} $\sX_{{\mrm{dR}}}\colon \CAlg^{\mrm{an}}_{\mbb Z}\ra \Spc$ (called  the ``absolute de Rham prestack" in \cite[Definition 18.2.1.1]{SAG}) is defined by assigning to any $R\in \CAlg^{\mrm{an}}_{\mbb Z}$ the space
		$$
		\sX_{{\mrm{dR}}}(R):= \colim_{I}\sX(\pi_0(R)/I),
		$$ 
		where the colimit runs over the filtered system of ideals $I\subset \pi_0(R)$ that are nilpotent (meaning $I^n=0$ for some $n$). The natural maps $R\ra R/I\ra \pi_0(R)^{\mrm{red}}$ induce maps
		$$\sX \ra \sX_{\mrm{dR}}\ra \sX_{\widecurlywedge{\mrm{dR}}}.$$
		\item Let $f\colon \sX \ra \sY$ be a map of prestacks. The {\bf thin derived formal completion} of $\sY$ along $f$ is a prestack defined as the fiber product
		$$
		{\sY}_{f}^\curlywedge:= \sY\underset{\sY_{\widecurlywedge{\mrm{dR}}}}{\times} \sX_{\widecurlywedge{\mrm{dR}}}.
		$$
		By construction, for $R\in \CAlg^{\mrm{an}}_{\mbb Z}$ we have
		$$
		{\sY}^\curlywedge_{f}(R)\simeq \sY(R)\underset{\sY(\pi_0(R)^{\mrm{red}})}\times \sX(\pi_0(R)^{\mrm{red}}).
		$$

		\noindent Similarly, we consider the (absolute) {\bf derived formal completion}, which we define as
		$$
		{\sY}_{f}^\wedge:= \sY\underset{\sY_{{\mrm{dR}}}}{\times} \sX_{{\mrm{dR}}}.
		$$
		By construction, $R$-points of ${\sY}^\wedge_{f}$ are given by 
		$$
		{\sY}^\wedge_{f}(R)\simeq \colim_{I\subset \pi_0(R)}\left(\sY(R)\underset{\sY(\pi_0(R)/I)}\times \sX(\pi_0(R)/I)\right).
		$$
		where $I$ runs over nilpotent ideals of $\pi_0(R)$. We have a natural commutative diagram of maps
		$$
	\begin{tikzcd}
		&& {\sY}^\wedge_{f}\arrow[rrd]\arrow[dd]&&\\
			\sX \arrow[rru] \arrow[rrd]& && & \sY\\
	&& {\sY}^\curlywedge_{f}\arrow[rru]&&
	\end{tikzcd}
		$$
		where both upper and lower compositions are canonically equivalent to the original map $f\colon \sX\ra \sY$.
	\end{enumerate}	
\end{defn}	

\begin{rem} The formalism of the two de Rham prestacks provides a robust geometric, functor-of-points approach towards formal completions. Among other advantages, note that it makes sense to complete along morphisms and not just immersions. To the authors' knowledge this viewpoint was first offered by Rozenblyum in his thesis \cite[Definition 2.4.5]{nick-thesis}. 
\end{rem}

\begin{rem}[Completion of a pull-back]\label{rem:pull-back of a formal completion}
	Given a commutative square
	$$
	\begin{tikzcd}
		\sW \arrow[r,"g"]\arrow[d] & \sZ\arrow[d]\\
		\sX \arrow[r,"f"] & \sY
	\end{tikzcd}
	$$
	in $\mrm{PStk}$ such that the corresponding square of classical stacks
	$$
	\begin{tikzcd}
		\sW^{\mrm{cl}} \arrow[r,"g"]\arrow[d] & \sZ^{\mrm{cl}}\arrow[d]\\
		\sX^{\mrm{cl}} \arrow[r,"f"] & \sY^{\mrm{cl}}
	\end{tikzcd}
	$$
	is a fiber square, there are natural equivalences
	$$
	\sZ^\curlywedge_{g}\simeq \sZ \times_{\sY}\sY^\curlywedge_{f}, \qquad \sZ^\wedge_{g}\simeq \sZ \times_{\sY}\sY^\wedge_{f}.
	$$
	Similarly, given $f\colon \sX\ra \sY$ and $g\colon \sZ \ra \sW$ we have natural equivalences 
	$$
	(\sY\times \sW)^\curlywedge_{f\times g}\simeq \sY^\curlywedge_f\times \sW^\curlywedge_g, \qquad (\sY\times \sW)^\wedge_{f\times g}\simeq \sY^\wedge_f\times \sW^\wedge_g.
	$$
\end{rem}

\begin{rem}[Other properties of derived formal completions] \label{rem:properties of completion}
	\begin{enumerate}
		
		\item \label{rem:laftnes and convergence of formal completion} Since the $R$-points of prestacks $\sX_{\widecurlywedge{\mrm{dR}}}$ and $\sX_{{\mrm{dR}}}$ depend only on $\pi_0(R)$, these prestacks are convergent for any $\sX$. Consequently, both  $\sY^\curlywedge_f$ and $\sY^\wedge_f$ are convergent if $\sY$ is.

		\item 	\label{rem:dependence of the formal completion} Thin formal completion ${\sY}^\curlywedge_{f}$ only depends on $\sY$ and the map 
		$$f^{\mrm{red}}\colon \sX^{\mrm{red}} \ra \sY^{\mrm{red}}$$ between the corresponding reduced prestacks (meaning prestacks on reduced rings). Moreover, if $f^{\mrm{red}}$ is an embedding\footnote{Meaning that it $(-1)$-truncated, or in other words, for any $R\in \CAlg^{\mrm{cn}}$ the map of spaces  $\sX(R)\ra \sY(R)$ is fully faithfull as a map of $\infty$-groupoids (meaning in turn that it is the embedding of a subset of connected components).} (e.g. open or closed embedding), then so is the map $\sY_{f}^\curlywedge \ra \sY$. In this case the data of a lifting of $y\in \sY(R)$ to a point in $\sY_{f}^\curlywedge(R)\subset \sY(R)$ is not a ``structure" but a property: namely, $y$ belongs to $\sY_{f}^\curlywedge(R)$ if and only if the restriction $y|_{\Spec (\pi_0(R)^{\mrm{red}})}$ lands in $\sX(\pi_0(R)^{\mrm{red}})\subset  \sY(\pi_0(R)^{\mrm{red}})$.
		
		\item \label{item:lurie's completion} Similarly, the absolute formal completion $\sY^\wedge_f$ only depends on $\sY$ and the corresponding map of classical stacks $$f^{\mrm{cl}}\colon \sX^{\mrm{cl}}\ra \sY^{\mrm{cl}}.$$ 
		If $f^{\mrm{cl}}$ is an embedding, then so is the map $\sY_{f}^\wedge \ra \sY$. In this case an $R$-point $y\in \sY(R)$ lies in $\sY_{f}^\wedge(R)\subset \sY(R)$ if for some nilpotent ideal $I\subset \pi_0(R)$ the restriction $y|_{\Spec (\pi_0(R)/I)}$ lands in $\sX(\pi_0(R)/I)\subset  \sY(\pi_0(R)/I)$.
		
		\item If $\sX$ is a derived stack, then so is $\sX_{\widecurlywedge{\mrm{dR}}}$. Indeed, to check that $\sX_{\widecurlywedge{\mrm{dR}}}\colon \CAlg^{\mrm{cn}}\ra \mrm{Spc}$  has \'etale descent it is enough to check that it commutes with products and that given an \'etale cover $R\ra R'$ the natural map
		$$
		\sX(\pi_0(R)^\mrm{red})\simeq \sX_{\widecurlywedge{\mrm{dR}}}(R)\ra \lim_{[\bullet]\in \Delta }\sX_{\widecurlywedge{\mrm{dR}}}(R'^{\otimes \bullet})\simeq \sX_{\widecurlywedge{\mrm{dR}}}(\pi_0(R'^{\otimes \bullet})^\mrm{red}) 
		$$
		is an equivalence. 
		This follows form the fact that $R\mapsto \pi_0(R)^{\red}$ commutes with products, and that the cosimplicial object $\pi_0(R'^{\otimes \bullet})^\mrm{red}$ is identified with the \v Cech object for the map $\pi_0(R)^\mrm{red}\ra \pi_0(R')^\mrm{red}$.
		
		More generally, if we have a map of derived stacks $f\colon \sX \ra \sY$ then the thin formal completion $\sY^\curlywedge_f$ is also a derived stack.
		
		\item \label{item:completion of the colimit} Let $f\colon \sX \ra \sY$ be a map of derived stacks and assume we have a presentation $$\sY = L_\et(\colim_{i\in I}\sY_i)$$ as a colimit over some diagram of derived stacks $I\ra \mrm{PStk}$. Let $\sX_i:=\sY_i\times_\sY \sX$ and consider $f_i\colon \sX_i \ra \sY_i$. Then the natural map 
		$$
		L_\et(\colim_{i\in I}(\sY_i)^\curlywedge_{f_i}) \ra \sY^\curlywedge_f 
		$$
		is an equivalence. Indeed, we have
		$$
		\sY^\curlywedge_f \simeq \sY^\curlywedge_f\times_{\sY}L_\et(\colim_{i\in I}\sY_i)\simeq L_\et( \sY^\curlywedge_f\times_\sY \colim_{i\in I}\sY_i)\simeq L_\et (\colim_{i\in I}(\sY^\curlywedge_f\times_\sY \sY_i)) \simeq  L_\et (\colim_{i\in I}(\sY_i)^\curlywedge_{f_i})
		$$
		which gives the inverse identification. Here we used that $L_\et$ commutes with fiber products and that colimits in prestacks are universal. 
	\end{enumerate}
\end{rem}

\medskip

The following example illustrates the potential difference between $\sY^\wedge_f$ and $\sY^\curlywedge_f$.

\begin{exam}\label{ex:affine completion} Let $f\colon \sX\ra \sY$ be a closed embedding of derived affine schemes (so $\sY:=\Spec A$ and $\sX:=\Spec B$ and $f^*\colon A\ra B$ is surjective on $\pi_0$). Let $I\subset \pi_0(A)$ be the kernel of $\pi_0(f^*)\colon \pi_0(A)\twoheadrightarrow \pi_0(B)$. 
	
	Following Remark \ref{rem:properties of completion}, both $\sY^\curlywedge_f(R)$ and $\sY^\wedge_f(R)$ are embedded in $\sY(R)=\mrm{Map}_\CAlg(A,R)$; more precisely, $\sY^\curlywedge_f(R)$ is given by the subspace of maps $A\ra R$ that send $I\subset \pi_0(A)$ to the nil-radical  $\mrm{Nil}(\pi_0(R))\subset \pi_0(R)$ while $\sY^\wedge_f(R)$ is the space of those maps $A\ra R$ for which the image of $I\subset \pi_0(A)$ generates a \textit{nilpotent} ideal of $R$. These subspaces are the same if we assume that $I$ is finitely generated, but not necessarily otherwise.

\end{exam}	

\begin{exam}[Completion at finitely generated ideal]\label{exam:completion of affine thing at finitely generated ideal}
	
	Let $f$ be as above and assume $I \subset \pi_0(A)$ is a finitely generated ideal.  Let $a_1,\ldots,a_n\in I$ be a set of generators. In that case, by the argument in \cite[Proposition 6.7.4]{Gaitsgory2014} we have an explicit colimit presentation of the completion :
	\[
\mc Y^\wedge_f\simeq 	\colim_{i_1,\ldots,i_n\ge 0} \Spec ([A/a_1^{i_1}]\otimes_A\ldots \otimes_A [A/a_n^{i_n}]) \ra \Spec A,
	\]
	where $[A/a]$ denotes the pushout $A\otimes_{\mbb Z[x]}\mbb Z$, where $x$ maps to $a$ in $A$ and 0 in $\mbb Z$.
\end{exam}	

\begin{rem}[Comparison with Section~\ref{sec:sag-completion}] \label{rem:compare} On the level of functions there is compatibility between formal completion via de Rham prestacks and the completion discussed in Section~\ref{sec:sag-completion} via semiorthogonal decomposition of categories. More precisely, we have 
	$$
	\mc O(\mc Y^\wedge_f) \simeq \lim_{i_1,\ldots,i_n\ge 0} [A/a_1^{i_1}]\otimes_A\ldots \otimes_A [A/a_n^{i_n}],
	$$
 which is naturally identified with the derived $I$-adic completion $A^\wedge_I$ of $A$ as a module over itself (see Remark \ref{rem:formula for derived completion on module}).

\end{rem}

\begin{exam}[Comparison with classical completion] \label{ex:classical completion agrees with derived}	
	In the case when $A$ is classical and noetherian, the derived completion also agrees with the classical one: namely, there is a natural equivalence of prestacks
	$$
	\sY^\wedge_f \simeq \colim_n \Spec(A/I^n),
	$$
	see \cite[Proposition 6.8.2]{Gaitsgory2014}. In particular, for the embedding $\{0\}\hookrightarrow \mbb A^n_{\mbb Z}\simeq \Spec (\Sym_{\mbb Z}(\mbb Z^{\oplus m}))$ we get
	$$
	(\mbb A^m_{\mbb Z})^\wedge_{0}\simeq \colim_n \Spec (\Sym^{\le n}_{\mbb Z}(\mbb Z^{\oplus m}))
	$$
	where $\Sym_{\mbb Z}^{\le n}(\mbb Z^{\oplus m}):=\Sym_{\mbb Z}({\mbb Z}^{\oplus m})/I_0^{n+1}$ and $I_0$ is the ideal of $\{0\}\in \mbb A^m_{\mbb Z}$.
\end{exam}	

\sssec{Comparing ${\sY}^\wedge_{f}$ and ${\sY}^\curlywedge_{f}$}

The following notion will be useful to identify ${\sY}^\wedge_{f}$ with ${\sY}^\curlywedge_{f}$ in some examples: 
\begin{defn}\label{def:locally of finite presentation} A map of classical prestacks $f\colon \sX \ra \sY$ is called \textbf{locally of finite presentation} if for any filtered system $\{R_\alpha\}$ of classical rings with $R\simeq \colim_\alpha {R_\alpha}$ the natural map
	$$
	\colim_\alpha \sX(R_\alpha) \ra \sX(R)\times_{\sY(R)} (\colim_\alpha \sY(R_\alpha) )
	$$
	is an equivalence.
\end{defn}	

\begin{rem}\label{rem:lfp is enough to check over all points}
	If $\sY=\Spec R$ is affine and classical, then $f\colon \sX \ra \sY$ is locally of finite presentation if and only if for any filtered colimit $A=\colim_{\alpha} A_\alpha$ of classical $R$-algebras one has 
	$$
	\sX(A)=\colim_{\alpha} \sX(A_\alpha).
	$$
	Moreover, a general map of classical prestacks $f\colon \sX \ra \sY$ is locally of finite presentation if and only if all fibers $\sX_y:=\Spec R\times_\sY \sX \ra \Spec R$ over all points $y\colon \Spec R\ra \sY$ are. Indeed, assume the latter; we need to show that for any filtered colimit $R=\colim_\alpha R_\alpha$ the natural map
	$$
	\colim_\alpha \sX(R_\alpha) \ra \sX(R)\times_{\sY(R)} (\colim_\alpha \sY(R_\alpha) )
	$$
	is an equivalence. This is enough to do on all fibers over $\colim_\alpha \sY(R_\alpha)$ and then further over all points of $y_\alpha\in \sY(R_\alpha)$ for some $\alpha$. The map on the corresponding fibers is given by 
	$$
	\colim_{\alpha \ra \alpha'} \sX_{y_{\alpha}}(R_{\alpha'}) \rightarrow \sX_{y_{\alpha}}(R)
	$$
	where we have put $\sX_{y_{\alpha}}:=\Spec R_{\alpha} \times_\sY \sX$.
\end{rem}	

\begin{rem}[Comparison to \cite{SAG}] In \cite[17.4.1]{SAG}, a notion of locally of finite presentation is defined where one demands that natural map in Definition~\ref{def:locally of finite presentation} is an equivalence for a filtered system of derived rings. We caution the reader that these two definitions are not, in general, equivalent; our notion is clearly weaker and could perhaps be called ``locally of finite presentation on classical rings.'' Definition~\ref{def:locally of finite presentation} is the only one that will be used in this paper. Our weakening is motivated by comparison lemmas between different completions; see for example Lemma~\ref{lem:two completions coincide} where the absolute de Rham and thin completions are proved to be equivalent when the map in question has truncation which is locally of finite presentation as in Definition~\ref{def:locally of finite presentation}.
\end{rem}

The following lemma allows to control maps of finite presentation after \'etale sheafification.
\begin{lem}\label{lem:etale sheafification preserves locally finite presentation}
	Let $f\colon \sX\ra \sY$ be a map of classical prestacks which is locally of finite presentation and assume that there exists $n\ge -2$ such that for all classical algebras $R$ the fibers of the map $\sX(R)\ra \sY(R)$ are $n$-truncated. Then the induced map 
	$$
	L_\mrm{et}(f)\colon L_\mrm{et}(\sX) \ra L_\mrm{et}(\sY)
	$$
	is also locally of finite presentation.
\end{lem}	

\begin{proof}
	The same argument as in \cite[Lemma 17.4.3.2]{SAG} works here (with the simplification that we only consider classical algebras as the source category for $\sX$ and $\sY$).
\end{proof}	
\begin{exam} \label{ex:maps of lfp} Let us also give some examples of maps locally of finite presentation:
	
	\begin{enumerate}	\item
		If $f\colon \sX\ra \sY$ is a finitely presented affine morphism (meaning that for any $R$-point of $\sY$ the classical affine scheme $\tau^{\mrm{cl}}(\Spec A\times_{\sY}\sX)$ is given by the spectrum of a finitely presented $A$-algebra $B$) then it is locally of finite presentation in the sense of Definition \ref{def:locally of finite presentation}. Indeed, this follows from Remark \ref{rem:lfp is enough to check over all points}. 
		\item Any colimit $\sX:=\colim \sX_\alpha \ra \sY$ of maps $\sX_\alpha\ra \sY$ locally of finite presentation is again locally of finite presentation (this uses the fact that colimits in $\mrm{Spc}$ are universal).
		\item Recall that any classical $n$-Artin stack $\sX$ over a classical ring $R$ is obtained as the \'etale sheafification of an $n$-skeletal simplicial diagram $X_\bullet\colon \Delta^\op \ra \mrm{PStk}$ where each $X_i$ is a union of affine schemes (and the boundary maps $X_i\ra X_j$ are smooth). If each $X_i$ is a disjoint union of affine schemes that are actually \textit{finitely presented} over $R$ then by \ref{lem:etale sheafification preserves locally finite presentation} and above two examples we get that $\sX \ra \Spec R$ is locally of finite presentation. 
		Note that since $X_\bullet$ is $n$-skeletal, $\sX=L_\mrm{et}(|X_\bullet|)$ as a functor on $\CAlg^{\mrm{cl}}$ takes values in $n$-truncated spaces. Thus, more generally, having a map $f\colon \sX \ra \sY$ of classical prestacks which is representable in $n$-Artin stacks that are locally of finite presentation, it follows from Remark \ref{rem:lfp is enough to check over all points} that $f$ is locally of finite presentation. 
		
		In particular, if $G$  is a finitely presented smooth affine group scheme over a classical ring $R$ then $BG$ or, more generally, $[X/G]$ for an algebraic space $X$ of locally finite presentation over $R$ is again of locally finite presentation over $R$.
	\end{enumerate}
\end{exam}	
Let us now return to the comparison of thin and absolute formal completions.
\begin{lem}\label{lem:two completions coincide}
	Let $f\colon \sX \ra \sY$ be a map of derived prestacks such that the underlying map $\sX^{\mrm{cl}}\ra \sY^{\mrm{cl}}$ is locally of finite presentation. Then the natural map 
	$$
	{\sY}^\wedge_{f} \ra {\sY}^\curlywedge_{f}
	$$
	is an equivalence.
\end{lem}	
\begin{proof}
	Let $A\in \CAlg^{\mrm{an}}_{\mbb Z}$ be an animated commutative ring. Note that one has an equivalence (of animated rings) $$\pi_0(A)^{\mrm{red}}\simeq \colim_{I\subset \pi_0(A)} \pi_0(A)/I,$$
	where the (homotopy) colimit runs over all nilpotent ideals of $\pi_0(A)$. Since $\tau^{\mrm{cl}}(f)\colon \tau^{\mrm{cl}}(\sX)\ra \tau^{\mrm{cl}}(\sY)$  is of finite presentation and both versions of de Rham prestack of a prestack $\sX$ only depend on $\tau^{\mrm{cl}}(\sX)$, we get that the commutative square
	$$
	\begin{tikzcd}\sX_{\mrm{dR}}(A) \arrow[r]\arrow[d]& \sX_{\widecurlywedge{\mrm{dR}}}(A)\arrow[d]\\
		\sY_{\mrm{dR}}(A) \arrow[r]& \sY_{\widecurlywedge{\mrm{dR}}}(A)
	\end{tikzcd}
	$$
	is a fiber square for any $A$. From this it  follows that the map 
	$$
	{\sY}^\wedge_{f}:=\mc Y\times_{\sY_{\mrm{dR}} }\sX_{\mrm{dR}} \ra \mc Y\times_{\sY_{\widecurlywedge{\mrm{dR}}}}\sX_{\widecurlywedge{\mrm{dR}}}=:{\sY}^\curlywedge_{f}
	$$
	is an equivalence.
\end{proof}

We finish with two more computations of formal completions which will be needed later. 
First, let us note that Example~\ref{ex:affine completion} descends to a computation of the completion of quotient stacks.

\begin{cor}\label{cor:completion of quotients}
	Let $G$ be a classical $k$-algebraic group. Let $f\colon X \ra Y$ be a $G$-equivariant closed embedding of derived affine 
	$k$-schemes such that $X^{\mrm{cl}}\ra Y^{\mrm{cl}}$ is finitely presented. Then the completion $Y^{\wedge}_f$ admits a natural $G$-action and for 
	the corresponding map of quotient stacks $f/G\colon [X/G] \ra [Y/G]$ there is a natural equivalence 
	\[
	[Y/G]^\wedge_{f/G}\simeq [Y^{\wedge}_f/G].
	\]
\end{cor}
\begin{proof} Note that since the map  $X^{\mrm{cl}}\ra Y^{\mrm{cl}}$ is finitely presented, by Lemma \ref{lem:two completions coincide} we have that $Y^{\wedge}_f\simeq Y^{\curlywedge}_f$. Moreover, by Lemma~\ref{lem:etale sheafification preserves locally finite presentation} and Example~\ref{ex:maps of lfp}(3) we have that $[X/G]^{\mrm{cl}} \ra [Y/G]^{\mrm{cl}}$ is also locally finitely presented, so $[Y/G]^\wedge_{f/G}\simeq [Y/G]^\curlywedge_{f/G}$. Thus it is enough to prove the statement for thin formal completions. Here, by \cite[Lemma~B.1.8]{BKRS} we have an equivalence 
	\[
	[Y/G]^\curlywedge_{f/G}\simeq {\colim_n}^{\mrm{Stk}}  [X(n)/G] 
	\]
	for a sequence of $G$-equivariant derived schemes $X(n)$ satisfying 
	\[
	\colim_n (X(n) \to X(n+1) \to \cdots ) \simeq Y^{\curlywedge}_f.
	\]
	The result follows since taking quotient stacks commutes with colimits (in stacks).
\end{proof}

\begin{lem}\label{lem:affine completion vs derived}
	Let $A\in \CAlg^{\mrm{cn}}$ and let $I\subset \pi_0(A)$ be a finitely generated ideal. Let $\Spf A^\wedge_I$ denote the formal completion of $\Spec A$ along the map $\Spec (\pi_0(A)/I) \ra \Spec A$. Let $\mc Y$ be a prestack with a map $\sY\ra \Spec A$ and denote by $\sY^\wedge_I$ the pullback
	\[
	\begin{tikzcd}
		\sY^\wedge_I \arrow[r]\arrow[d] & \sY\arrow[d]\\
		\Spf A^\wedge_I \arrow[r] & \Spec A.
	\end{tikzcd}
	\]
	Then there is an equivalence 
	\[
	\mc O(\sY^\wedge_I) \simeq \mc O(\sY)^\wedge_I
	\]
	where the right-hand is the derived $I$-completion as an $A$-module.
\end{lem}
\begin{proof} Fiber product $-\times_{\Spec A} \Spf A^\wedge_I$ commutes with all colimits since colimits in any $\infty$-topos, in particular $\mrm{PStk}$ or $\mrm{PStk}/\! \Spec A$, are universal. Thus the functor $\sY\mapsto \mc O(\sY^\wedge_I)$ sends colimits to limits; since derived $I$-completion preserves limits, the same is true for $\mc Y\mapsto \mc O(\sY)^\wedge_I$. Any prestack $\sY$ is identified with the colimit $\colim_{\Spec B\ra \sY} \Spec B$, thus by the above we can reduce to the case $\mc Y=\Spec B$, where we have canonical the identification  
	$\mc O(\sY^\wedge_I) \simeq B^\wedge_{I}$ (see Remark~\ref{rem:compare}).
\end{proof}	

\begin{rem}
	Note that in the context of Lemma \ref{lem:affine completion vs derived} we have an isomorphism $\sY^\wedge_I \simeq \sY^\curlywedge_I $, since in this case $\Spf A^\wedge_I \simeq \Spf A^\curlywedge_I $ by Lemma \ref{lem:two completions coincide}.
\end{rem}

\sssec{Completed vector bundles} Recall that for an animated $\mbb Q$-algebra $R$ we have an equivalence $\CAlg_{R}^{\mrm{an}}\simeq \CAlg(D(R)_{\ge 0})$, where $\CAlg$ denotes the category of $\mbb E_\infty$-algebras.

 Let $\sX$ be a prestack over $\mbb Q$. Consider the free $\mbb E_\infty$-algebra functor $\Sym_{\mc O_\sX}\colon \mrm{D}_{\mrm{qc}}(\sX) \ra \CAlg(\mrm{D}_{\mrm{qc}}(\sX))$; in the limit presentations $$\mrm{D}_{\mrm{qc}}(\sX)\simeq \lim_{\Spec R\ra \sX} D(R) \qquad \qquad \CAlg(\mrm{D}_{\mrm{qc}}(\sX))\simeq \lim_{\Spec R\ra \sX} \CAlg( D(R))$$ it is given by $\lim_{\Spec R\ra \sX} \Sym_R$ where $\Sym_R$ is the free $\mbb E_\infty$-algebra functor on $ D(R)$. Functor $\Sym_{\mc O_\sX}$ is naturally $\mbb N$-graded: $\Sym_{\mc O_\sX}\!(\mc F)\simeq \oplus_{i=0}^\infty \Sym_{\mc O_\sX}^i(M)\simeq \oplus_{i=0}^\infty (\underbrace{\mc F\otimes_{\mc O_{\sX}} \ldots \otimes_{\mc O_{\sX}} \mc F}_i)_{\Sigma_i}$ and promotes to a functor with values in commutative algebras in the $\mbb N$-graded category\footnote{With the monoidal structure given by Day's convolution: $(\mc F\otimes \mc G)_n=\otimes_{i+j=n} \mc F_i\otimes \mc G_j$.} $\mrm{D}_{\mrm{qc}}(\sX)^{\mbb N\mrm{-gr}}:= \mrm{Fun}(\mbb N, \mrm{D}_{\mrm{qc}}(\sX))$. Since the subcategory $\mrm{D}_{\mrm{qc}}(\sX)_{\ge 0}\subset \mrm{D}_{\mrm{qc}}(\sX)$ is closed under the monoidal structure, $\Sym_{\mc O_\sX}$ also restricts to a functor $$\mrm{D}_{\mrm{qc}}(\sX)_{\ge 0} \ra \CAlg(\mrm{D}_{\mrm{qc}}(\sX)^{\mbb N\mrm{-gr}}_{\ge 0}).$$
 
 Let $n\ge 0$ be a number and also consider the category 
 $$\mrm{D}_{\mrm{qc}}(\sX)^{[0,n]\mrm{-gr}}:=\mrm{Fun}([0,n], \mrm{D}_{\mrm{qc}}(\sX))$$ 
of quasi-coherent sheaves graded by numbers from $0$ to $n$. There is a restriction functor $$(-)^{\le n}\colon \mrm{D}_{\mrm{qc}}(\sX)^{\mbb N\mrm{-gr}}\ra \mrm{D}_{\mrm{qc}}(\sX)^{[0,n]\mrm{-gr}}$$ which has a fully faithful left and right adjoint $\iota_n\colon \mrm{D}_{\mrm{qc}}(\sX)^{[0,n]\mrm{-gr}} \subset \mrm{D}_{\mrm{qc}}(\sX)^{\mbb N\mrm{-gr}}$ (with $\iota_n(\mc F)_i = \mc F_i$ if $0\le i \le n$ and $\iota_n(\mc F)_i=0$ for $i>n$). In fact, $\mrm{D}_{\mrm{qc}}(\sX)^{[0,n]\mrm{-gr}}$ is a localization of $\mrm{D}_{\mrm{qc}}(\sX)^{\mbb N\mrm{-gr}}$ with respect to the class of morphisms $f$ such that $\mrm{fib}(f)_i=0$ for $i\le n$. The monoidal structure on $\mrm{D}_{\mrm{qc}}(\sX)^{\mbb N\mrm{-gr}}$ preserves that class and so induces a monoidal structure on $\mrm{D}_{\mrm{qc}}(\sX)^{[0,n]\mrm{-gr}}$, such that the functors $(-)^{\le n}$ and $\iota_n$ lift to a strong and lax symmetric monoidal functors correspondingly. In particular, they induce functors
$$
\iota_n\colon  \CAlg(\mrm{D}_{\mrm{qc}}(\sX)^{[0,n]\mrm{-gr}}) \leftrightarrows \CAlg(\mrm{D}_{\mrm{qc}}(\sX)^{\mbb N\mrm{-gr}}) \ \! :\! (-)^{\le n}
$$
with $(-)^{\le n}$ being the left adjoint to $\iota_n$. Functor $\iota_n$ is fully-faithful and its image is spanned by those commutative algebras in $\mrm{D}_{\mrm{qc}}(\sX)^{\mbb N\mrm{-gr}}$ whose underlying object lies in $\mrm{D}_{\mrm{qc}}(\sX)^{[0,n]\mrm{-gr}}$.

We define $\Sym^{\le n}_{\mc O_\sX}\colon \mrm{D}_{\mrm{qc}}(\sX) \ra \CAlg(\mrm{D}_{\mrm{qc}}(\sX)^{[0,n]\mrm{-gr}})$ as the composition $(-)^{\le n}\circ \Sym^\star_{\mc O_\sX}$. For a given $\sE\in  \mrm{D}_{\mrm{qc}}(\sX)$ one has:
\begin{equation}\label{eq:maps from graded truncation}
\mrm{Map}_{\CAlg(\mrm{D}_{\mrm{qc}}(\sX)^{\mbb N\mrm{-gr}})}(\Sym^\star_{\mc O_X}(\mc E), \iota_n(\sA))\simeq \mrm{Map}_{\CAlg(\mrm{D}_{\mrm{qc}}(\sX)^{[0,n]\mrm{-gr}})}(\Sym^{\le n}_{\mc O_\sX}(\sE), \sA)
\end{equation}
for any $\sA\in  \CAlg(\mrm{D}_{\mrm{qc}}(\sX)^{[0,n]\mrm{-gr}})$. The same discussion applies to $\mrm{D}_{\mrm{qc}}(\sX)$ replaced by $\mrm{D}_{\mrm{qc}}(\sX)_{\ge 0}$.

%

\begin{defn}Let $\sX$ be a prestack over $\mbb Q$ and let $\sE\in \mrm{D}_{\mrm{qc}}(\sX)_{\ge 0}$ be a connective quasi-coherent sheaf. We define the corresponding \textbf{(generalized) vector bundle} to be the relative spectrum
$$
\mbf{V}_{\sX}(\sE):=\Spec_{\sX}(\Sym_{\mc O_{\mc X}}\!(\sE)).
$$
For any $n\in \mbb N$ will also denote by 
$$
\mbf V_{\sX}^{\le n}(\sE):=\Spec_{\sX}(\Sym^{\le n}_{\mc O_{\mc X}}\!(\sE))
$$
the relative spectrum of the corresponding truncation. Note that $\mbf V_{\sX}^{\le 0}(\sE)\simeq \Spec_{\sO_{\sX}}(\mc O_{\sX})\simeq \sX$. 
\end{defn}
\begin{exam}
	In the case $\sE\in \Perf(\sX)$ one can view $\mbf{V}_{\sX}(\sE)$ as the total space of the dual perfect complex $\sE^\vee \in \Perf(\sX)$. In general we do not impose any dualizability assumptions on $\sE$ and thus call the corresponding vector bundles ``generalized".
\end{exam}	
\begin{rem}
By the definition of relative spectrum we have an affine map $$p\colon \mbf V_{\sX}(\sE)\ra \sX$$ whose fiber over a given point $(f\colon \Spec R \ra \sX)\in \sX(R)$ is naturally identified with the derived affine scheme $\Spec(\Sym_R(f^*\sE)))$. Similarly, we have an affine map $\mbf V^{\le n}_{\sX}(\sE)\ra \sX$  whose fiber over $f\colon \Spec R \ra \sX$ is naturally identified with $\Spec(\Sym_R^{\le n} (f^*\sE)))$. Commutative algebra maps $\Sym_{\mc O_{\mc X}}\!(\sE) \ra \Sym^{\le n}_{\mc O_{\mc X}}\!(\sE)$ induce closed embeddings 
$$
z_n\colon \mbf V^{\le n}_{\sX}(\sE) \ra \mbf V_{\sX}(\sE)
$$
of prestacks over $\sX$. The map $z_0\colon \sX =\mbf V^{\le 0}_{\sX}(\sE)\ra \mbf V_{\sX}(\sE)$ models the ``embedding of the zero section".
\end{rem}

\begin{defn}
	The \textbf{completed (generalized) vector bundle} $\widehat{\mbf V}_{\sX}(\sE)$ is the absolute formal completion $\mbf V_{\sX}(\sE)^\wedge_{z_0}$ of $\mbf V_{\sX}(\sE)$ along the zero section $z_0\colon \sX \ra \mbf V_{\sX}(\sE)$.
\end{defn}	

\begin{rem} 
	For any $f\colon \Spec A \ra \sX$ and $n\ge 0$, 
	$$\pi_0(\Sym_A^{\le n}(f^*\sE))\simeq \Sym_{\pi_0(A)}^{\le n} \pi_0(f^*\sE).$$
	We have that the kernel $I$ of the natural surjection 
	$$\pi_0(\Sym_A^{\le n}(f^*\sE)) \twoheadrightarrow \Sym_{\pi_0(A)}^{\le 0}\! \pi_0(f^*\sE)\simeq \pi_0(A)$$ is nilpotent (of order $n$). Consequently, given any $f\colon \Spec A \ra \mbf V_{\sX}^{\le n}(\sE)$, the preimage $f^{-1}(I)\subset A$ is also a nilpotent ideal. It follows that the maps $z_n\colon \mbf V_{\sX}^{\le n}(\sE) \ra \mbf V_{\sX}(\sE)$ all factor uniquely through the formal completion $\widehat{\mbf V}_{\sX}(\sE)$.
\end{rem}	

By the above remark, taking colimit over $n$, we get a natural map 
$$
\colim_n \mbf V^{\le n}_\sX(\sE) \ra \widehat{\mbf V}_{\sX}(\sE)
$$
of prestacks over $\sX$. Even though this map is not necessarily an isomorphism (see Remark \ref{rem:not a completion in general}) it is if we consider the source and target as functors on bounded rings:
\begin{prop}\label{prop:completion at 0 section} Let $\sX$ be a convergent prestack and let $\sE\in \mrm{D}_{\mrm{qc}}(\sX)_{\ge 0}$. Assume that for any $f\colon \Spec A \ra \sX$ the $\pi_0(A)$-module $\pi_0(f^*\sE)$ is finitely generated. Then the natural map
	\[
	{\colim_n}^{\mrm{conv}} \mbf V^{\le n}_\sX(\sE)\ra \widehat{\mbf V}_\sX(\sE)
	\]
	is an isomorphism, where $\colim^{\mrm{conv}}$ denotes the colimit in convergent prestacks.
\end{prop}

\begin{proof} 
See Appendix \ref{app:completion at 0 section}.
\end{proof}

\begin{rem}\label{rem:not a completion in general}
	It is \textit{not} necessarily true that 
	$$
	\colim_n \mbf V^{\le n}_\sX(\sE) \simeq \widehat{\mbf V}_\sX(\sE).
	$$ 
	For example, consider $\sX=\Spec \mbb Q$ and $\sE=k[2]\in \mrm{D}(\mbb Q)_{\ge 0}$. Then $\pi_*(\Sym_{\mbb Q}(\mbb Q[2]))\simeq \mbb Q [x]$ and $\pi_*(\Sym_{\mbb Q}^{\le n}(\mbb Q[2]))\simeq \mbb Q[x]/x^{n+1}$ with $\deg x=2$. Since $\pi_0(\Sym_{\mbb Q}( \mbb Q[2]))\simeq  \mbb Q$, the zero section map is a closed embedding which is an isomorphism on $\pi_0$, and thus (by Example \ref{ex:affine completion})
	$$
	\widehat{\mbf V}_\sX(\sE)\simeq \mbf V_\sX(\sE).
	$$
	The identity map $\Sym_{ \mbb Q}( \mbb Q[2])\xrightarrow{\mrm{id}} \Sym_{ \mbb Q}( \mbb Q[2])$ gives a $\Sym_{ \mbb Q}( \mbb Q[2])$-point of $\widehat{\mbf V}_\sX(\sE)$; however it doesn't factor through $\colim_n \mbf V_\sX^{\le n}(\sE)$ since for any algebra map $\Sym_{ \mbb Q}( \mbb Q[2]) \ra \Sym^{\le n}_{ \mbb Q}( \mbb Q[2])$ the image of the generator $x\in \pi_2(\Sym_{\mbb Q}(\mbb Q[2]))$ would necessarily be nilpotent, while its image under the identity map is clearly not.
\end{rem}

\subsection{Loop prestacks and their completions} We now discuss Hochschild homology under its geometric guise of loop prestacks. We will work towards the Atiyah-Segal theorem for Hochschild homology as formulated in Theorem~\ref{thm:completion-th-}. For this, it will be necessary to explain the relationship between various completions of the loop prestack. We note that even if our stacks are over a field $k$, for applications to $K$-theory we are still forced to consider Hochschild homology of $\Perf(\sX)$ as a $\mbb Q$-linear category.

\sssec{Hochschild homology as functions on loops}\label{sec:loops}
When $\sX$ is a perfect stack, the Hochschild homology $HH(\sX/\mbb Q):= HH({\Perf}(\sX)/\mbb Q)$ relative to $\mbb Q$ admits a 
geometric description as functions on another stack whose definition we will recall below. 

Let $\underline{S^1}$ denotes the constant prestack corresponding to the circle; this is an instance of Example \ref{ex:constant prestack}.
\begin{defn} Let $\sX\in \mrm{dPStk}_{\mbb Q}$ be a derived prestack.
	\begin{enumerate}
\item  The {\bf free loop prestack} $\sL_{\mbb Q}\sX$ is defined as the mapping prestack
		$$
		\mathcal L_{\mbb Q}\sX := \underline{\sM\mrm{aps}}(\underline{S^1},\sX).
		$$ 
		Since (see Example \ref{ex:constant prestack}) one has $$\underline{S^1}\simeq \Spec \mbb Q\!\!\!\!\!\!\!\underset{\Spec \mbb Q\sqcup \Spec \mbb Q}{\bigsqcup}\!\!\!\!\!\!\! \Spec \mbb Q\in \mrm{PStk}_{\mbb Q},$$ 
		the prestack $\mathcal L_{\mbb Q}\sX$  is explicitly described as the fiber product
		$$\mathcal L_{\mbb Q}\sX \simeq \sX\underset{\sX\times \sX}{\times} \sX$$
		with respect to the diagonal maps $\Delta_{\sX}\colon \sX \ra \sX\times \sX$.
		
	\item	Restriction along the projection map $S^1 \to \mrm{pt}$ induces an inclusion of \textbf{constant loops} inside all loops:
		\[
		c_\sX\colon \sX \to \sL_{\mbb Q} \sX.
		\]

		\end{enumerate}
\end{defn}

\begin{rem}\label{rem:properties of the loop stack}
	Since sheaves are closed under limits, if $\sX$ is a derived stack, then so is $\sL_{\mbb Q}\sX$. Similarly, since convergent prestacks are closed under finite limits, if $\sX$ is convergent, so is $\sL_{\mbb Q}\sX$. Finally, if $\sX$ is a geometric stack, then so is $\sL_{\mbb Q}\sX$ (see \cite[Lemma 2.1.12]{Chen_2020}).
\end{rem}

\begin{rem}\label{rem:relative loop stacks} If $\sX$ comes with a map to a prestack $\sS$, we can also consider the relative version of the free loop prestack:
	$$
	\sL_{\sS}\sX \simeq \sX\underset{\sX\times_{\sS} \sX}{\times} \sX,
	$$
	that is the mapping prestack from $S^1$ to $\sX$ in the category of prestacks over $\sS$. 
	One has a natural map $\sL_{\sS}\sX\ra \sL_{\mbb Q}\sX$ whose composition with $
	\sL_{\mbb Q}\sX \ra \sL_{\mbb Q} \sS
	$ factors naturally through $c_\sS\colon \sS \ra  \sL_{\mbb Q} \sS$, and it's not hard to see that in fact
	$$
	\sL_{\sS}\sX\simeq \sL_{\mbb Q} \sX \times _{\sL_{\mbb Q} \sS} \sS.
	$$
	When $\sS=\Spec R$ for some $R\in \CAlg^{\mrm{an}}_{\mbb Q}$ we will denote $\sL_{\sS}\sX$ by $\sL_R\sX$.
\end{rem}

Recall that a derived stack $\sX$ is called perfect if it has affine diagonal and the natural functor $\Ind(\Perf(\sX))\ra \mrm{D}_{\mrm{qc}}(\sX)$ is an equivalence. As explained in \cite[Example~2.2.20]{Chen_2020}, if $\sX$ is a perfect stack then global functions on $\sL_{\mbb Q}\sX$ compute Hochschild homology of $\sX$: 
\begin{prop}\label{prop:loops_hh}
For a perfect stack $\sX$ over $\mbb Q$ we have a natural equivalence 
\[
HH(\sX/\mbb Q) \simeq \mathcal{O}(\sL_{\mbb Q}\sX).
\]
\end{prop}

\begin{rem}
		In the case a perfect stack $\sX$ is equipped with a map to $\Spec R$ one also has an equivalence
	$$
	HH(\sX/R):=HH(\Perf(\sX)/R)\simeq \sO(\sL_R\sX).
	$$
	
\end{rem}

\sssec{Formal loops}
Our next goal is to give a similar geometric description for the right Kan extension of Hochschild homology. 
Following \cite[Section~6]{Ben_Zvi_2012} and \cite[Section~2.1]{Chen_2020} we recall the following definition.

\begin{defn}\label{defn:completed_loops}
Let $\sX\in \mrm{PStk}_{\mbb Q}$ be a prestack. 
The {\bf formal loop}  prestack $\widehat{\sL}_{\mbb Q}\sX$ is defined as the \textit{absolute} derived formal completion of $\sL_{\mbb Q}\sX$ along constant loops $c_\sX\colon \sX \ra \widehat{\sL}_{\mbb Q}\sX$ (in the sense of Definition~\ref{defn:de Rham prestack and completion}).

\end{defn}

\begin{rem}One can instead consider the thin completion $({\sL}_{\mbb Q}\sX)^{\curlywedge}_{c_\sX}$. However, for a geometric stack this does not make a difference:  in this case, by Lemma \ref{lem:two completions of loops are the same} the natural map $({\sL}_{\mbb Q}\sX)^{\curlywedge}_{c_\sX} \simeq  \widehat{\sL}_{\mbb Q}\sX$ is an isomorphism.
	\end{rem}
	

Let us discuss some basic examples of loop prestacks, as well as their completions.

\begin{exam}\label{ex:formal loops in some examples}  
	Let $S=\Spec R$ be a derived affine scheme over $\mbb Q$.
	
	\begin{enumerate}
		\item Let us denote $R\otimes R:=R\otimes_{\mbb Q}R$. By definition, we have $$\sL_{\mbb Q}S:= S\underset{S\times S}{\times} S \simeq \Spec(R\otimes_{R\otimes R}R).$$ 
The embedding of constant loops $c_S\colon S \hookrightarrow \sL_{\mbb Q} S$ corresponds to the multiplication map $$R\otimes_{R\otimes R}R \ra R$$ 
on the level of functions. Note that this map is an isomorphism on $\pi_0$; thus, via Example \ref{ex:affine completion} ${\sL }_{\mbb Q}S$ is already complete along $c_S$; in other words
$$
\widehat{\sL}_{\mbb Q} S\xrightarrow{\sim} \sL_{\mbb Q}S.
$$

\item Let $\sX\ra S$ be a prestack over $S$; let $\sL_R\sX$ be the corresponding relative loop stack. By the above, we have that the map $c_S^{\mrm{cl}}\colon S^{\mrm{cl}}\ra (\widehat{\sL}_{\mbb Q} S)^{\mrm{cl}}$ of the underlying classical stacks is an isomorphism. Following Remark \ref{rem:relative loop stacks}, we then get that the natural map of classical stacks
$$
(\sL_R \sX)^{\mrm{cl}} \ra (\sL_{\mbb Q} \sX)^{\mrm{cl}}
$$
is an isomorphism.

\end{enumerate}
\end{exam}
\begin{exam}[Loop stacks on $BG$]\label{ex:formal loops in some examples2}
Let $G$ be an algebraic group over a field $k\supset \mbb Q$, and let $\sX=BG:= [\Spec k/G]$ be the corresponding classifying stack. We would like to describe the relative loop stacks $\mc L_{k} (BG)$ and $\widehat{\mc L}_{k} (BG)$, as well as say something about $\mc L_{\mbb Q} (BG)$. Let $S:=\Spec k$. 

\begin{enumerate} \item First, we describe $\mc L_k (BG)$. By definition, we have 
$$
\mc L_{k} (BG):=BG\times_{BG\times_S BG}BG.
$$
Let us identify $BG$ with the quotient $[G/(G\times_{S} G)]$ by the action $(h_1,h_2)\circ g=h_1gh_2^{-1}$. This gives an identification
$$\mc L_k (BG)\simeq [(G\times_{S}G)/(G\times_{S} G)],$$
where the action on the right is given by $(h_1,h_2)\circ (g_1,g_2)= (h_1g_1h_2^{-1}, h_1g_2h_2^{-1})$. The quotient of $G\times_{k} G$ by the action of $\{1\}\times G\subset G\times_{S} G$ can be identified with $G$, on which the remaining action of $G\times \{1\}\subset G\times_{k} G$ is given by conjugation $h\circ g= hgh^{-1}$. This gives an identification
$$
\mc L_k (BG)\simeq [G/_{\! \mrm{ad}}G]
$$
as the quotient stack for the adjoint action of $G$ on itself.

 Via this identification the inclusion of constant loops $c_{\sX} \colon \sX\ra \mc L_k \sX$ is given by the embedding 
$$
[e/G]\colon BG \hookrightarrow [G/G]
$$
where $e\colon S\hookrightarrow G$ is the unit element. By Corollary \ref{cor:completion of quotients} the corresponding derived completion $\widehat{\mc L}_k \sX \ra \mc L_k\sX$ is given by the embedding 
$$
\widehat{\mc L}_k(BG)\simeq [G^\wedge_e/G] \hookrightarrow [G/G]\simeq \mc L_k(BG)
$$
of the quotient of the corresponding formal group by the conjugation of $G$.

\item \label{item:relative loops of BG}
By Example \ref{ex:formal loops in some examples}(2) we also get an isomorphism of classical stacks
$$
(\mc L_{\mbb Q} (BG))^{\mrm{cl}} \xrightarrow{\sim} (\mc L_{k} (BG))^{\mrm{cl}}\simeq [G/G];
$$
the map $c_{BG}\colon BG \ra \mc L_{\mbb Q} (BG)$ factors through $\tau^{\mrm{cl}}(\mc L_{\mbb Q} (BG))\simeq [G/G]$.
The formal loops $\widehat{\mc L}_{\mbb Q} (BG)$ can be described as the completion of $\mc L_{\mbb Q} (BG)$ along $BG\hookrightarrow [G/G]$ in its classical locus.
\end{enumerate}
\end{exam}


More generally, for $\sX=[X/G]$ the difference between $\widehat{\sL}_{\mbb Q}\sX$ and ${\sL}_{\mbb Q}\sX$ is controlled by the difference between $\widehat{\sL}_{\mbb Q}(BG)$ and ${\sL}_{\mbb Q}(BG)$:

\begin{prop}\label{prop:pull-back square for formal loops} Let $\mbb Q\subset k$ be a field.  Let $X$ be a derived algebraic space endowed with an action of an algebraic group $G$, both over $k$. 
	Then $\widehat{\sL}_{\mbb Q}([X/G])$ is identified with the fiber product
	$$
	\begin{tikzcd}
		\widehat{\sL}_{\mbb Q}([X/G])\arrow[r]\arrow[d]& {\sL}_{\mbb Q}([X/G])\arrow[d]\\
		\widehat{\mc L}_{\mbb Q}(BG)  \arrow[r] & \mc L_{\mbb Q}(BG).
	\end{tikzcd}	
	$$
	
\end{prop}	
\begin{proof}
	This follows from Lemma \ref{lem:completion passes along representable maps} applied to $\sX=[X/G]$ and $\sY=BG$.
\end{proof}

\sssec{Exponential map} 
Recall that for a derived algebraic stack $\sX$ over $\mbb Q$  the suspension 
 $\mbb L_{\sX/\mbb Q}[1]\in \mrm{D}_{\mrm{qc}}(\sX)$ is connective (\cite[Chapter 1, Proposition 7.4.2]{GRderalg2}). 
 In particular, we can consider the corresponding generalized vector bundle $\mbf{V}_{\sX}(\mbb L_{\sX/\mbb Q}[1])$.

\begin{thm}[{\cite[Theorem~1.23]{Ben_Zvi_2012}}]\label{thm:exponent}
Let $\sX$ be a derived geometric stack over  $\mbb Q$. 
Then there is a natural equivalence 
\[
\mrm{exp}_\sX\colon \widehat{\mbf{V}}_{\sX}(\mbb L_{\sX/\mbb Q}[1]) \simeq \widehat{\sL}_{\mbb Q}\sX.
\]
\end{thm}
\begin{proof}
	See Appendix \ref{thm:Nadler-Ben-Zvi}.
\end{proof}	

This leads to the following:
\begin{prop}\label{prop:Ran as functions on formal loops} 
Let $\sX$ be a geometric stack over $\mbb Q$. 
Then there is a natural equivalence
$$
R^{\mrm{\mrm{dAff}}} HH(\sX/\mbb Q)\simeq \mc O(\widehat{\sL}_{\mbb Q}\sX).
$$
\end{prop}
\begin{proof}
Let $\sX$ be an affine scheme. We have natural equivalences
$$
R^{\mrm{\mrm{dAff}}} HH(\sX/\mbb Q)\simeq HH(\sX/\mbb Q)\simeq  \mc O({\sL}_{\mbb Q}\sX)\simeq  \mc O(\widehat{\sL}_{\mbb Q}\sX)
$$
where second and third isomorphism are given by Proposition~\ref{prop:loops_hh} and Example~\ref{ex:formal loops in some examples} correspondingly. 
It thus suffices to show that the functor $\sX \mapsto \mc O(\widehat{\sL}_{\mbb Q}\sX)$ is right Kan extended from derived affine schemes.
By Theorem~\ref{thm:exponent} (and Remark \ref{rem:for algebraic stack pi_0 L_X shifted is finitely generated}) we know that $\mc O(\widehat{\sL}_{\mbb Q} \sX) \simeq \mc O (\widehat{\mbf V}_{\sX}(\mbb L_{\sX/\mbb Q}[1]))$ and by Lemma \ref{app:functions on formal completion at 0} 
we may compute 
\[
\mc O (\widehat{\mbf V}_{\sX}(\mbb L_{\sX/\mbb Q}[1])) \simeq 
\prod_{i=0}^\infty R\Gamma(\sX ,\Sym^i_{\mc O_{\mc X}}\!(\mbb L_{\sX/\mbb Q}[1])).
\]
The right-hand side now is right Kan extended from derived affine schemes by Proposition \ref{prop:etale-descen-cot}.
\end{proof}

\begin{rem}
	Following the construction of the isomorphism in Proposition \ref{prop:Ran as functions on formal loops}, in the case $\sX$ is a perfect stack we can identify the map 
	$$
	HH(\sX/\mbb Q) \ra R^{\mrm{\mrm{dAff}}} HH(\sX/\mbb Q)
	$$
	with the pull-back map $\mc O(\sL_{\mbb Q}\sX) \ra \sO(\widehat{\sL}_{\mbb Q}\sX)$ for the inclusion $\widehat{\sL}_{\mbb Q}\sX \ra \sL_{\mbb Q}\sX$.
\end{rem}

\sssec{AS-completion $HH$ and unipotent loops}\label{ssec:AS-completion of HH and unipotent loops}

We will now identify the AS-completion of Hochschild homology. Let $G$ be a reductive algebraic group over a field $k$. The stack $BG$ is both geometric and perfect. Following Proposition~\ref{prop:loops_hh} 
and Example~\ref{ex:formal loops in some examples2} we can identify 
\[
HH(BG/k)\simeq \mc O([G/G])\simeq \mc O(G)^G.
\]

The Dennis trace map $K_0(BG) \ra HH_0(BG/k)$, via the identifications $\mc R(G)\simeq K_0(BG)$ and 
$HH_0(BG)\simeq \mc O(G)^G$, is given by the map 
\[
\tr\colon \mc R(G) \ra \mc O(G)^G,
\]
sending a representation $V\in \mrm{Rep}^{\mrm{f.g.}}(G)$ to the corresponding trace function 
$\tr_V\colon g\mapsto \tr(g|V)\in \mc O(G)^G$. 

Let us denote by $J_{G}$ the kernel of the map $e^*\colon \mc O(G)^G\ra k$ given by evaluation at the unit element $e\in G(k)$. Note that its composition with $\tr$ is given by the dimension map $\dim\colon \mc R(G) \ra k$; in particular, $I_G$ maps to $J_G$ under $\tr$; this is recorded in the following diagram of exact sequences
\[
\begin{tikzcd}
0 \ar{d} & 0 \ar{d}\\
I_G \ar{r} \ar{d} & J_G \ar{d} \\
\mc R(G) \cong K_0(BG) \ar{r}{\mrm{tr}} \ar[swap]{d}{\mrm{dim}} & HH_0(BG) \cong \mc O(G)^G \ar{d}{e^*} \\
\mathbb{Z} \ar{r} \ar{d} & k \ar{d}\\
0 & 0. 
\end{tikzcd}
\]
It will be useful to know that completions at either of the ideals coincide. Often, the ideals coincide on the nose:

\begin{lem}\label{lem:G-reps vs G-invariant functions}  Let $G$ be a reductive group over a field $k$ of characteristic 0. \begin{enumerate}
		\item If $k$ is algebraically closed the map 
		$$
		\tr\colon \sR(G)\otimes_{\mbb Z}k \ra \mc O(G)^G
		$$
		is an isomorphism.
		\item The map $\tr$
is also an isomorphism if $G$ connected and split. 
\item Assume that $G^\circ$ is split and $\pi_0(G)$ is a constant group scheme. Then $\tr$ induces an isomorphism
$$
\tr(I_G) \cdot  \mc O(G)^G \simeq J_G.
$$
\end{enumerate}
\end{lem}
\begin{rem}
	As will soon be clear from the proof, in general case the main subtlety is in the potential existence of irreducible representations $V$ of $G$ such that $V_{\overline k}$ is not irreducible (as a $G_{\overline k}$-representation). In this case $\End(V)^G\not\simeq k\cdot \mrm{Id}_V$ and this is exactly what prevents $\tr$ from being surjective.
\end{rem}	
\begin{proof} (1) By (the algebraic version of) Peter-Weyl theorem $\mc O(G)$ decomposes $G$-equivariantly into the 
	direct sum 
	$$\mc O(G)\simeq \oplus_{V\in \mrm{Irr}(G)} \End(V)$$ with $G$-action on the right given by conjugation; here $A\in \End(V)$ defines a function $f_A(g) :=\tr(\rho_V(g)A)$. In particular, $\mrm{Id}_V\in \End(V)$ corresponds to the function $\tr_V$ (sending $g\mapsto \tr_V(\rho_V(g))$).
	We have $\mc O(G)^G\simeq \oplus_{V\in \mrm{Irr}(G)} \End(V)^G$  and, by Schur's lemma, we also have
	$\End(V)^G\simeq k\cdot \mrm{Id}_V$, which shows that $\mc O(G)^G$ is linearly spanned over $k$ by functions of the form $\tr_V$ (and thus $\tr$ is an isomorphism). 
	
	(2) Note that it would be enough to show that $\tr$ becomes an isomorphism after base change to $\overline k$. Identifying $\mc O(G)^G\otimes_k\overline k\simeq \mc O(G_{\overline k})^{G_{\overline k}}$ and using part (1) it then remains to show that the natural map $\sR(G)\ra \sR(G_{\overline k})$ is an isomorphism. Indeed, for a connected split reductive group $G$ in characteristic 0, all the irreducible representations are of the form $H^0(G/B, \mc O(\lambda))$ for a fixed choice $B\subset G$ of a Borel subgroup, $\lambda$ running through dominant weights of the maximal torus $T\subset G$ and $\mc O(\lambda)$ being the corresponding $G$-equivariant line bundle over $G/B$. By base change, one has $H^0((G/B)_{\overline k}, \mc O(\lambda))\simeq H^0(G/B, \mc O(\lambda))\otimes_k \overline k$, which shows that basis of irreducibles for $G$ in $\sR(G)$ maps identically to the basis of irreducibles for $G_{\overline k}$. 
	
	(3) Let us first consider the case $G=\pi_0(G)$; so $G$ here is a finite, constant group. Let $k[G]$ denote the regular representation; the element $[k[G]]-|G|\in \sR(G)$ lies in $I_G$. Note that $\sO(G)^{G}\simeq k^{\mrm{Conj}(G)}$ is the ring of $k$-valued functions on the set of conjugacy classes on $G$. The function $\tr([k[G]]-|G|)=\tr_{k[G]}-|G|\in \sO(G)^{G}$ is 0 on $\{e\}$ but equal to $-|G|$ on any other conjugacy class; in particular it generates $J_G$ (which is the zero ideal of the subset $\{e\}\subset \mrm{Conj}(G)\simeq \Spec \sO(G)^{G}$). Consequently, $\tr(I_G) \cdot  \mc O(G)^G \simeq J_G$ as desired.
	
In general, the projection $G\ra \pi_0(G)$ induces a commutative diagram
$$
\begin{tikzcd}
	\Spec \sR(G)_k  \arrow[d,"p"] & \Spec \sO(G)^G \arrow[l,"\tr"']\arrow[d,"q"] &\\
	\Spec \sR(\pi_0(G))_k & \Spec \sO(\pi_0(G))^{\pi_0(G)}\arrow[l, "\tr_{\pi_0}"']\arrow[r,"\sim"]& \mrm{Conj}(\pi_0(G)).
\end{tikzcd}	
$$ 
Our claim amounts to saying that the pullback along $\tr$ of $I_{G,k}$ agrees with $J_G$. The affine scheme $\Spec \sO(G)^G$ splits as disjoint union of fibers over points of $\mrm{Conj}(\pi_0(G))$ and since $J_{\pi_0(G)}$ maps inside $J_G$ under $q^*$ we see that $J_G$ is supported on the preimage of $\{e\}\in \mrm{Conj}(\pi_0(G))$. Since $I_{\pi_0(G),k}$ maps inside $I_{G,k}$ under $p^*$ and the pullback of $I_{\pi_0(G),k}$ under $\tr_{\pi_0}$ is $J_{\pi_0(G)}$ we get that $\tr^*I_{G,k}$ is supported on the fiber over $\{e\}$ as well. Thus we only need to check if $\tr^*I_{G,k}$ agrees with $J_G$ on the fiber over the conjugacy class of $e$.

There  we have a commutative diagram
$$
\begin{tikzcd}
\sR(G)_k\otimes_{\sR(\pi_0(G))}k \arrow[r]\arrow[d]& \sO(G)^G\otimes_{\sO(\pi_0(G))^{\pi_0(G)}} k\arrow[d]\\
\sR(G^\circ)_k \arrow[r]& \sO(G^\circ)^{G^\circ}
\end{tikzcd}
$$
where the vertical maps are induced by the commutative (in fact fiber) square
$$
\begin{tikzcd}
	BG^\circ \arrow[r,"f"]\arrow[d]& BG \arrow[d,"p"]\\
	\{*\} \arrow[r,"q"]& B\pi_0(G)
\end{tikzcd}
$$
upon applying $K_0(-)_k$ and $HH(-)$ correspondingly. The map $f\colon BG^\circ \ra BG$ is a $\pi_0(G)$-torsor, in particular it is finite \'etale. So the push-forward $f_*$ (in terms of representations given by the induction $\Ind_{G_0}^G$) restricts to a functor $\Perf(BG^\circ) \ra \Perf(BG)$ giving ``transfer" maps $\sR(G^\circ)_k\ra \sR(G)_k$ and $\sO(G^\circ)^{G^\circ}\ra \sO(G)^G$ after taking $K_0$ and $HH_0$ respectively. 

By projection formula the composition 
$$
\Perf(BG) \xrightarrow{f^*} \Perf(BG^\circ)\xrightarrow{f_*}\Perf(BG)
$$
is given by tensoring with $f_*f^*\mc O_{BG}$, which then can also be identified with $p^* q_* \mc O_{\{*\}}$ via base change. In the language of representations the latter object is the group algebra $k[\pi_0(G)]$ considered as a representation of $G$.

We claim that $f^*$ induces an isomorphism $\sR(G)_k\otimes_{\sR(\pi_0(G))}k  \cong \sR(G^\circ)_k^{\pi_0(G)}$. The endofunctor $f_*f^*$ induces on $\sR(G)_k\otimes_{\sR(\pi_0(G))}k$ multiplication on $p^*(k[G])\in \sR(G)_k$. This is equivalent to multiplication by $|G|$ in the above tensor product. Since $k$ is of characteristic zero, the vector space $\sR(G)_k\otimes_{\sR(\pi_0(G))}k$ splits off $\sR(G^\circ)_k$. We further claim that, we can identify its image in $\sR(G)_k$ exactly with $\pi_0(G)$-invariants. 

First, the image of $f^*$ lies in the subgroup of invariants $\sR(G^\circ)_k^{\pi_0(G)} \subset \sR(G)_k$. Indeed, the action of $G$ by conjugation on $G^\circ$ induces an action of $G$ on $BG^\circ$ whose restriction to $G^\circ$ is trivial. This action induces a non-trivial actions of $\pi_0(G)=G/G^\circ$ on $K_0(BG)\simeq \sR(G)$.  Since the action of $G$ on $BG$ is trivial, given any $V\in \sR(G)$ its pull-back $f^*V\in \sR(G)$ is $\pi_0(G)$-invariant. 

Next, we claim that any element in the group of invariants $\sR(G^\circ)_k^{\pi_0(G)}$ also lies in the image of $f^*$: indeed,
$$
\sR(G^\circ)_k^{\pi_0(G)}\simeq \left(\frac{1}{|G|}\cdot \sum_{g\in \pi_0(G)} g \right)\cdot \sR(G^\circ)_k,
$$
while for any $G^\circ$-representation $V$, one has $f^*f_*V\simeq \mrm{Res}^G_{G^\circ}(\Ind_{G^\circ}^G(V))\simeq \oplus_{g\in \pi_0(G)} V^g$ (where $V^g$ denotes the pull-back of $V$ via the outer automorphism of $G^\circ$ induced by $g$). Therefore, we see that the class of $f^*f_*V$ in $\sR(G^\circ)$ is given by $\sum_{g\in \pi_0(G)} g \cdot [V]$, and so the claim is proved. 

Applying the same argument with $K_0(-)_k$ replaced with $HH(-)$ we get that $$
\sO(G)^G\otimes_{\sO(\pi_0(G))^{\pi_0(G)}} k\simeq (\mc O(G^\circ)^{G^\circ})^{\pi_0(G)}\simeq \mc O(G^\circ)^G
$$
(which is also more or less obvious to see directly).

Summarizing, we get a commutative diagram
$$
\begin{tikzcd}
	\sR(G)_k\otimes_{\sR(\pi_0(G))}k \arrow[r]\arrow[d, "\wr" ]& \sO(G)^G\otimes_{\sO(\pi_0(G))^{\pi_0(G)}} k\arrow[d,"\wr"]\\
	\sR(G^\circ)_k^{\pi_0(G)} \arrow[r]& (\mc O(G^\circ)^{G^\circ})^{\pi_0(G)}
\end{tikzcd}
$$
Now by part (2), the low horizontal map is an isomorphism, thus so is the top horizontal one. It follows that $I_G\otimes k$ maps isomorphically to $J_G\otimes k$, as we wanted.
\end{proof} 
	
We now deduce the result for a general reductive group $G$. 

\begin{prop}\label{prop:AS-ideal vs unit ideal}
	Let $G$ be a reductive group over $k$. For $n\gg 0$ we also have 
	$$
	J_G^n \subset \tr(I_G)\cdot \mc O(G)^G\subset J_G.
	$$
	Consequently, given a prestack $\sX\ra BG$, the AS-completion $
	HH(\sX)^{\wedge}_{I_G}
	$ coincides with $J_G$-adic completion of $HH(\sX)$ as a $HH(BG)$-module.
\end{prop}	
\begin{proof}
	Let as before $G^\circ$ be the connected component of $e$ and $\pi_0(G)$ be the ``group of connected components"; the latter is a finite \'etale group scheme. We have a short exact sequence
	$$1\ra G^\circ \ra G \ra\pi_0(G)\ra 1.$$
We will split the argument in two steps. First we discuss the case when the AS-ideal generates $J_G$ on the nose.
Now, for a general reductive $G$, there is a Galois extension $\ell/k$ such that part (3) of Lemma \ref{lem:G-reps vs G-invariant functions} applies to the base change $G_\ell$ (considered as a group scheme over $\ell$). To get back to $G$ we can look at the commutative diagram 
$$
\begin{tikzcd}
	\sR(G) \arrow[r,"\tr_k"] \arrow[d]&  \mc O(G)^G \arrow[d]\\
	\sR(G_{\ell}) \arrow[r, "\tr_\ell"] & \mc O(G_{\ell})^{G_{\ell}}.
\end{tikzcd}	
$$
Note that $\mc O(G_{\ell})^{G_{\ell}}\simeq \mc O(G)^G \otimes_k \ell$ and that $J_{G_\ell}\simeq J_G\otimes_k\ell$. Thus, $J_G^n\subset \tr_k(I_G)\cdot \mc O(G)^G$ if and only if $J_{G_\ell}^n\subset  \tr_k(I_G)\cdot \mc O(G_\ell)^{G_\ell}$. But, by Corollary \ref{cor:AS-ideal for base change} we have $I_{G_{\ell}}^n\subset I_G\cdot \sR(G_{\ell})$ for $n\gg 0$ and by Lemma \ref{lem:G-reps vs G-invariant functions} it follows that 
$$
J_{G_\ell}^n = \tr_{\ell}(I_{G_\ell}^n)\cdot \mc O(G_\ell)^{G_\ell} \subset \tr_k(I_G) \cdot \mc O(G_\ell)^{G_\ell}.
$$
\end{proof}
We next would like to interpret the ideal $J_G$ (as well as the $J_G$-adic completion) more geometrically. 
	
\begin{defn}
	Let $G$ be a reductive group over $k$ and denote by $G/\!\!/G:=\Spec \mc O(G)^G$ the GIT quotient (by conjugation). We have a point $[e]\in (G/\!\!/G)(k)$ corresponding to $e^*\colon  \mc O(G)^G \ra k$ and a map $G\ra G/\!\!/G$ corresponding to the embedding $\mc O(G)^G\subset \sO(G)$. We define \textbf{the locus of unipotent elements} $\mrm{Uni}(G)\subset G$ as the (a priori derived) fiber product
	$$
		\begin{tikzcd}
		\mrm{Uni}(G)\arrow[r]\arrow[d]& G\arrow[d]\\
		{[e]}  \arrow[r] & G/\!\!/G.
	\end{tikzcd}	
	$$
	
\end{defn}	

\begin{rem}
	The above definition diverges in general from the classical notion of ``unipotent cone of $G$" (which is typically considered in the case of algebraically closed field). Namely, our $\mrm{Uni}(G)$ might in general neither be reduced nor classical. However, we will be interested in the completion of $G$ along $\mrm{Uni}(G)$; the latter construction only depends on the reduced locus $\mrm{Uni}(G)^{\mrm{red}}$. Moreover, since over $\overline k$ the ring of $G$-invariants $\sO(G)^G$ is generated by functions of the form $\tr_V$, the $\overline k$-points $\mrm{Uni}(G)(\overline k)$ are given exactly by group elements $g\in G(\overline k)$ whose characteristic polynomial in any $G$-representation $V$ is the same as for the identity element $e\in G(\overline k)$; these are exactly the unipotent elements in $G(\overline k)$ .
\end{rem}	

\begin{exam}\label{ex:unipotent loci}\begin{enumerate}\item Let $G$ be abelian (and reductive). Then $\mc O(G)^G\simeq \mc O(G)$ and the map $G\ra G/\!\!/G$ is the identity map, so $\mrm{Uni}(G)$ is given by the unit element $e\in G$.
		
\item	Let $G$ be nice. Then we claim that the map $\{e\}\ra \mrm{Uni}(G)^{\mrm{red}}$ is an isomorphism. Indeed, we have a short exact sequence
$$
1\ra G^\circ \ra G \xrightarrow{p} \pi_0(G) \ra 1
$$
where $G^\circ$ is a torus (in particular abelian). It is enough to check the statement after base change to $\overline k$, in particular we can assume that $\pi_0(G)$ is constant. We have a commutative diagram 
$$
\begin{tikzcd}
G \arrow[r,"p"]\arrow[d,"\pi_G"]& \pi_0(G)\arrow[d,"\pi_{\pi_0(G)}"]&\\
	G/\!\!/G \arrow[r]& \pi_0(G)/\!\!/\pi_0(G)&\mrm{Conj}(\pi_0(G))\arrow[l,"\sim"']
\end{tikzcd}
$$
from which one sees that $\mrm{Uni}(G)=\pi_G^{-1}([e])\subset G$ is in fact a subscheme of $p^{-1}(e)\simeq G^\circ$. More precisely, it is identified with the spectrum of $\mc O(G^\circ)\otimes_{\mc O(G^\circ)^{\pi_0(G)}} k$ as a $\mc O(G^\circ)$-algebra. Since the defining ideal $I_e\subset \mc O(G^\circ)$ of $e\in G^\circ$ is $\pi_0(G)$-invariant, by Hilbert's theorem in invariant theory it is true that, up to taking radicals, $I_e$ generated by elements in $\mc O(G^\circ)^{\pi_0(G)}$; in other words
\[
\sqrt{I_e} = \sqrt{(\mc O(G^\circ)^{\pi_0(G)})}.
\]
Thus $\mc O(G^\circ)\otimes_{\mc O(G^\circ)^{\pi_0(G)}} k$ is a nilpotent extension of $k$, or, in other words, $\mrm{Uni}(G)^{\mrm{red}}=\{e\}$.
\end{enumerate}
\end{exam}	

We now will describe the AS-completion of $HH(\sX/\mbb Q)$ as functions on the completion along a certain substack of a loop stack. Below let us denote by $G_{\mrm{uni}}^\wedge\subset G$ the formal completion of $G$ along $\mrm{Uni}(G)$. 

\begin{constr}[Unipotent loops]\label{constr:unipotent loops} Recall the identification $\tau^{\mrm{cl}}(\sL_{\mbb Q}BG)\simeq [G/G]$ (see Example \ref{ex:formal loops in some examples2}). Let $(\sL_{\mbb Q}BG)_{\mrm{uni}}$ denote the completion of $\sL_{\mbb Q}BG$ along the composition 
	$$
	[\mrm{Uni}(G)/G]\hookrightarrow [G/G]\simeq  \tau^{\mrm{cl}}(\sL_{\mbb Q}BG) \hookrightarrow \sL_{\mbb Q}BG.
	$$
More generally, for a prestack $\sX\ra BG$ over $BG$ let $(\sL_\mbb Q\sX)_{\mrm{uni}}$ be the fiber product
$$
	\begin{tikzcd}
	(\sL_{\mbb Q}\sX)_{\mrm{uni}}\arrow[r]\arrow[d]& \sL_{\mbb Q}\sX\arrow[d]\\
	 (\sL_{\mbb Q}BG)_{\mrm{uni}} \arrow[r] & \sL_{\mbb Q}BG.
\end{tikzcd}	
$$
\end{constr}	

Note that $HH(BG/\mbb Q)$ is connective. Indeed, $\Perf(BG)$ splits into direct sum of copies of $\Vect(k)$ indexed by irreducible representations of $G$, and so $HH(BG/\mbb Q)\simeq \oplus_{[V_\lambda\in \mrm{Irr}(G)]} HH(k/\mbb Q)$. Thus we can consider the derived affine scheme $\Spec HH(BG/\mbb Q)$. Moreover, since $HH_0(k/\mbb Q)\simeq k$, we also see that $HH_0(BG/\mbb Q)\simeq HH_0(BG/k)\simeq \mc O(G)^G$. We denote by $\Spf (HH(BG/\mbb Q)^{\wedge}_{I_G})$ the derived completion of $\Spec HH(BG/\mbb Q)$ with respect to the (finitely generated) Atiyah--Segal ideal $I_G \subset HH_0(BG/\mbb Q)$.
\begin{lem}\label{lem:describing unipotent completion of loops of BG}There is a natural fiber square 
	$$
		\begin{tikzcd}
		(\sL_{\mbb Q}BG)_{\mrm{uni}} \arrow[r]\arrow[d]& \sL_{\mbb Q}BG\arrow[d]\\
		\Spf (HH(BG/\mbb Q)^{\wedge}_{I_G})\arrow[r] & \Spec (HH(BG/\mbb Q)).
	\end{tikzcd}	
	$$
\end{lem}	
\begin{proof} By definition, the top arrow is the completion of $\sL_{\mbb Q}BG$ along the embedding $[\mrm{Uni}(G)/G]\hookrightarrow [G/G]$ into the classical locus. By definition of $\mrm{Uni}(G)$ (and isomorphism $HH_0(BG/\mbb Q)\simeq \mc O(G)^G$) the latter map fits into a fiber square
	$$
	\begin{tikzcd}
		{[\mrm{Uni}(G)/G]} \arrow[r]\arrow[d]& {[G/G]}\arrow[d]\\
	\Spec (HH_0(BG/\mbb Q)/J_G)\arrow[r] & \Spec (HH_0(BG/\mbb Q)).
	\end{tikzcd}	
	$$
	Thus, by Remark \ref{rem:pull-back of a formal completion} we get that there is a fiber square for formal completions
	$$
			\begin{tikzcd}
		(\sL_{\mbb Q}BG)_{\mrm{uni}} \arrow[r]\arrow[d]& \sL_{\mbb Q}BG\arrow[d]\\
		\Spf (HH(BG/\mbb Q)^{\wedge}_{J_G})\arrow[r] & \Spec (HH(BG/\mbb Q)).
	\end{tikzcd}	
	$$
	Now, by Proposition \ref{prop:AS-ideal vs unit ideal}, $\Spf (HH(BG/\mbb Q)^{\wedge}_{I_G}) \simeq \Spf (HH(BG/\mbb Q)^{\wedge}_{J_G})$, and the statement of the lemma follows.
\end{proof}	

As a consequence we get the following:

\begin{prop}\label{prop:ig-uni}
	Let $\sX$ be a perfect stack endowed with a map $\sX\ra BG$. Then 
	$$
	HH(\sX/\mbb Q)^\wedge_{I_G} \simeq \sO((\sL_{\mbb Q}\sX)_{\mrm{uni}}).
	$$
\end{prop}	
\begin{proof}
From Lemma \ref{lem:describing unipotent completion of loops of BG} and the definition of $(\sL_{\mbb Q}\sX)_{\mrm{uni}}$ we get a fiber square
$$
\begin{tikzcd}
(\sL_{\mbb Q}\sX)_{\mrm{uni}}\arrow[r]\arrow[d]& \sL_{\mbb Q}\sX\arrow[d]\\
	\Spf (HH(BG/\mbb Q)^{\wedge}_{I_G})\arrow[r] & \Spec (HH(BG/\mbb Q)).
\end{tikzcd}
$$
By Proposition \ref{prop:loops_hh}, using that $\sX$ is perfect, we have a natural identification $\mc O(\sL_{\mbb Q}\sX)\simeq HH(\sX/\mbb Q)$. By Lemma \ref{lem:affine completion vs derived}, we then have that $\sO((\sL_{\mbb Q}\sX)_{\mrm{uni}})\simeq HH(\sX)^\wedge_{I_G}$, as desired.

\end{proof}	

\begin{rem}\label{rem:relation to Chen}
	In \cite[Proposition 2.1.25]{Chen_2020} Chen discusses that\footnote{Strictly speaking, in his situation $k$ is algebraically closed and $\sX \ra BG$ is representable in derived schemes.} the stack $(\sL\sX)_{\mrm{uni}}$ (defined as above) is an instance of the \textit{unipotent loop stack} as defined by Ben-Zvi--Nadler (\cite[Definition 6.1]{Ben_Zvi_2012}). In other words he identifies $(\sL\sX)_{\mrm{uni}}$ with the following mapping stack
	$$
	(\sL\sX)_{\mrm{uni}} \simeq \underline{\sM\mrm{aps}}(B\mbb G_a,\sX).
	$$
	The embedding $(\sL\sX)_{\mrm{uni}}\ra \sL\sX$ in these terms is induced by the natural \textit{affinization} map $S^1 \ra B\mbb G_a$ (see Appendix \ref{sec:appendix_formal loops} for a short reminder on the latter). The key computation in his argument is the identification
	$$
	\underline{\sM\mrm{aps}}(B\mbb G_a,BG) \simeq [G_{\mrm{Uni}(G)}^\wedge/G],
	$$
	and then one deduces a result for $\sX=[X/G]$ via a ``unipotent analogue" of Lemma \ref{lem:completion passes along representable maps}.
\end{rem}	
\ssec{AS completion theorem for $HH$} Let $G$ be a reductive group over a field $k$ of characteristic 0 and let $\sX\ra BG$ be a representable map of derived stacks; in other words $\sX=[X/G]$ where $X$ is a derived algebraic space. The closed embeddings $e\hookrightarrow \mrm{Uni}(G)\hookrightarrow G$ induce maps of quotients $[e/G]\hookrightarrow [\mrm{Uni}(G)/G]\hookrightarrow [G/G]$. Identifying $[G/G]\simeq \tau^{\mrm{cl}}(\mc L_{\mbb Q}BG)$ and completing $\mc L_{\mbb Q}BG$ along these closed substacks gives inclusions
$$
\widehat{\sL}_{\mbb Q}BG \ra (\sL_{\mbb Q}BG)_{\mrm{uni}} \ra \sL_{\mbb Q}BG
$$
(see Example \ref{ex:formal loops in some examples2}(2) and Construction \ref{constr:unipotent loops}). Pulling back along $\sL_{\mbb Q}\sX\ra \sL_{\mbb Q}BG$ gives
$$
\widehat{\sL}_{\mbb Q}\sX \ra (\sL_{\mbb Q}\sX)_{\mrm{uni}} \ra \sL_{\mbb Q}\sX 
$$
(see Proposition \ref{prop:pull-back square for formal loops} and Construction \ref{constr:unipotent loops}). In the case $\sX=[X/G]$ is perfect and geometric (e.g. if it is quasi-compact and ANS), we then get the natural maps
$$
HH(\sX/\mbb Q) \ra HH^\wedge_{I_G}(\sX/\mbb Q) \ra R^{\mrm{dAff}}HH(\sX/\mbb Q)
$$
passing to global functions. Atiyah-Segal theorem, if true, tells that the second map is an equivalence. Even though this is \textit{not} true in general (Remark \ref{rem:no AS in general}), it is true if we assume $G$ to be nice.

\begin{thm}\label{thm:completion-th-}
	Let $\sX=[X/G]$ where $X$ is a quasi-compact derived algebraic space and $G$ is a nice algebraic group. Then the natural map
	\[
	HH(\sX/\mbb Q)^\wedge_{I_G} \ra (R^{\mrm{dAff}} HH)(\sX/\mbb Q)
	\]
	is an equivalence.
\end{thm}	
\begin{proof} Recall that by our convention any derived algebraic space has affine diagonal, so $\sX$ is ANS. Since $\sX$ is also quasi-compact it follows that it is perfect (see Example \ref{ex:perfect stacks}). Thus, it is enough to show that 
	$$
	\widehat{\sL}_{\mbb Q}\sX \xrightarrow{\sim} (\sL_{\mbb Q}\sX)_{\mrm{uni}}.
	$$
	Since $G$ is nice, by Example \ref{ex:unipotent loci}(2), $\mrm{Uni}(G)^{\mrm{red}}\simeq \{e\}$. From this it follows that the corresponding completions $\widehat{\sL}_{\mbb Q}BG \ra (\sL_{\mbb Q}BG)_{\mrm{uni}}$, and consequently $\widehat{\sL}_{\mbb Q}\sX \ra (\sL_{\mbb Q}\sX)_{\mrm{uni}}$, are isomorphic.
\end{proof}	

\begin{rem} \label{rem:no AS in general}
	If $G$ is reductive but not necessarily nice, neither the map 
	\[
	 \widehat{\sL}_{\mbb Q}\sX\ra (\sL_{\mbb Q}\sX)_{\mrm{uni}},
	\]
	nor the induced map on global functions need to be isomorphisms in general.
	
	As a concrete example, let 
	$\mbb G_a=
	\{\begin{pmatrix}1&*\\
		0& 1\end{pmatrix} \}
	\subset SL_2$  over $\mbb Q$
	and consider $\sX=[X/G]$ for
	$X:=SL_2/\mbb G_a$ with the action of $G=SL_2$ by left translations. Note that 
	$\sX\simeq [(SL_2/\mbb G_a)/SL_2]\simeq  B\mbb G_a.$ By Example \ref{ex:formal loops in some examples2} we can identify
	\[
	\sL_{\mbb Q}\sX\simeq [\mbb G_a/\mbb G_a]\simeq \mbb G_a\times B\mbb G_a
	\;\;\;\;\;\;\;\; \text{ and } \;\;\;\;\;\;\;\;
	\widehat{\sL}_{\mbb Q}\sX\simeq \widehat{\mbb G_a} \times B\mbb G_a,
	\]
	with $\widehat{\mbb G_a}$ being the completion of $\mbb G_a$ at $0$. 
	
	On the other hand, note that $\mbb G_a\subset \mrm{Uni}(G)$, so the composite map $\sL_{\mbb Q}\sX\ra [G/G]\ra G/\!\!/G$ factors 
	though $[e]\in G/\!\!/G$. It follows that all loops in $\sX$ are unipotent:
	\[
		(\sL_{\mbb Q}\sX)_{\mrm{uni}}\simeq  \sL_{\mbb Q}\sX\times_{G/\!\!/G}(G/\!\!/G)^\wedge_{[e]}\simeq \sL_{\mbb Q}\sX^\wedge_{\sL_{\mbb Q}\sX}\simeq \sL_{\mbb Q}\sX.
	\]
	We get that the map $\widehat{\sL}_{\mbb Q}\sX\ra (\sL_{\mbb Q}\sX)_{\mrm{uni}}$ is explicitly described by the natural embedding
	$$
	\widehat{\mbb G_a} \times B\mbb G_a \subset {\mbb G_a} \times B\mbb G_a.
	$$
	Passing to global functions on both sides we then also get 
	\[
	R^{\mrm{dAff}} HH(\sX/\mbb Q)\simeq \mbb Q[[x]] \oplus \mbb Q[[x]][-1], \quad \text{ while } \quad HH(\sX/\mbb Q)^\wedge_{I_G}\simeq \mc O({\mbb G_a} \times  B\mbb G_a)\simeq \mbb Q[x]\oplus \mbb Q[x][-1].
	\]
	This illustrates that the AS-completion in general is weaker than the one given by $R^{\mrm{dAff}}$.
\end{rem}	

For further applications to $K$-theory we will need a slight variant of Theorem \ref{thm:completion-th-}. 

\begin{thm} \label{thm:general completion thm for HH}
	Let $\sX=[X/G] \rightarrow BG$ be a representable morphism from an ANS stack $\sX$. Assume that $\sX^{\mrm{cl}} \rightarrow BG$ is of finite type. Then the comparison map
	\[
	HH(\sX/\mbb Q)^\wedge_{I_G} \ra (R^{\mrm{dAff}} HH)(\sX/\mbb Q)
	\]
	is an equivalence.
\end{thm}	
\begin{proof}
Again, it is enough to show that the map $\widehat{\sL}_{\mbb Q}\sX \ra (\sL_{\mbb Q}\sX)_{\mrm{uni}}$  is an isomorphism. Note that $({\sL}_{\mbb Q}\sX)^{\mrm{cl}} \simeq \sI(\sX^{\mrm{cl}})$ (the inertia stack). We have a natural map $f\colon  \sI(\sX^{\mrm{cl}})\ra \sI(BG)\simeq [G/G]$ and an affine projection $p\colon \sI(\sX^{\mrm{cl}}) \ra \sX^{\mrm{cl}}$. Note that by \cite[Tag 04XS]{stacks} $p$ is locally of finite type. By construction and Proposition \ref{prop:pull-back square for formal loops}, the two loop stacks in question are the completions of ${\sL}_{\mbb Q}\sX$ along the two closed substacks in the classical locus $\sI(\sX^{\mrm{cl}})$ given by preimages $f^{-1}([\{e\}/G]) $ and $f^{-1}([\mrm{Uni(G)}/G])$ correspondingly. Note that by our assumption $\sX^{\mrm{cl}}$ itself is also of finite type over the field $k$. Consequently, there is a simplicial (classical) affine scheme of finite type $Z_\bullet$ such that $\sX^{\mrm{cl}}\simeq L_\et (|Z_\bullet|)\in \mrm{Stk}^\mrm{cl}$. By Remark \ref{rem:properties of completion}\ref{item:completion of the colimit} it is enough to show that for any $i$ the completion along substacks $f^{-1}([\{e\}/G]) $ and $f^{-1}([\mrm{Uni(G)}/G])$ agree after we pull-back further along $q_i\colon \sI(\sX^{\mrm{cl}})\times_{\sX^{\mrm{cl}}}Z_i  \ra \sI(\sX^{\mrm{cl}})$. Both $(f\circ q_i)^{-1}([\{e\}/G]) $ and $(f\circ q_i)^{-1}([\mrm{Uni(G)}/G])$ are some closed subschemes of the finite type affine scheme $\sI(\sX^{\mrm{cl}})\times_{\sX^{\mrm{cl}}}Z_i$ and to check that the corresponding completions agree it's enough to check that the reduced parts of all fibers over $\overline k$-points of $Z_i$ coincide. Let $z\in Z_i(\overline k)$ and let $x\in \sX^{\mrm{cl}}$ be its image; then the fiber $\sI(\sX^{\mrm{cl}})\times_{\sX^{\mrm{cl}}}\{x\}$ is identified with the stabilizer $G_x:= \underline{\mrm{Aut}}_{\sX^{\mrm{cl}}}(\{x\})$. Moreover, the projection to $ \sI(BG)\simeq [G/G]$ is identified with the natural compostion $G_x\hookrightarrow G \ra [G/G]$. Note that since $\sX$ is ANS the group $G_x$ is a nice $\overline k$-group scheme. The two closed subschemes in question are then identified with the preimage of $[\{e\}/G]$ and $[\mrm{Uni}(G)/G]$ correspondingly; they are given by the unit $\{e\}\in G_x$ and zeros of the ideal $J_G\cdot \mc O(G_x)\subset \mc O(G_x)$ correspondingly. By the change-of-groups lemma for AS-ideals (Lemma \ref{lem:comparing Atiyah-Segal completions}) and Proposition \ref{prop:AS-ideal vs unit ideal} we get $J_G\cdot \mc O(G_x)$ agrees with $J_{G_x}$ up to passing to radicals. By Example \ref{ex:unipotent loci}, using that $G_x$ is nice, we then get that the reduced loci of the fibers over $\{x\}$ are the same, which concludes the proof.
\end{proof}

\ssec{Cyclic homology}$HH(\sC/\mbb Q)$ admits a canonical action of $S^1$ which leads to the three other localizing invariants:
\[
\begin{tikzcd}[row sep = 0.1]
	\text{ negative cyclic homology:} &HC^-(\sC/\mbb Q) := HH(\sC/\mbb Q)^{hS^1};\\
	\text{ periodic cyclic homology:} &HP(\sC/\mbb Q) := HH(\sC/\mbb Q)^{tS^1}; \\
	\text{ cyclic homology:}  &HC(\sC/\mbb Q) := HH(\sC/\mbb Q)_{hS^1}.
\end{tikzcd}
\]

We observe that Theorem~\ref{thm:completion-th-} also implies the AS completion theorem for all the above versions of cyclic homology.

\begin{cor}\label{cor:completion-th-HP}
Let $G$ be a reductive group over a field of characteristic 0. Let $\sX=[X/G] \rightarrow BG$ be a representable morphism from an ANS stack $\sX$, and assume that $\sX^{\mrm{cl}} \rightarrow BG$ is of finite type.
Then the maps 
\begin{align*}
	HC^-(\sX/\mbb Q)^\wedge_{I_G} &\to R^{\mrm{dAff}}HC^-(\sX/\mbb Q)\\
HP(\sX/\mbb Q)^\wedge_{I_G} &\to R^{\mrm{dAff}}HP(\sX/\mbb Q)\\
HC(\sX/\mbb Q)^\wedge_{I_G} &\to R^{\mrm{dAff}}HC(\sX/\mbb Q)
\end{align*}
are all equivalences.
\end{cor}
\begin{proof}
Recall that there is a fiber sequence 
\[
HC \to HC^- \to HP
\]
of functors, so to prove the claim it suffices to prove it for $HC^-$ and for $HP$. The map in Theorem~\ref{thm:general completion thm for HH} is $S^1$-equivariant by construction, and so induces an equivalence of $S^1$-equivariant $k$-modules. Applying $S^1$-fixed points commutes both with right Kan extension and completion, this gives the claim for $HC^-$. The result for $HP$ follows from Theorem~\ref{thm:main_2} (recall that 
$HP$ is $\mathbb{A}^1$-invariant \cite[Theorem~4.1(iii)]{farfromA1_weibel} and is truncating 
\cite[Theorem~IV.2.1]{Goodwillie_derivations}).
\end{proof}

\section{Atiyah-Segal completion theorem for \texorpdfstring{$K$}{K}-theory}\label{sec:ask} We finally arrive at a proof of our main results. We begin by extending the relationship between $K$-theory, homotopy $K$-theory and negative cyclic homology to the case of derived algebraic stacks.


%
%
%

\subsection{Relations between K-theory, homotopy K-theory and negative cyclic homology of stacks} Recall that we are working in characteristic zero; to simplify matters we simply set $k = \bbQ$. One of the main tools in accessing the algebraic $K$-theory of non-smooth schemes in characteristic zero is the \textbf{Goodwillie Jones chern character} $K \rightarrow HC^{-}(-/\bbQ)$ which refines the Dennis trace map along the map $HC^-(-/\bbQ) \rightarrow HH(-/\bbQ)$. This map fits as the top horizontal arrow of the following square, natural in qcqs schemes:

\begin{equation}\label{eq:goodwillie}
\begin{tikzcd}
K(X) \arrow[r]\arrow[d] & HC^-(X/\bbQ) \arrow[d] \\
KH(X) \arrow[r] & L_{\mrm{cdh}}HC^-(X/\bbQ)\footnote{The notation $L_{\mrm{cdh}}HC^-(X/\bbQ)$ is a slight abuse of notation: it should really be $L_{\mrm{cdh}}HC^-(-/\bbQ)(X)$.}. \\
\end{tikzcd}
\end{equation}

The bottom horizontal map is the $\cdh$-sheafification of the top horizontal map and the square itself is cartesian; these results are highly non-obvious. They encode three ``miracles'' on algebraic $K$-theory: 1) that $KH$ satisfies cdh descent \cite{Cisinski}, 2) that $L_{\cdh}\K$ is $\mathbb{A}^1$-invariant \cite{KST} and that 3) the fiber of the chern character $K \rightarrow HC^-(-/\bbQ)$ is a cdh sheaf, a result due to Corti\~{n}as \cite{cortinas}. We note that 1) and 2) in characteristic zero is one of the key results in Haesemeyer's thesis \cite{Haesemeyer}, which gives the first hints that 1) and 2) can hold in very general settings.

In the next proposition, we extend~\eqref{eq:goodwillie} to certain derived algebraic stacks. As we have avoided speaking of the cdh topology on derived algebraic stacks, we will make use of the following definition instead; see also \cite[Lemma 5.7]{e-morrow}.

\begin{defn}\label{def:hc-} Define the localizing invariant $L_{\cdh}HC^-(-/\bbQ)$ to be the pushout of functors
\[
L_{\cdh}HC^-(-/\bbQ):= \colim\left(HC^-(-/\bbQ) \leftarrow K \rightarrow KH\right): \Cat^{\perf}_k \rightarrow \Spt.
\]
Noting that pushouts commute with fibre sequences, the pushout above can be taken pointwise in the $\infty$-category $\Fun(\Cat^{\perf}_{\bbQ}, \Spt)$. 
\end{defn}

\begin{prop}\label{prop:agreement} The following hold for $L_{\cdh}HC^{-}(-/\bbQ)$:
\begin{enumerate}
\item on qcqs derived $k$-schemes, $L_{\cdh}HC^{-}(-/\bbQ)$ agrees with the $\cdh$-sheafification of $HC^{-}(-/\bbQ)$;
\item on a $k$-smooth algebraic stack $\sX$, the map $HC^-(\sX/\bbQ) \rightarrow L_{\cdh}HC^{-}(\sX/\bbQ)$ is an equivalence;
\item  $L_{\cdh}HC^{-}(-/\bbQ)$ is a localizing invariant, is truncating and, in particular, nilinvariant;
\item on ANS algebraic $k$-stacks, $L_{\cdh}HC^{-}(-/\bbQ)$ is cdh-excisive.
\end{enumerate}
\end{prop}

\begin{proof} 
\begin{enumerate}
\item This follows from the fact that~\eqref{eq:goodwillie} is a pushout as well on the category of qcqs derived $k$-schemes.
\item By construction, it suffices to prove that if $\sX$ is a smooth algebraic $k$-stack, then $K(\sX) \simeq KH(\sX)$. In this case, we just need to prove that $K(\sX) \simeq K(\sX \times \mathbb{A}^n)$ for all $n \geq 0$. There is a map $K(\sX \times \mathbb{A}^r) \rightarrow G(\sX \times \mathbb{A}^r)$ induced by the functor of stable $\infty$-categories $\Perf(\sX \times \mathbb{A}^r) \rightarrow \mrm{Coh}(\sX \times \mathbb{A}^r)$. As explained in \cite[3.5]{khan-kg}, $G(\sX \times \mathbb{A}^r) \simeq G(\sX)$ for any noetherian derived Artin stack, hence it suffices to prove that the functor $\Perf(\sX \times \mathbb{A}^r) \rightarrow \mrm{Coh}(\sX \times \mathbb{A}^r)$ is an equivalence. By flat descent for both $\Perf$ and $\mrm{Coh}$, the functor agrees since they agree on smooth $k$-schemes.
\item By construction we have a fiber sequence
\[
\mathrm{fib}\left(K \rightarrow HC^{-}(-/\bbQ)\right) \rightarrow KH \rightarrow  L_{\cdh}HC^{-}(-/\bbQ).
\]
Since $KH$ is truncating \cite[Proposition 3.14]{land-tamme}, the result then follows from the fact that the fibre is truncating by Goodwillie's theorem \cite{goodwillie-theorem}. 
\item By \cite[Corollary 5.2.6]{dgm-elden-vova}, both $\mathrm{fib}\left(K \rightarrow HC^{-}(-/k)\right)$ and $KH$ are cdh-excisive since they are truncating. This implies that $L_{\cdh}HC^{-}(-/\bbQ)$ is as well. 
\end{enumerate}

\end{proof}

\begin{rem}\label{rem:tc} We note that some aspects of Proposition~\ref{prop:agreement} holds more generally. On the $\infty$-category of derived algebraic stacks, we can also define $L_{\cdh}TC$ of stacks via the pushout of $TC \leftarrow K \rightarrow KH$ and the analogs of Proposition~\ref{prop:agreement} will also hold; here we use a version of~\eqref{eq:goodwillie} that holds for qcqs schemes and involves $TC$ \cite[Theorem 3.8]{e-morrow}. We also note that, as one can see from the proof, the agreement between $HC^-(\sX/\bbQ)$ and $ L_{\cdh}HC^{-}(\sX/\bbQ)$ holds for stacks which are $K$-regular: those that satisfy $K(\sX \times \bbA^r) \simeq K(\sX)$ for all $r \geq 0$. However, it is not clear how to prove the next corollary without resolution of singularities. In particular, it establishes an Atiyah-Segal completion theorem for a truncating invariant which is not $\mathbb{A}^1$-invariant.
\end{rem}

\begin{cor}\label{prop:hc-as} Fix $G$ a reductive group over a characteristic zero field $k$. Let $\sX \rightarrow BG$ be representable morphism where $\sX$ is derived algebraic stack, such that the morphism $\sX^{\mrm{cl}} \rightarrow BG$ is finite type. Assume further that $\sX$ is ANS. Then the canonical map
\[
L_{\cdh}HC^{-}(\sX/\bbQ)^{\wedge}_{I_G} \rightarrow R^{\mathrm{dAff}}L_{\cdh}HC^{-}(-/\bbQ)(\sX)
\]
is an equivalence.
\end{cor}

\begin{proof} Since both functors are cdh-excisive (by Proposition~\ref{prop:agreement} for the domain and Lemma~\ref{lem:localizing-invt} for the target) we can reduce to the case of $\sX$ a smooth, ANS $k$-stack by an inductive argument similar to the proof of Theorem~\ref{thm:main_2}. In this case, Proposition~\ref{prop:agreement}(3) (resp. Lemma~\ref{lem:agree-daff}) shows that $L_{\cdh}HC^{-}(\sX/\bbQ) \simeq HC^{-}(\sX/\bbQ)$ (resp. $R^{\mathrm{dAff}}L_{\cdh}HC^{-}(\sX/\bbQ) \simeq HC^{-}(\sX/\bbQ)$). Therefore, the result follows from the analogous statement for $HC^{-}(-/\bbQ)$ in place of $L_{\cdh}HC^{-}(-/\bbQ)$, which is Corollary~\ref{cor:completion-th-HP}. 


\end{proof}

\begin{rem}[Corollary~\ref{prop:hc-as} without resolution of singularities]\label{rem:resolution} It might be possible to prove Corollary~\ref{prop:hc-as} without resolution of singularities; such a proof method will surely establish an analogous result for $L_{\cdh}TC$ in mixed and positive characteristics. A possible strategy is to develop seriously the cdh topology on algebraic stacks, prove a hypercompleteness result just as in \cite{EHIK} and characterize the points of such a topology. The role of valuation rings in a stacky situation should be played by the root stacks. A hint that this is true is given by the valuative criterion for proper maps of tame stacks, recently established by Bresciani and Vistoli \cite{bresciani-vistoli}. 
\end{rem}

We have used the following lemma in Corollary~\ref{prop:hc-as}.

\begin{lem} \label{lem:agree-daff} Let $\sX$ be a $k$-smooth, algebraic stack. Then we have a canonical equivalence
\[
R^{\mathrm{dAff}}HC^{-}(\sX/\bbQ) \rightarrow R^{\mathrm{dAff}}L_{\cdh}HC^{-}(-/\bbQ)(\sX).
\]
\end{lem}

\begin{proof} We first claim that any \'etale surjection of algebraic stacks $f:\sY \rightarrow \sX$ is of universal descent for $R^{\mathrm{dAff}}L_{\cdh}HC^{-}(-/\bbQ)$. Using Proposition~\ref{prop:spc}, it suffices to prove that $L_{\cdh}HC^{-}(-/\bbQ)$ satisfies \'etale descent for classical commutative rings (since it is nilinvariant); by Proposition~\ref{prop:agreement} this is computed by $\cdh$-sheafifying $HC^{-}(-/\bbQ)$ and then restricting to affine schemes. To see this we can appeal to \cite[Theorem A.3]{e-morrow}, applying it to the HKR filtration on $HC^{-}(-/\bbQ)$ \cite[Theorem 4.2]{e-morrow}. Note that we can reduce to the quasisyntomic case for $L_{\cdh}HC^{-}(-/\bbQ)$ by resolution of singularities. Alternatively, we know that the fibre of the map $HC^{-}(-/\bbQ) \rightarrow L_{\cdh}HC^{-}(-/\bbQ)$ is the fibre of $K \rightarrow KH$ by~\eqref{eq:goodwillie}. The fibre of $K \rightarrow KH$ then satisfies \'etale descent by a theorem of van der Kallen \cite{vdk-descent}.

Now, since $\sX$ is a smooth algebraic $k$-stack, it admits an \'etale surjection $\Spec R \rightarrow \sX$ where $R$ is a smooth $k$-algebra. Since any \'etale surjection is of universal $HC^{-}(-/\bbQ)$-descent it suffices to note that $HC^{-}(X/\bbQ)\simeq L_{\cdh}HC^{-}(X/\bbQ)$ on smooth $k$-schemes\footnote{The argument of the first paragraph of \cite[Theorem 3.12]{CHSW} applies, using resolution of singularities.}. 
\end{proof}

\subsection{{Proof of Theorem~\ref{thm:main}}} Now we are ready to prove our main theorem. We restate it for the reader's convenience. 

\begin{thm}\label{thm:main} Fix $G$ a reductive group over a characteristic zero field $k$. Let $\sX \rightarrow BG$ be representable morphism where $\sX$ is derived algebraic stack, such that the morphism $\sX^{\mrm{cl}} \rightarrow BG$ is finite type. Assume further that $\sX$ is ANS. Then we have an equivalence:
\begin{equation}\label{eq:main}
K(\sX)^{\wedge}_{I_G} \simeq R^{\mrm{dAff}}K(\sX),
\end{equation}
where the limit is taken across all morphisms from an derived affine scheme to $\sX$.
\end{thm}

\begin{proof}

First, take the right Kan extension of the square~\eqref{eq:goodwillie} to obtain the cartesian square on derived algebraic stacks over $\bbQ$:
\[
\begin{tikzcd}
R^{\mrm{dAff}}K \arrow[d]\arrow[r] & R^{\mrm{dAff}}HC^{-}(-/\bbQ)\arrow[d]\\
R^{\mrm{dAff}}KH \arrow[r] & R^{\mrm{dAff}}L_{\cdh}HC^{-}(-/\bbQ)
\end{tikzcd}
\]
We have a comparison map of cartesian squares

%
\[
\begin{tikzcd}
K(\sX)^\wedge_{I_G} \arrow[d]\arrow[r] & HC^{-}(\sX/\bbQ)^\wedge_{I_G} \arrow[d]\\
KH(\sX)^\wedge_{I_G} \arrow[r] & L_{\cdh}HC^{-}(\sX/\bbQ)^\wedge_{I_G}
\end{tikzcd}
\Rightarrow
\begin{tikzcd}
R^{\mrm{dAff}}K(\sX) \arrow[d]\arrow[r] & R^{\mrm{dAff}}HC^{-}(\sX/\bbQ)\arrow[d]\\
R^{\mrm{dAff}}KH(\sX) \arrow[r] & R^{\mrm{dAff}}L_{\cdh}HC^{-}(\sX/\bbQ).
\end{tikzcd}
\]
We remark that the cartesian-ness of the left-hand-side is by Definition~\ref{def:hc-}. 

It thus suffices to prove that the map on the bottom left, bottom right and top right vertices are equivalences. This follows from Theorem~\ref{thm:main_2}, Corollary~\ref{prop:hc-as} and Corollary~\ref{cor:completion-th-HP} respectively.

\end{proof}

\ssec{A counterexample}\label{sec:counterexample}

Let $R\in \CAlg^{\mrm{an}}_{\mbb Q}$ be an animated $\mbb Q$-algebra and denote by $B\mbb G_{a, R}:= B\mbb G_a \times \Spec R$. Let $\mbb Q[\varepsilon]:=\mbb Q[x]/x^2$. By $HH$ (resp. $HC^-$) below we mean Hochschild homology (resp. negative cyclic homology) relative to $\mbb Q$. Our goal in this section is to show that AS-completion theorem for $B\mbb G_{a, \mbb Q[\varepsilon]}$ (viewed as a stack over $B\mrm{SL_2}$, as in Remark \ref{rem:no AS in general}) does \textit{not} hold. Our 
argument will rely on Levy's generalization of Dundas--Goodwillie--McCarthy theorem for $(-1)$-connective rings 
\cite{levy2022algebraic} and the identification of $S^1$-equivariant structure on $HH(B\mbb G_a)$.

1. Note that $B\mbb G_{a,\mbb Q}$ is an affine stack in the sense of To{\"e}n, corepresented by the algebra 
$\Sym_{\mbb Q}(\Omega \mbb Q)$ (see \cite[D{\'e}finition~2.2.4~and~Lemma~2.2.5]{Ton2006}). 
In particular, $\mathrm{D}_{\mrm{qc}}(B\mbb G_{a,\mbb Q})$ is generated under colimits and shifts by the 
structure sheaf $\mc O$, and $\End_{\Perf(B\mbb G_{a, \mbb Q})}\simeq \Sym_{\mbb Q}(\Omega \mbb Q)$ 
(see \cite[Proposition~4.5.2]{dagVIII}). Since by \cite[Theorem~A(1)]{HallRydh2} perfect complexes over $B\mbb G_{a,\mbb Q}$ 
compactly generate $\mathrm{D}_{\mrm{qc}}(B\mbb G_{a,\mbb Q})$, this induces an equivalence of categories 
$$\Perf(B\mbb G_{a,\mbb Q})\simeq \Perf(\Sym_{\mbb Q}(\Omega \mbb Q)),$$
and a further $S^1$-equivariant identification 
$$
HH(B\mbb G_{a,\mbb Q})\simeq HH(\Sym_{\mbb Q}(\Omega \mbb Q))\simeq (\Sym_{\mbb Q}(\Omega \mbb Q))^{\otimes S^1} \simeq \Sym_{\mbb Q}(\Omega \mbb Q[S^1]),
$$ 
where $\mbb Q[S^1]$ is the regular $S^1$-representation (given by singular homology $C_*^{\mrm{sing}}(S^1,\mbb Q)$). Note that $H^*(\Sym_{\mbb Q}(\Omega \mbb Q[S^1]))\simeq \mbb Q[x] \oplus \Omega \mbb Q[x]$ and that $\Sym_{\mbb Q}(\Omega \mbb Q[S^1])$ is an $S^1$-equivariant module over $\mbb Q[x]$ (with the trivial $S^1$-action on the latter). 

Let $\mrm{pt} \ra B\mbb G_{a,\mbb Q} \ra \mrm{pt}$ be the natural maps; they split $HH(B\mbb G_{a,\mbb Q})$ as 
$$
HH(B\mbb G_{a,\mbb Q})\simeq \mbb Q \oplus \overline{HH}(B\mbb G_{a,\mbb Q}).
$$
We have a natural $S^1$-equivariant map $\Omega \mbb Q[S^1]\ra \Sym_{\mbb Q}(\Omega\mbb Q[S^1])\simeq HH(B\mbb G_{a,\mbb Q})$; it's easy to see (looking at what this map does on cohomology) that the map 
$$
\mbb Q[x]\otimes \Omega \mbb Q[S^1] \ra HH(B\mbb G_{a,\mbb Q})
$$
extended via $\mbb Q[x]$-module structure is an isomorphism onto $\overline{HH}(B\mbb G_{a,\mbb Q})$. To summarize, we get an $S^1$-equivariant identification 
$$
HH(B\mbb G_{a,\mbb Q})\simeq \mbb Q \oplus \left(\Omega \mbb Q[x]\otimes \mbb Q[S^1]\right),
$$
where the actions on $\mbb Q$ and $\mbb Q[x]$ are trivial. Taking $S^1$-fixed points then yields an identification
$$
HC^-(B\mbb G_{a,\mbb Q})\simeq HC^-(\mbb Q) \oplus \Omega \mbb Q[x]
$$
where $HC^-(\mbb Q)\simeq  C^*_{\mrm{sing}}(BS^1,\mbb Q)$. Here recall that $S^1$-fixed points of $V\otimes  \mbb Q[S^1]$ are given by $V$ for any $S^1$-module $V$.

2. 
By \cite[Theorem 1.2(1)]{BFN} we have $\Perf(B\mbb G_{a,R})\simeq \Perf(B\mbb G_{a,\mbb Q})\otimes \Perf(R)$; this gives (again, $S^1$-equivariant) splitting as
$$
HH(B\mbb G_{a,R})\simeq HH(B\mbb G_{a,\mbb Q})\otimes HH(R)\simeq HH(R) \oplus \left(\Omega HH(R)[x]\otimes \mbb Q[S^1]\right).
$$
Taking $S^1$-fixed points gives 
$$
HC^{-}(B\mbb G_{a,R})=HC^-(R) \oplus \Omega HH(R)[x].
$$

3. Recall from Remark \ref{rem:no AS in general} that we have $R^\mrm{dAff}HH(B \mbb G_{a,\mbb Q})\simeq \mbb Q[[x]] \oplus 
\Omega \mbb Q[[x]]$. Moreover, the natural map $HH(\mbb G_{a,\mbb Q}) \ra R^\mrm{dAff}HH(\mbb G_{a,\mbb Q})$ can be identified 
with the map 
\[
\mbb Q[x] \oplus \Omega \mbb Q[x]\ra  \mbb Q[[x]] \oplus \Omega \mbb Q[[x]]
\] 
sending $x$ to $x$. The $S^1$-module $R^\mrm{dAff}HH(\mbb G_{a,\mbb Q})$ is a module over its 0-th cohomology $\mbb Q[[x]]$ (on which the action is trivial) and, arguing as for $HH(\mbb G_{a,\mbb Q})$, we get an $S^1$-equivariant decomposition 
$$
R^\mrm{dAff}HH(B\mbb G_{a,\mbb Q}) \simeq \mbb Q \oplus \left(\Omega \mbb Q[[x]]\otimes \mbb Q[S^1]\right).
$$
Passing to $S^1$-fixed points (and commuting them through the right Kan extension) we get 
$$
R^\mrm{dAff}HC^-(B\mbb G_{a,\mbb Q})\simeq HC^-(\mbb Q) \oplus \Omega \mbb Q[[x]].
$$

4. Analogously, for $R^\mrm{dAff}HH(B\mbb G_{a,R})$ we get an $S^1$-equivariant splitting as 
$$
R^\mrm{dAff}HH(B\mbb G_{a,R})\simeq HH(R) \oplus \left(\Omega HH(R)[[x]]\otimes \mbb Q[S^1]\right);
$$
indeed, left hand side is given by functions on the formal loop stack of $B\mbb G_{a,R}$. Arguing as in Example \ref{ex:formal loops in some examples2}\ref{item:relative loops of BG}, the classical locus of $\sL B\mbb G_{a,R}$ is identified with $\mbb G_{a,\mbb Q} \times B\mbb G_{a,\mbb Q}\times \Spec \pi_0(R)$, with constant loops given by $\{e\}\times B\mbb G_{a,\mbb Q}\times \Spec \pi_0(R)$. This allows to identify $R^\mrm{dAff}HH(B\mbb G_{a,R})$ with the $x$-adic completion of $HH(R)$, giving the formula above.  
Applying $S^1$-fixed points gives a formula for negative cyclic homology
$$
R^\mrm{dAff}HC^-(B\mbb G_{a,R})\simeq HC^-(R) \oplus \Omega HH(R)[[x]].
$$

5. Following Remark \ref{rem:no AS in general} we view $B\mbb G_{a,\mbb Q}$ as the quotient (of the general affine space $SL_2/\mbb G_a$) by $SL_2$. We also have a map $B\mbb G_{a,\mbb Q[\varepsilon]} \ra B\mbb G_{a,\mbb Q} \ra BSL_2$. As discussed in loc.cit. the image of the AS-ideal $I_{SL_2}\subset K_0(BSL_2)$ in $HH_0(B\mbb G_{a,\mbb Q})$ (and, consequently, $HH_0(B\mbb G_{a,\mbb Q[\varepsilon]})$) is zero.

6. To pass to K-theory, consider the commutative diagram 
\[
\begin{tikzcd}
K(B\mbb G_{a,\mbb Q[\varepsilon]})^\wedge_{I_{SL_2}}\arrow[d]\arrow[r] & HC^-(B\mbb G_{a,\mbb Q[\varepsilon]})^\wedge_{I_{SL_2}}\arrow[d]\arrow[r] & R^\mrm{dAff}HC^-(B\mbb G_{a,\mbb Q[\varepsilon]})\arrow[d]\\
K(B\mbb G_{a,\mbb Q})^\wedge_{I_{SL_2}}\arrow[r] & HC^-(B\mbb G_{a,\mbb Q})^\wedge_{I_{SL_2}}\arrow[r] & R^\mrm{dAff}HC^-(B\mbb G_{a,\mbb Q}).
\end{tikzcd}
\]
It follows from \cite[Theorem~B]{levy2022algebraic} that the functor 
\[
\sC \mapsto \Fib\left(K(\Perf(B\mbb G_{a}) \otimes \sC) \to HC^-(\Perf(B\mbb G_{a}) \otimes \sC) \right)
\]
is a truncating localizing invariant. Hence by \cite[Corollary~3.6]{land-tamme} it is also nilinvariant. 
Hence, the left square in the diagram is a pullback square. 
If the AS completion theorem were true in this case, then the outer rectangle would be a pullback as well, by 
Proposition~\ref{lem:localizing-invt} applied to $\Fib(K \to HC^-)$. 
By the pasting law for pullbacks, this means that the right square would be a pullback. 
However, it follows from the computations above that that commutative square may be identified with the square 
\[
\begin{tikzcd}
HC^-(\mbb Q[\varepsilon]) \oplus \Omega HH(\mbb Q[\varepsilon])[x] \arrow[r]\arrow[d] &HC^-(\mbb Q[\varepsilon]) \oplus \Omega HH(\mbb Q[\varepsilon])[[x]]\arrow[d]\\ 
HC^-(\mbb Q) \oplus \Omega HH(\mbb Q)[x] \arrow[r] &HC^-(\mbb Q) \oplus \Omega HH(\mbb Q)[[x]]
\end{tikzcd}
\]
which is not pullback. This follows, for example, by noting that the homotopy groups of the left vertical fibre is countable while the ones on the right are the uncountable. 
\section{Applications} \label{sec:apps}

In this application section, we offer two consequences of our main theorem. One application reads the theorem from right to left and the other from left to right. 

\subsection{Functorial pushforwards}  For this application, we fix a field $k$ of charateristic zero, although some of the preliminary discussion works more generally. We begin by recalling some functoriality of the formation of perfect complexes. We first work in the generality of perfect derived stacks; given a morphism  $f: \sX \rightarrow \sY$ we have the pullback functor $f^*: \mrm{D}_{\mrm{qc}}(\sY) \rightarrow \mrm{D}_{\mrm{qc}}(\sX)$. This functor preserves perfect complexes and thus defines a functor $f^*: \Perf(\sX) \rightarrow \Perf(\sY)$. If $f$ is furthermore representbale, proper, laft and is of finite tor-amplitude then the right adjoint preserves perfect complexes producing the functor
\[
f_*: \Perf(\sX) \rightarrow \Perf(\sY).
\]
We might as well define a morphism $f: \sX \rightarrow \sY$ of perfect derived stacks to be \textbf{quasiperfect} $f_*$ preserves perfect complexes; this class of morphisms evidently includes those which are representable, proper, laft and is of finite tor-amplitude. Fixing a base prestack $\sX$, we contemplate the following $\infty$-category of spans
\[
\Span(\mrm{dAlgStk}^{\mathrm{qcqs,perf}}_{\sX}, \mrm{all}, \mrm{quasiperfect}),
\]
whose objects are representable morphisms $\sY \rightarrow \sX$ where the domain is a perfect derived stack and morphisms are spans
\[
\begin{tikzcd}
& \sZ \ar[swap]{dl}{f} \ar{dr}{g}& \\
\sY & & \sY',
\end{tikzcd}
\]
where $f$ is any representable morphism and $g$ is a quasiperfect and representable. For a precise construction of this $\infty$-category see \cite{BarwickMackey}. 

\begin{constr}\label{constr:mackey} Let $\sX$ be perfect stack. As explained by Barwick in \cite[Example D]{BarwickMackey}, using the base change explained in \cite[Proposition 6.3.4.1]{SAG} the formation of perfect complexes assembles into a functor
\[
\Perf:\Span(\mrm{dAlgStk}^{\mathrm{qcqs,perf}}_{\sX}, \mrm{all}, \mrm{quasiperfect}) \rightarrow \Cat^{\mathrm{perf}}_{\infty}(\Perf(\sX)).
\]
Postcomposing with the functor of algebraic $K$-theory we get
\[
K \circ \Perf: \Span(\mrm{dAlgStk}^{\mathrm{qcqs,perf}}_{\sX}, \mrm{all}, \mrm{quasiperfect}) \rightarrow \Cat^{\mathrm{perf}}_{\infty}(\Perf(\sX)),
\]
with a coherent pushforward functoriality. 
\end{constr}

Now, we work with $\sX = BG$, where $G$ is a reductive group in characteristic zero. 

\begin{lem}\label{lem:spans} There is a natural factorization: 
\begin{equation}\label{eq:kbg}
\begin{tikzcd}
 & \Mod_{K(BG)} \ar{d}{\mrm{forget}} \\
 \Span(\mrm{dAlgStk}^{\mathrm{qcqs,perf}}_{BG}, \mrm{all}, \mrm{quasiperfect}) \ar[dashed]{ur} \ar{r}{K \circ \Perf} & \Spt.
\end{tikzcd}
\end{equation}
\end{lem}

\begin{proof} After the factorization  explained in~\eqref{eq:k-factors}, the only thing to verify is that for any quasiperfect map $f: \sX \rightarrow \sY$ the pushforward map $f_*:K(\sX) \rightarrow K(\sY)$ is a $K(BG)$-module map. This boils down to the projection formula, which is a consequence of the fact that $f$ is schematic and quasicompact (since it is a $BG$-morphism between two representable derived algebraic stacks which are assumed to be qc) and thus the projection formula holds by \cite[Chapter 3, Lemma 3.2.4]{GRderalg2}.
\end{proof}

By the functoriality of completion of modules along ideals explained in Section~\ref{sec:sag-completion} Lemma~\ref{lem:spans} gives us a functor:
\[
(-)^{\wedge}_{I_G} \circ K \circ \Perf: \Span(\mrm{dAlgStk}^{\mathrm{qcqs,perf}}_{BG}, \mrm{all}, \mrm{quasiperfect}) \rightarrow \Mod_{K(BG)} \xrightarrow{\text{forget}} \Spt.
\]
Now, replacing $\mrm{dAlgStk}^{\mathrm{perf}}_{BG}$ with the smaller, full subcategory, $\mrm{ANS}^{\mrm{clft}}_{BG}$, ANS-stacks equipped with a map to $BG$ such that $\mrm{\sX}^{\mrm{cl}} \rightarrow BG$ is finite type, our main theorem identifying $R^{\mrm{dAff}}K$ with the Atiyah-Segal completion of $K$-theory shows that
\begin{thm}[Pushforward structure]\label{thm:pushforward} The functor
\[
R^{\mrm{dAff}}K:(\mrm{ANS}^{\mrm{clft}}_{BG})^{\mathrm{op}} \rightarrow \Spt,
\]
admits a canonical extension along the faithful functor $\mrm{ANS}^{\mrm{clft}}_{BG} \rightarrow \Span(\mrm{ANS}^{\mrm{clft}}_{BG}, \mrm{all}, \mrm{quasiperfect})$:
\[
R^{\mrm{dAff}}K:\Span(\mrm{ANS}^{\mrm{clft}}_{BG}, \mrm{all}, \mrm{quasiperfect}) \rightarrow \Spt,
\]
such that, on $\mrm{dSch}^{\mrm{qcqs}} \cap \mrm{ANS}^{\mrm{clft}}_{BG} \subset  \mrm{ANS}^{\mrm{clft}}_{BG}$ (the subcategory of those $\sX \rightarrow BG$ where $\sX$ is a qcqs derived scheme), the pushforward functoriality agrees with the one induced by pushforward of perfect complexes. 
\end{thm}

We remark that Theorem~\ref{thm:pushforward} makes no reference to the Atiyah-Segal ideal or the Atiyah-Segal completion. It produces a coherent transfer structure on the right Kan-extended $K$-theory for a large class of derived stacks. As emphasized throughout the text, the right Kan extended theory does not have any \emph{a priori} to extend to a localizing invariant or something close to it. Without identifying it with the completion of a localizing invariant, such a coherent transfer structure would be very difficult to produce.

\subsection{Motivic filtrations} In this section we answer the question of how to endow the $K$-theory of a derived stacks in characteristic zero with a motivic filtration. 

\subsubsection{Motivic filtrations on derived algebraic spaces} We refer to \cite{bouis-mixed, kelly-saito} for general accounts of the motivic filtration in mixed characteristic situations and Appendix~\ref{app:mot-algpsc} for a record of its extension to derived algebraic spaces. Since we work over $\mathbb{Q}$, we only need to refer to the motivic filrations of \cite{e-morrow}. We define a motivic filtration on derived algebraic stacks over $\mathbb{Q}$-as follows.

\begin{constr}\label{constr:mot-dalgstk} On the category of derived algebraic stacks over $\mathbb{Q}$, we define $\mathrm{Fil}_{\mot}^{\star}K$ as the following pullback in $\mathrm{CAlg}(\mathrm{FilSpt})$-valued presheaves (where the terms on the right are given the constant filtration):
\begin{equation}\label{eq:motfilt-stk}
\begin{tikzcd}
\mrm{Fil}_{\mot}^{\star}K \ar{d} \ar{r} & R^{\mrm{dAff}}\mrm{Fil}_{\mot}^{\star}K \ar{d}\\
K \ar{r} & R^{\mrm{dAff}}K.
\end{tikzcd}
\end{equation}

%
\end{constr}

Since the motivic filtration on algebraic spaces is $\mathbb{N}$-indexed, we have equivalences:
\[
\colim_{\star \rightarrow - \infty} \mathrm{Fil}_{\mot}^{\star}K \xrightarrow{\simeq} \mathrm{Fil}^0_{\mot} \xrightarrow{\simeq} K;
\]
in other words the filtration is exhaustive. Taking graded pieces gives us the following candidate:

\begin{defn}\label{def:mot} Let $\sX$ be a derived algebraic stack over $\mathbb{Q}$. We define its \textbf{motivic cohomology in weight $j$} as
\[
\mathbb{Z}(j)^{\mrm{mot}}(\sX) := \mathrm{gr}^j_{\mot}K(\sX)[-2j].
\]
In other words, it is given by the right Kan extension of motivic cohomology from derived algebraic spaces. 
\end{defn}

We can justify that Definition~\ref{def:mot} is a good definition at least for certain ANS stacks, due to our main theorem. More precisely the motivic cohomology of $\sX$ are the graded pieces of a certain non-complete filtration on the $K$-theory of $\sX$ whose completion coincides with the Atiyah-Segal completion. To formulate this theorem we recall how completion along a filtration is defined. Let $\mathrm{Fil}_{\mot}^{< j}K := \mathrm{cofib}(\mathrm{Fil}_{\mot}^{\geq j}K \rightarrow K)$ so that we have an inverse system
\[
K \rightarrow \cdots \mathrm{Fil}_{\mot}^{< j}K \rightarrow \mathrm{Fil}_{\mot}^{< j-1}K \rightarrow \cdots \mathrm{Fil}_{\mot}^{<1}K \rightarrow 0.
\]
The \textbf{motivic completed $K$-theory} of a derived algebraic stack $\sX$ is set to be
\[
K(\sX)^{\wedge}_{\mot} := \lim_j \mathrm{Fil}_{\mot}^{< j}K(\sX). 
\]
The next lemma shows that motivic completion of $K$-theory is equivalent to the right Kan extended theory:
\begin{lem}\label{lem:motivic-completion} Let $\sX$ be a $\bbQ$-derived algebraic stack such that $\sX$ can be written as a colimit in Nisnevich sheaves: $\sX \simeq \colim^{\mrm{Nis}} X_{\alpha}$ where $X_{\alpha}$ is a derived algebraic space for which $\mrm{Fil}_{\mot}^{\star}K(X_{\alpha})$ are completed filtrations. Then there is a canonical equivalence
\[
K(\sX)^{\wedge}_{\mot} \rightarrow R^{\mrm{dAff}}K(\sX).
\]
\end{lem}

\begin{proof} By construction, we have a commutative diagram
\[
\begin{tikzcd}
K(\sX) \ar{r} \ar{d} & K(\sX)^{\wedge}_{\mot} \ar{d} \\
R^{\mrm{dAff}}K(\sX) \ar{r} &  \lim_j R^{\mrm{dAff}}\mathrm{Fil}_{\mot}^{< j}K(\sX).
\end{tikzcd}
\]
By construction, the right vertical map is an equivalence. We claim that, under the stated hypotheses, the bottom horizontal map is an equivalence. Indeed, 
\begin{eqnarray*}
R^{\mrm{dAff}}K(\sX)  & \simeq & \lim_{\alpha} K(X_{\alpha})\\
& \simeq &\lim_{\alpha}  \lim_j \mathrm{Fil}_{\mot}^{< j}K(X_{\alpha})\\
& \simeq &  \lim_j  \lim_{\alpha} \mathrm{Fil}_{\mot}^{< j}K(X_{\alpha})\\
& \simeq &  \lim_j R^{\mrm{dAff}}\mathrm{Fil}_{\mot}^{< j}K(\sX),
\end{eqnarray*}
where we have used the hypothesis on the $X_{\alpha}$'s for the second equivalence.

\end{proof}

\begin{thm}\label{thm:main-motivic} Fix $G$ a reductive group over a characteristic zero field $k$. Let $\sX \rightarrow BG$ be representable morphism where $\sX$ is derived algebraic stack, such that the morphism $\sX^{\mrm{cl}} \rightarrow BG$ is finite type. Assume further that $\sX$ is ANS. Then we have canonical equivalences:
\begin{equation}\label{eq:main-mot}
K(\sX)^{\wedge}_{I_G} \xrightarrow{\simeq} K(\sX)^{\wedge}_{\mot}.
\end{equation}
\end{thm}

\begin{proof} After Lemma~\ref{lem:motivic-completion}, it is enough to prove that such a $\sX$ can be written as a Nisnevich colimit of derived algebraic spaces whose motivic filtration is complete. Indeed, by Theorem~\ref{thm:deshmukh} we may assume that $\sX$ admits a Nisnevich surjection from a derived algebraic space whose classical locus is finite type over $k$. Hence, the result follows from Theorem~\ref{thm:deralgspc}.(3).
\end{proof}

\appendix
\section{Geometric version of HKR isomorphism (after Ben-Zvi--Nadler)}\label{sec:bzn}

The goal of this appendix is to prove Theorem~\ref{thm:Nadler-Ben-Zvi} (originally due to Ben-Zvi and Nadler), as well as fill in some technical details that were left out from Section~\ref{sec:as-hh}. 

\subsection{Completions of vector bundles} The following result is Proposition~\ref{prop:completion at 0 section} in the main text, whose proof we now supply. 

\begin{prop}\label{app:completion at 0 section} Let $\sX$ be a convergent derived prestack over $\mbb Q$ and let $\sE\in \mrm{D}_{\mrm{qc}}(\sX)_{\ge 0}$. Assume that for any $f\colon \Spec R \ra \sX$ the $\pi_0(R)$-module $\pi_0(f^*(\sE))$ is finitely generated. Then the natural map
	\[
	{\colim_n}^{\mrm{conv}} \mbf V_\sX^{\le n}(\sE) \ra \widehat{\mbf V}_\sX(\sE), 
	\]
	where $\colim^{\mrm{conv}}$ denotes the colimit in convergent prestacks, is an equivalence.
\end{prop}
\begin{rem} In formulation above we implicitly used the following fact: if $\mc Y\ra \sX$ is an affine map of prestacks and $\sX$ is convergent, then $\sY$ is convergent. In particular, $\widehat{\mbf V}_\sX(\sE)$ and $\mbf V^{\le n}_\sX(\sE)$ are convergent, if $\sX$ is.
\end{rem}	
\begin{proof} 
	By Remark \ref{rem:convergent prestacks} it is enough to show that the map induces an equivalence on $R$-points for all $R\in \CAlg^{\mrm{an},[0,\infty)}_{\mbb Q}$. Fixing $R$, this is enough to do check the equivalence on fibres over $\sX(R)$. Fixing a point $f\colon \Spec R\ra \sX$ and looking at the description of fibers of $(	{\colim}_n^{\mrm{conv}} \mbf V_\sX^{\le n}(\sE))(R)$ and $(\widehat{\mbf V}_\sX(\sE))(R)$ over it, we reduce to the affine case $\sX=\Spec R$ with $\sE$ given by the connective module $M:=f^*\sE\in \mrm{D}(R)_{\ge 0}$.

	By our assumption on $\sE$ there is a map $R^{\oplus m}\ra M$ inducing a surjection on $\pi_0$. It induces a closed embedding
	$$
	\mbf V_{\Spec R}(M) \ra \mbf V_{\Spec R}(R^{\oplus m})\simeq \Spec R\times  \mbb A^m.
	$$
	Moreover, if we take preimage of 0 under the resulting projection $\mbf V_{\Spec R}(M)\ra \mbb A^m$, the classical locus
	$(\mbf V_{\Spec R}(M)\times_{\mbb A^n} \{0\})^{\mrm{cl}}\hookrightarrow \mbf V_{\Spec R}(M)^{\mrm{cl}}$ is exactly identified with the zero section $\Spec \pi_0(R)\hookrightarrow \mbf V_{\Spec R}(M)^{\mrm{cl}}$. Thus, by Example \ref{ex:affine completion}, $\widehat{\mbf V}_{\Spec R}(M)$ agrees with derived completion of $\mbf V_{\Spec R}(M)$ along $\mbf V_{\Spec R}(M)\times_{\mbb A^n} \{0\}$. By Remark \ref{rem:pull-back of a formal completion} and Example \ref{ex:classical completion agrees with derived} we get an equivalence 
	$$
	\widehat{\mbf V}_{\Spec R}(M)\simeq \mbf V_{\Spec R}(M)\times_{\mbb A^m}(\mbb A^m)^\wedge_{0}\simeq \colim_n \Spec (\Sym^{\le n}_{\mbb Q}(\mbb Q^{\oplus m})\otimes_{\Sym_{\mbb Q}(\mbb Q^{\oplus m})} \Sym_R(M)).
	$$
	We will show that for each $s\ge 0$ the pro-systems 
	$$
	\{\tau_{\le s}(\Sym^{\le n}_{\mbb Q}(\mbb Q^{\oplus m})\otimes_{\Sym_{\mbb Q}(\mbb Q^{\oplus m})} \Sym_R(M))\}_{n\in \mbb N} \qquad \text{and} \qquad \{\tau_{\le s}(\Sym^{\le n}_R(M))\}_{n\in \mbb N}
	$$ 
	are equivalent; since we are only interested in values on the truncated algebras this will give the claim.
	
	The map from the left to the right is induced by scalar multiplication, and to construct the inverse it will be enough to show that the map 
	$$
	\tau_{\le s} \Sym_R(M) \ra \tau_{\le s}(\Sym^{\le n}_{\mbb Q}(\mbb Q^{\oplus m})\otimes_{\Sym_{\mbb Q}(\mbb Q^{\oplus m})} \Sym_R(M))
	$$
	naturally factors through $\tau_{\le s} \Sym^{\le N}_R(M)$ for $N \gg 0$. Note that by \eqref{eq:maps from graded truncation}, given an $\mbb N$-graded animated $R$-algebra $B^\star$ such that $B^m\simeq 0$ for $m> N$, there is a natural equivalence of spaces of graded algebra maps
	$$
	\mrm{Map}_{\mrm{gr}}(\Sym^\star_R(V), B^\star) \simeq \mrm{Map}_{\mrm{gr}}(\Sym^{\le N}_R(V), B^\star);
	$$
	in other words, any graded $R$-algebra map $\Sym^\star_R(V) \ra B$ canonically factors through $\Sym^{\le N}_R(V)$. Since the map above is naturally a graded map it remains to show that the $N$-th graded component of 
	$$\tau_{\le s}(\Sym^{\le n}_{\mbb Q}(\mbb Q^{\oplus m})\otimes_{\Sym^\star_{\mbb Q}(\mbb Q^{\oplus m})} \Sym^\star_R(M))\simeq \tau_{\le s}(\Sym^{\le n}_R(R^{\oplus m})\otimes_{\Sym^\star_R(R^{\oplus m})} \Sym^\star_R(M))$$ is 0 for $N\gg n$.

	Considering the cofiber of $R^{\oplus m} \ra M$ we get a fiber sequence 
	$$
	R^{\oplus m} \ra M \ra M'[1]
	$$
	for some connective $R$-module $M'$.  This endows $M$ with increasing 2-step filtration $F_\bullet (M)$ ($F_i(M)=0$ for $i<0$, $F_0(M):=R^{\oplus m}$ and $F_i(M)=M$ for $i>0$), which induces an increasing filtration on $\Sym_R(M)$ as $\Sym_R(R^{\oplus m})$-module. By \cite[Construction~25.2.5.4]{SAG} its associated graded is given by the tensor product
	$$
	\Sym_R(R^{\oplus m}) \otimes_R \Sym^\star_R(M'[1])
	$$
	with the filtration degree given by the degree on $\Sym^\star_A(M'[1])$.
	By \cite[Proposition~25.2.4.1]{SAG}, $\Sym^i_A(M'[1])$ is $i$-connective, thus the filtration above is exhaustive. As a result, after tensoring with $\Sym^{\le n}_R(R^{\oplus m})$ over $\Sym^\star_R(R^{\oplus m})$ we get an exhaustive filtration on $\Sym^{\le k}_R(R^{\oplus m})\otimes_{\Sym^\star_R(R^{\oplus m})}\Sym^\star_R(M)$ as graded $\Sym^{\le n}_R(R^{\oplus m})$-algebra with the associated graded $\Sym^{\le n}_R(R^{\oplus m}) \otimes_R \Sym^\star_R(M'[1])$ (and with the grading given by total grading). Since $\Sym^i_R(M'[1])$ is $i$-connective, we get that for a given $s\ge 0$ the $N$-th graded component of $\tau_{\le s}(\Sym^{\le n}_R(R^{\oplus m}) \otimes_R \Sym^\star_R(M'[1]))$ is 0 for $N\gg 0$ (more precisely $N>n-s$). Using long exact sequences of homotopy groups to pass back from the associated graded to the original object (or, equivalently, using the corresponding spectral sequence) we get that the same is true for the truncation 
	$$
	\tau_{\le s}(\Sym^{\le n}_R(R^{\oplus m})\otimes_{\Sym^\star_R(R^{\oplus m})} \Sym^\star_R(M)).
	$$
	
\end{proof}

\begin{rem}
	From the proof it also follows that the natural map 
	$$
	\widehat{\mbf V}_\sX(\sE) \ra \mbf V^\curlywedge_{\sX}(\sE)
	$$
	to the thin formal completion along the zero section is an isomorphism. Indeed, it is enough to check this after pulling back to every $\Spec R \ra \sX$, where it reduces to the fact that the map 
	$$
	(\mbb A^n)_0^\wedge \ra (\mbb A^n)_0^\curlywedge
	$$
	is an isomorphism (e.g. by Lemma \ref{lem:two completions coincide}).
\end{rem}

The next results are meant to assuage the reader of the worry that colimits in convergent derived prestacks could be destructive. The embedding $\mrm{dPStk}^{\mrm{conv}}\ra \mrm{dPStk}$ has a left adjoint which we denote $\sY \mapsto \sY^{\mrm{conv}}$. Via the identifications $\mrm{dPStk}^{\mrm{conv}} \simeq \mrm{Fun}(\CAlg^{[0,\infty)},\Spc)$, $\mrm{dPStk}\simeq \mrm{Fun}(\CAlg^{\mrm{cn}},\Spc)$ this functor is calculated by restricting $\sY\colon \CAlg^{\mrm{cn}} \ra \Spc$ to $\CAlg^{[0,\infty)}\subset \CAlg$. 

As noted in \cite[Chapter 3, 1.2.5]{GRderalg1} it is not true that the pull-back functor $$  D_{\mrm{qc}}(\sY) \ra  D_{\mrm{qc}}(\sY^{\mrm{conv}})
$$ 
induced by the unit map $\sY \ra \sY^{\mrm{conv}}$ is an equivalence. Nevertheless, we note that this is true for the full subcategories spanned by bounded below objects.
\begin{lem}\label{lem:functions on the convergencified stack}
	For any prestack $\mc Y$ the natural map $\sY \ra \sY^{\mrm{conv}}$ induces an equivalence
	$$
	  D_{\mrm{qc}}( \sY^{\mrm{conv}})_{>-\infty} \simeq  D_{\mrm{qc}}( \sY)_{>-\infty}.
	$$
Furthermore, it induces an equivalence
	$$
	\sO( \sY^{\mrm{conv}}) \simeq \sO(\sY).
	$$
\end{lem}
\begin{proof}
Following \cite[Chapter 2, Section 1.3]{GRderalg1} let us denote by $\tau^{\le n}(\sY)\in \mrm{dPStk}$ the left Kan extension of the restriction $\sY_{|\CAlg^{[0,n]}}$ (to $n$-truncated algebras) along the embedding $\CAlg^{[0,n]}\subset \CAlg^{\mrm{an}}$. For any prestack $\sY$ we claim that the natural map $\sY \ra \colim_{n} \tau^{\le n}(\sY)$ induces an equivalence
$$
 D_{\mrm{qc}}(\sY)_{>-\infty}\simeq  D_{\mrm{qc}}(\colim_{n} \tau^{\le n}(\sY))_{>-\infty}.
$$
Note that $ D_{\mrm{qc}}(\sY)_{>-\infty}=\colim_n  D_{\mrm{qc}}(\sY)_{\ge -n}$ and $ D_{\mrm{qc}}(\sY)_{\ge -n}\simeq  D_{\mrm{qc}}(\sY)_{\ge 0}$ via the shift $[n]$, so it is enough to show that 
$$
 D_{\mrm{qc}}(\sY)_{\ge 0}\simeq  D_{\mrm{qc}}(\colim_{n} \tau^{\le n}(\sY))_{\ge 0}.
$$

Since both sides send colimits to limits in the $\sY$-variable, it is enough to check this for $\sY=\Spec R$. In this case $\tau^{\le n}(\Spec R)$ is represented by $\Spec \tau_{\le n} R$ and we need to prove that the functor
$$
 D(R)_{\ge 0} \ra \lim_{n\ge 0}  D(\tau_{\le n}(R))_{\ge 0}
$$
sending $M$ to $\{M\otimes_R\tau_{\le n}(R)\}_n$ is an equivalence. Note that $D(A)$ is left $t$-complete for any $A\in \CAlg^{\mrm{cn}}$, and so $D(A)_{\ge 0}\simeq \lim_{m\ge 0} D(A)_{[0,m]}$ (with map induced by truncations $\{\tau_{\le m}\}$). We then have 
$$
\lim_{n\ge 0}  D(\tau_{\le n}(R))_{\ge 0} \simeq \lim_{n,m\ge 0}  D(\tau_{\le n}(R))_{[0,m]}\simeq \lim_{n\ge 0}  D(\tau_{\le n}(R))_{[0,n]}.
$$
Moreover, the corresponding functors $ D(R)_{\ge 0}\ra  D(\tau_{\le n}(R))_{[0,n]}$ are given by the truncations
$$
M\mapsto \tau_{\le n}(M\otimes_R \tau_{\le n} R)\simeq \tau_{\le n}M
$$
where the right equivalence holds since the fiber of the map $R\ra  \tau_{\le n} R$ is $(n+1)$-connective. This way (under the equivalence $D(R)_{[0,n]}\simeq D(\tau_{\le n}R)_{[0,n]}$ induced by $\tau_{\le n}$) the functor $ D(R)_{\ge 0}  \ra \lim_{n\ge 0}  D(\tau_{\le n}(R))_{[0,n]}$ can be identified with the left $t$-completion functor
$$
D(R)_{\ge 0}\simeq \lim_{n\ge 0} D(R)_{[0,n]}.
$$
It is an equivalence by left $t$-completeness of $D(R)$. Finally, for any $n\ge 0$ we have  $\tau^{\le n}(\sY)\simeq \tau^{\le n}(\sY^{\mrm{conv}})$ (since the restrictions of $\sY$ and $\sY^{\mrm{conv}}$ to $\CAlg^{<\infty}$ canonically agree) and so $$ D_{\mrm{qc}}( \sY^{\mrm{conv}})_{>-\infty} \simeq  D_{\mrm{qc}}( \sY)_{>-\infty}.$$

 The equivalence for $\sO(-)$ is proved similarly, by using that for $\sY=\Spec R$ one has
 $$
 R\simeq \lim_{n\ge 0} \tau_{\le n}R.
 $$

\end{proof}		
\begin{cor}\label{app:functions on formal completion at 0} Let $\sX$ be a prestack over $\mbb Q$. Let $\sE \in  D_{\mrm{qc}}(\sX)_{\ge 0}$ and assume that for any $f\colon \Spec R \ra \sX$ the $\pi_0(R)$-module $\pi_0(f^*(\sE))$ is finitely generated. Then the map $\colim_n \mbf V^{\le n}(\sE) \ra \widehat{\mbf V}(\sE)$ induces natural equivalences
	$$
	\mc O( \widehat{\mbf V}(\sE)) \simeq \mc O(\colim_n \mbf V^{\le n}(\sE)) \simeq \prod_{i=0}^{\infty}\Gamma (\sX, \Sym^i_{\mc O_{\sX}}\!(\sE))).
	$$
\end{cor}	

\begin{proof}	Functor $\sY \mapsto \sY^\mrm{conv}$ commutes with colimits and so 
	$$
(\colim_n \mbf V^{\le n}_\sX(\sE))^{\mrm{conv}}\simeq {\colim_n}^{\mrm{conv}} \mbf V^{\le n}_\sX(\sE)
$$
By Lemma \ref{lem:functions on the convergencified stack} we then get 
	$$
	\mc O({\colim_n}^{\mrm{conv}} \mbf V^{\le n}_\sX(\sE)) \simeq \mc O(\colim_n \mbf V^{\le n}_\sX(\sE)), 
	$$
	and the first isomorphism follows from Proposition \ref{app:completion at 0 section}.
	Since $\sY\mapsto \sO(\sY)\in  D(\mbb Q)$ sends colimits to limits we also get
	$$
	\mc O(\colim_n \mbf V^{\le n}_\sX(\sE))\simeq \lim_n \mc O(\mbf V_{\sX}^{\le n}(\sE))\simeq \lim_n(\oplus_{i=0}^{n} \Gamma(\sX, \Sym^i_{\mc O_{\mc X}}\!(\sE)))\simeq \prod_{i=0}^{\infty}\Gamma (\sX, \Sym^i_{\mc O_{\mc X}}\!(\sE))).
	$$
	
\end{proof}	

\begin{rem}\label{rem:for algebraic stack pi_0 L_X shifted is finitely generated}
	Let $\sX$ be a derived algebraic stack over an affine derived scheme $S$. Then for any $f\colon \Spec R\ra \sX$ the $\pi_0(R)$-module $\pi_0(f^*\mbb L_{\sX/S}[1])$ is finitely generated. Indeed, by \cite[Tag 03C4]{stacks} we can check finite generatedness flat locally. Let $p\colon U\ra \sX$ be a smooth cover by a 0-Artin stack $U$ and let
	$$
	\begin{tikzcd}
		Z \arrow[r,"q"]\arrow[d, "g"] & \Spec R \arrow[d, "f"]\\
		U\arrow[r,"p"]&\sX
	\end{tikzcd}	
	$$
	be the fiber product; since $\sX$ is algebraic $Z$ is a derived algebraic space. We have a fiber sequence 
	$$
	g^*\mbb L_{U/\sX} \ra g^*p^* \mbb L_{\sX/S}[1] \ra g^*\mbb L_{U/S}[1]
	$$ 
	where $\mbb L_{U/S}$ is connective (since $U$ is 0-Artin) and $\mbb L_{U/\sX}$ is a finite locally free sheaf on $U$ (by smoothness of $p$). Thus we get a surjection 
	$$
	\pi_0(g^*\mbb L_{U/\sX}) \twoheadrightarrow \pi_0(g^*p^* \mbb L_{\sX/S}[1])\simeq \pi_0(q^*f^*\mbb L_{\sX/S}[1])
	$$
	where $\pi_0(g^*\mbb L_{U/\sX})$ is locally finite free, which shows that $\pi_0(q^*f^*\mbb L_{\sX/S}[1])$ is locally finite.
\end{rem}

\begin{lem}[Nakayama lemma] \label{lem:conservativity of pull-back to zero section} Let $\sX$ be a prestack. Let $\sE \in  D_{\mrm{qc}}(\sX)_{\ge 0}$ and assume that for any $f\colon \Spec R \ra \sX$ the $\pi_0(R)$-module $\pi_0(f^*(\sE))$ is finitely generated. Let $z_\sX\colon \sX \ra \widehat{\mbf V}_{\sX}(\sE)$ be the zero section map. Then the pull-back functor
	$$
	z_\sX^*\colon  D_{\mrm{qc}}(\widehat{\mbf V}_{\sX}(\sE))_{<\infty} \ra  D_{\mrm{qc}}(\sX)_{<\infty}
	$$
	is conservative.
\end{lem}	
\begin{proof}
Assume $\sF\in   D_{\mrm{qc}}(\widehat{\mbf V}_{\sX}(\sE))_{<\infty}$ is such that $z_{\sX}^*\sF\simeq 0$ but $\mc F\not \simeq 0$. Then there exists a map $f\colon \Spec R\ra \widehat{\mbf V}_{\sX}(\sE)$ such that $f^*\sF\not\simeq 0$. Note that $f$ factors\footnote{Indeed, $\widehat{\mbf V}_{\Spec R}(f^*\sE)$ is identified with the fiber product $\Spec R\times_\sX \widehat{\mbf V}_{\sX}(\sE)$ where the map $\Spec R \ra \sX$ is given by composition of $f$ with the projection $p_\sX\colon \widehat{\mbf V}_{\sX}(\sE) \ra \sX$ and where $\widehat{\mbf V}_{\sX}(\sE)\ra \sX$ is the same projection.} through the natural map $\tilde f\colon \widehat{\mbf V}_{\Spec R}(f^*\sE) \ra \widehat{\mbf V}_{\sX}(\sE)$ and thus $\tilde f^*\sF$ is also non-zero. Since we have a commutative diagram
$$
\begin{tikzcd}
	\widehat{\mbf V}_{\Spec R}(f^*\sE) \arrow[r,"\tilde f"]&  \widehat{\mbf V}_{\sX}(\sE)\\
	\Spec R\arrow[u,"z_{\Spec R}"]\arrow[r,"p_\sX\circ f"]& \sX \arrow[u,"z_\sX"']
\end{tikzcd}	
$$
it will be enough to show the statement of the lemma for $\sX=\Spec R$ and a connective $R$-module $M:=f^*\sE$.

By Proposition \ref{app:completion at 0 section} and Lemma \ref{lem:functions on the convergencified stack} we have an equivalence
$$
 D_{\mrm{qc}}(\widehat{\mbf V}_{\Spec R}(M) )_{>-\infty} \xrightarrow{\sim } \lim_{n\ge 0} D_{\mrm{qc}}({\mbf V}^{\le n}_{\Spec R}(M) )_{>-\infty} \simeq  \lim_{n\ge 0} D(\Sym^{\le n}_R(M))_{>-\infty}
$$
and so given a non-zero $\sF\in  D_{\mrm{qc}}(\widehat{\mbf V}_{\Spec R}(M) )_{>-\infty}$ the pull-back $z_n^*\mc F\in  D(\Sym^{\le n}_R(M))_{>-\infty}$  is also non-zero for some $n\ge 0$ (here $z_n\colon {\mbf V}^{\le n}_{\Spec R}(M) \ra \widehat{\mbf V}_{\Spec R}(M)$ is the natural map). To arrive at a contradiction we need to show that $z_0^*\mc F\simeq z_{\Spec R}^*\mc F$ is also non-zero. 

Note that $\Sym^{\le n}_R(M)$ as a module over itself has a finite filtration by $\Sym^{[n-i,n]}_R(M)\subset \Sym^{\le n}_R(M)$ for $0\le i \le n$, such that $\Sym^{\le n}_R(M)$-action on the associated graded factors through the projection to the 0-th graded component $\Sym^{\le n}_R(M)\ra R$. This gives a finite filtration on $z_n^*\mc F$ with associated graded terms of the form $z_{\Spec R}^*\sF \otimes_R \Sym^i_R(M)$, from which we see that if $z_{\Spec R}^*\mc F$ is zero then so should be $z_n^*\mc F$.

\end{proof}

\subsection{Descent for cotangent complex}

Below we will also need a descent statement for the exterior powers of the cotangent complex.

\begin{prop}\label{prop:etale-descen-cot} Let $\sX$ be a derived geometric stack over an affine derived scheme $S$. Then for any $n$ the natural map
	\begin{equation}\label{eq:cotangent complex is right Kan extended}
	\Sym^n_{\mc O_\sX}\!(\mbb L_{\sX/S}[1]) \ra \lim_{x: T \ra \sX} x_* \Sym^n_{\mc O_T}(\mbb L_{T/S}[1])
	\end{equation}
	with $T\in \mrm{dAff}_{S}$ is an equivalence.
\end{prop}	
\begin{proof}
	Note that the functor  $(x\colon \Spec R \ra \sX) \mapsto x_* \Sym^n_{\mc O_T}(\mbb L_{T/S}[1]) $ from $\mrm{dAff}_{/\sX}$ to $ D_{\mrm{qc}}(\sX)$ satisfies \'etale and, consequently, smooth descent. Indeed, having an \'etale cover $f\colon  T'\ra T$ we have $\mbb L_{T'/T}\simeq 0$, so $f^*\mbb L_{T'/S}\simeq \mbb L_{T/S}$ and, more generally, $f^*\Sym^n_{\mc O_{T}}(\mbb L_{T/S}[1])\simeq \Sym^n_{\mc O_{T'}}(\mbb L_{T'/S}[1])$. Then for the corresponding \v Cech cover $f_\bullet\colon T_\bullet'\ra T$ the map 
	$$
	\Sym^n_{\mc O_T}(\mbb L_{T/S}[1]) \ra \lim_{[\bullet]\in \Delta} f_{\bullet*}\Sym^n_{\mc O_{T'}}(\mbb L_{T'/S}[1]) \simeq \lim_{[\bullet]\in \Delta} f_{\bullet*} f^*_\bullet\Sym^n_{\mc O_T}(\mbb L_{T/S}[1])  
	$$
	is an equivalence by flat descent for quasi-coherent sheaves. 
	
	It is easy to see that if $\sX$ is 0-Artin the map in (\ref{eq:cotangent complex is right Kan extended}) is an equivalence: namely, if $\sX=\sqcup_i U_i$ with $s_i\colon U_i \ra \sX$ being the embedding of a connected component, we have
	$$
	\Sym^n_{\mc O_\sX}\!(\mbb L_{\sX/S}[1]) \simeq \prod_{i} (s_i)_*\Sym^n_{\mc O_{U_i}}(\mbb L_{U_i/S}[1]).
	$$ For general $\sX$, the right hand side of (\ref{eq:cotangent complex is right Kan extended}) is the value on the final object of the right Kan extension along $\mrm{dAff}_{/\sX}\subset \mrm{Fun}(\mrm{dAff}_{/\sX}^{\mrm{op}},\Spc)$ of the functor $(x\colon T \ra \sX) \mapsto x_* \Sym^n_{\mc O_T}(\mbb L_{T/S}[1]) $. By smooth descent we can replace the  final object by the geometric realization of the \v Cech nerve $q_\bullet\colon U_\bullet \rightarrow \sX$ of some smooth cover $U_0:=U\twoheadrightarrow \sX$ by a 0-Artin stack $U$. It then remains to show that the map
	$$
	\Sym^n_{\mc O_\sX}\!(\mbb L_{\sX/S}[1]) \ra \lim_{[\bullet]\in \Delta} q_{\bullet *}\Sym^n_{\mc O_{U_\bullet}}(\mbb L_{U_\bullet/S}[1])
	$$
	is an equivalence. This can be done as in \cite[Proposition 1.1.5]{kubrak2022hodge} replacing $\Lambda^n$ with $\Sym^n$ in the proof.
\end{proof}	
\subsection{Formal loops} \label{sec:appendix_formal loops} Let $R\in \CAlg_{\mbb Z}^{\mrm{an}}$ be a base connective ring.
For a (not necessarily connective) derived commutative $R$-algebra $C\in \mrm{DAlg}_{\mbb Q}$ let $\Spec C\in \mrm{dPStk}_{R}$ denote the functor corepresented by $C$:
$$(\Spec C)(A):= \Maps_{\CAlg_R^{\mrm{an}}}(C,A)$$ for $A\in \CAlg^{\mrm{an}}_{R}$. For any prestack $\sX\in \mrm{dPStk}_R$ one has a natural ``affinization map"
$$a_\sX\colon \sX \ra \Spec \mc O(\sX),$$
and any map to a derived affine scheme from $\sX$ factors through $s_{\sX}$.
If $\sX=\underline{K}$ is the constant $R$-prestack given by a homotopy type $K\in\mrm{Spc}$, then $$\mc O(\underline{K})\simeq \lim_K \mc O(\Spec R)\simeq \lim_K k \simeq C^*_{\mrm{sing}}(K,R),$$ thus we get a map
$$
a_{\underline{K}}\colon \underline{K} \ra \Spec(C^*_{\mrm{sing}}(K,R)).
$$
In the case $K\simeq S^1$ and $\mbb Q\subset R$ one can further identify\footnote{The reason is that the splitting $C^*_{\mrm{sing}}(S^1,R) \simeq R\oplus R[-1]$ as $R$-modules induces an equivalence $\Sym_RR[-1]\simeq C^*_{\mrm{sing}}(S^1,R)$, if $R$ is a $\mbb Q$-algebra. Then one computes $\Maps_{\CAlg_R^{\mrm{an}}}(\Sym_RR[-1],A)\simeq \Omega^\infty (A[1])\simeq (B\mbb G_a)(A)$. } (\cite[Lemma 3.13]{Ben_Zvi_2012})
$$\Spec(C^*_{\mrm{sing}}(S^1,R))\simeq B\mbb G_{a,R}$$

\begin{constr}\label{constr:exponential map}
Let $S:=\Spec R$ and assume that $R$ is a $\mbb Q$-algebra. For an affine derived $S$-scheme $T=\Spec A$ the affinization map $$\underline{S^1}\ra B\mbb G_{a,R}$$ induces an equivalence
$$
\mc LT:=  {\underline{\Maps}}_{/S}(\underline{S^1},T) \simeq \underline{\Maps}_{/S}(B\mbb G_{a,R},T).
$$
Following\footnote{More precisely, identifying $C^*_{\mrm{sing}}(S^1,R)\in \CAlg$ with the split square zero extension $R\oplus R[-1]$, they use obstruction theory in terms of cotangent complex.} \cite[Proposition 4.4]{Ben_Zvi_2012} we can further identify 
$$\underline{\Maps}_{/S}(B\mbb G_a,T)\simeq \mbf V_{T}(\mbb L_{T/S}[1]):=\Spec (\Sym_A(\mbb L_{T/S}[1])).
$$  This gives a derived version of Hochshild-Kostant-Rosenberg isomorphism\footnote{Both sides are derived affine schemes and after passing to global functions one gets an isomorphism $HH(A)\simeq \Sym_A^\star(\mbb L_A[1])\simeq \oplus_i(\Lambda^{i}_A\mbb L_A)[i].$}:
\begin{equation}\label{eq:HKR-isomorphism}
\mbf V_T(\mbb L_{T/S}[1]) \xrightarrow{\sim} \mc L_ST.
\end{equation}

Let now $\mc X$ be a geometric stack over $S$. Let $x\colon T=\Spec A\ra \mc X$ be a map from an affine derived scheme.   For each $n$ and $m$ we consider a composition 
$$ \mbf V^{\le m}_T(\mbb L_{T/S}[1]) \ra \mbf V_{T}(\mbb L_{T/S}[1]) \simeq \mc L_ST \ra \mc L_S\mc X\xrightarrow{p} \sX,$$
where $p\colon \mc L_S\mc X = X\times_{X\times_SX} X \ra X$. 
Since $\sX$ is geometric, its diagonal is affine, and thus so is $p$, as well as the whole composition above. The direct image of the structure sheaf under it is identified with $(x)_* \Sym^{\le m}_{\mc O_T}\!(\mbb L_{T/S}[1])$. Taking a limit over all $x$ we get a natural map of algebras in $ D_{\mrm{qc}}(\sX)$
\begin{equation*}
\mc O_\sX \ra	p_*\mc O_{\sL_S\sX}\ra \lim_{x: T \ra \sX}x_* \Sym^{\le m}_{\mc O_{T}}(\mbb L_{T/S}[1]),
\end{equation*}	
where $T$ in the limit runs over all derived affine schemes over $\sX$. Recall that the natural map 
$$
\Sym^{\le m}_{\mc O_{\mc X}}(\mbb L_{\mc X/S}[1]) \ra \lim_{x: T \ra \sX} (x)_*\Sym^{\le m}_{\mc O_{T}}(\mbb L_{T/S}[1])
$$
is an equivalence by Proposition \ref{prop:etale-descen-cot}.

This way, for each $m$ we get a natural map 
$$
p_*\mc O_{\sL_S\sX}\ra \Sym^{\le m}_{\mc O_{\mc X}}(\mbb L_{\mc X/S}[1]),
$$
which, by passing to $\Spec_{\sX}(-)$ and taking colimit over $m$, gives a map 
$$
{\colim_m} \mbf V^{\le m}_\sX(\mbb L_{\mc X}[1]) \ra \mc L_S\mc X;
$$
note that here we again used that $\sX$ is geometric to identify $\mc L_S\sX\simeq \Spec_{\sX}(p_*\sO_{\sX})$.

Since $\sL_S \sX$ is convergent (Remark \ref{rem:properties of the loop stack}) this map factors through the similar colimit in convergent prestacks, giving a map
\begin{equation}\label{eq:map to loops}
\exp_\sX\colon \widehat{\mbf V}_\sX(\mbb L_{\mc X/S}[1])\simeq {\colim_m}^{\mrm{conv}} \mbf V^{\le m}_\sX(\mbb L_{\mc X/S}[1]) \ra \mc L_S\mc X
\end{equation}
with the first identification provided by Proposition \ref{app:completion at 0 section} and Remark \ref{rem:for algebraic stack pi_0 L_X shifted is finitely generated}.

The pre-composition of $\exp_\sX$ with the zero section $z\colon \sX \ra  \widehat{\mbf V}_\sX(\mbb L_{\mc X/S}[1])$ is naturally identified with the embedding of constant loops $c_{\sX}\colon \sX \ra \mc L_S\mc X$. This gives a factorization\footnote{More precisely, completion of $\mbf V^{\le m}_\sX(\mbb L_{\mc X}[1])$ at the zero section does nothing, so completing both $\mbf V^{\le m}_\sX(\mbb L_{\mc X}[1])$ and $\mc L\mc X$  along $\sX$ 
gives a natural map $\mbf V^{\le m}_\sX(\mbb L_{\mc X}[1])\ra  \widehat{\sL}\sX$ for any $m$. } of (\ref{eq:map to loops}) through the formal completion $\widehat{\mc L}_S \mc X$, producing a map 
$$
\exp_\sX\colon \widehat{\mbf V}_\sX(\mbb L_{\mc X/S}[1]) \ra \widehat{\mc L}_S\mc X.
$$
\end{constr}	

The goal of this subsection is to prove the following:

\begin{thm}\label{thm:Nadler-Ben-Zvi}
Let $\sX$ be a derived geometric stack over a derived affine $\mbb Q$-scheme $S$. Then the map
$$\exp_\sX\colon \widehat{\mbf V}_\sX(\mbb L_{\mc X/S}[1]) \ra \widehat{\mc L}_S\mc X$$
is an equivalence.
\end{thm}	

\begin{proof}

	Prestacks $\widehat{\mbf V}_\sX(\mbb L_{\mc X/S}[1])$ and  $ \widehat{\mc L}_S\mc X$ are completions of derived algebraic stacks $\mc L_S\mc X$ and $\mbf V_\sX(\mbb L_{\mc X/S}[1])$ and so admit corepresentable deformation theory in the sense of \cite[Chapter 1, Definition 7.1.2]{GRderalg2} (see also \cite[Chapter 1, Proposition 7.4.2]{GRderalg2} and Remark \ref{rem:relative cotangent complex of formal completion}).

	  Following Ben-Zvi--Nadler (\cite[Section 6.3.2]{Ben_Zvi_2012}), the key computation is that the relative cotangent complex $$\mbb L_{\exp_{\sX}}:=\mbb L_{\widehat{\mbf V}_\sX(\mbb L_{\mc X/S}[1])/\widehat{\mc L}_S\mc X}$$ vanishes. Let us explain how this vanishing implies the statement of the theorem. By Remark \ref{rem:properties of completion}\ref{rem:laftnes and convergence of formal completion} both sides are convergent stacks, so it is enough to show that $\exp_\sX$ induces an equivalence on $R$-points for bounded $R$. 
	  
	  Recall the description of the formal completion in Remark \ref{rem:properties of completion}\ref{item:lurie's completion}. For a nilpotent ideal $I\subset \pi_0(R)$ denote $F_{I,R}\subset \widehat{\mbf V}_\sX(\mbb L_{\mc X/S}[1])(R)$ the subspace given by those $R$-points which send $\Spec (\pi_0(R)/I)$ to $\sX\hookrightarrow {\mbf V}_\sX(\mbb L_{\mc X/S}[1])$. Similarly, let $G_{I,R}\subset \widehat{\mc L}_S\mc X(R)$ be the subspace given by $R$-points which send $\Spec (\pi_0(R)/I)$ to $\sX\hookrightarrow {\mc L}_S\mc X$. By the Definition \ref{defn:de Rham prestack and completion} of (absolute) formal completion the natural maps of spaces
	  $$
	  \left(\colim_{I\subset \pi_0(R)}F_{I,R} \right)\longrightarrow \widehat{\mbf V}_\sX(\mbb L_{\mc X/S}[1])(R), \qquad \left(\colim_{I\subset \pi_0(R)}G_{I,R} \right)\longrightarrow \widehat{\mc L}_S\mc X(R)
	  $$
	  are equivalences. Thus it would be enough to show that the natural map 
	  $
	  F_{I,R} \ra G_{I,R}
	  $
	  is an equivalence for all nilpotent ideals $I\subset \pi_0(R)$ and bounded rings $R$. By the argument in \cite[Chapter 1, Proposition 5.5.3]{GRderalg2}, any such $R$ is obtained 
	  from $\pi_0(R)/I$ via a \textit{finite} amount of square-zero extensions:
	  $$
	  R=R_n \ra R_{n-1} \ra \ldots \ra R_0=\pi_0(R)/I.
	  $$
	  We will argue by induction on $i$ in $R_i$. For $i=0$ we have a diagram of isomorphisms
	$$
	\begin{tikzcd}
		 F_{I,R_0}\arrow[rr, equal, "\alpha"]&&G_{I,R_0}\\
		 &\sX(R_0)\arrow[ur,"c"', equal]\arrow[ul,"z",equal]&
	\end{tikzcd}	
	$$
Then, assuming the map $F_{I,R_i} \ra G_{I,R_i}$ is an equivalence we need to show that so is $F_{I,R_{i+1}} \ra G_{I,R_{i+1}}$. By deformation theory \cite[Chapter 1, Lemma 10.2.6]{GRderalg2}, the fiber of the natural map
$$
F_{I,R_{i+1}} \ra F_{I,R_{i}}
$$
over an $R_i$-point $s\colon \Spec R_{i}\ra \widehat{\mbf V}_\sX(\mbb L_{\mc X/S}[1])$ (sending $\Spec (\pi_0(R)/I)\hookrightarrow \Spec R_i$ to $\sX$) is given by the space of nil-homotopies of the composition
$$
s^*\mbb L_{\widehat{\mbf V}_\sX(\mbb L_{\mc X/S}[1])} \ra \mbb L_{\Spec R_{i}/S} \ra \sI_i[1]
$$
where the second map is the one defining the square-zero extension $R_{i+1}\ra R_i$ (see \cite[Chapter 1, Section 5.1]{GRderalg2}, $\sI_i$ is the ``ideal" of the square-zero extension). Similarly, the fibers of $G_{I,R_{i+1}} \ra G_{I,R_{i}}$ are described in terms of the cotangent complex of $\widehat{\mc L}_S \mc X$. 

Now, given a point $s\in F_{I,R_i}\subset \widehat{\mbf V}_\sX(\mbb L_{\mc X}[1])(R_i)$ we have a fiber sequence
$$
(\exp_\sX\circ s)^* \mbb L_{\widehat{\sL}_S\sX} \ra s^*\mbb L_{\widehat{\mbf V}_\sX(\mbb L_{\mc X/S}[1])} \ra s^* \mbb L_{\exp_\sX}.
$$
Since the third term vanishes, by the above we get that $\exp_\sX$ induces an isomorphism between the fiber of $F_{I,R_{i+1}}$ over $s\in F_{I,R_i}$ and the fiber of $G_{I, R_{i+1}}$ over $\exp_\sX\circ s \in G_{I, R_{i}}$ correspondingly. Since this is true for any $s\in F_{I,R_i}$, we get that 
$$
\exp_\sX\colon F_{I,R_{i+1}} \ra G_{I,R_{i+1}}
$$
is an equivalence.

It remains to show that $\mbb L_{\exp_\sX}\simeq 0$. 
By Lemma \ref{lem:conservativity of pull-back to zero section} it is enough to show this after restricting to the zero section $z_{\sX}\colon \sX \ra  \widehat{\mbf V}_\sX(\mbb L_{\mc X/S}[1])$, which is equivalent to
$$ 
\exp_{\sX}^*\colon c^*_{\sX}\mbb L_{\widehat{\mc L}_S\mc X/\sX}\ra z^*_{\sX}\mbb L_{\widehat{\mbf V}_\sX(\mbb L_{\mc X/S}[1])/\sX}
$$
being a quasi-isomorphism (here we consider relative cotangent complexes with respect to projections $q_\sX\colon \widehat{\mbf V}_\sX(\mbb L_{\mc X/S}[1])\ra \sX$ and $p_{\sX}\colon \widehat{\mc L}_S\mc X\ra \sX$).

Let $x\colon T\ra \sX$ be a map where $T$ is a derived affine scheme and let $c_{\sX}\colon \sX \ra \sL_S\sX$ and $c_T\colon T\ra \sL_ST$ denote the embeddings of constant loops. We have a commutative diagram
$$
\begin{tikzcd}
	\sL_S\sX \arrow[d, "p_\sX", shift left =1]& \sL_ST\arrow[l, "\sL(x)"'] \arrow[d, "p_T", shift left =1]\\
	\sX\arrow[u,"c_{\sX}", shift left =1] & T\arrow[u, "c_T", shift left =1]\arrow[l,"x"']
\end{tikzcd}	
$$
which induces a natural map 
$$
\sL(x)^*\colon c^*_\sX \mbb L_{\sL_S\sX/\sX}\ra x_* c_T^* \mbb L_{\sL_ST/T}.
$$
By the argument of \cite[Lemma 6.7]{Ben_Zvi_2012} $c^*_\sX \mbb L_{\sL_S\sX/\sX}$ is naturally identified with $\mbb L_{\sX/S}[1]$, further identifying the map above with the pull-back map
$$
\mbb L_{\sX/S}[1]\ra x_*\mbb L_{T/S}[1].
$$
Similarly, we have a commutative square
$$
\begin{tikzcd}
	\widehat{\mbf V}_{\sX}(\mbb L_{\sX/S}[1]) \arrow[d, "q_\sX", shift left =1]& \widehat{\mbf V}_{T}(\mbb L_{T/S}[1]) \arrow[l, "\widehat{\mbf V} (x)"'] \arrow[d, "q_T", shift left =1]\\
	\sX\arrow[u,"z_{\sX}", shift left =1]& T\arrow[u, "z_T", shift left =1]\arrow[l,"x"']
\end{tikzcd}	
$$
where $\widehat{\mbf V} (x)$ is induced by the map $\mbb L_{T/S}[1] \ra  x^*\mbb L_{\sX/S}[1]$.  It induces the corresponding map
$$
\widehat{\mbf V} (x)^*\colon z^*_\sX \mbb L_{\widehat{\mbf V}_{\sX}(\mbb L_\sX[1])/\sX}\ra x_* z_T^* \mbb L_{\widehat{\mbf V}_{T}(\mbb L_{T}[1])/T},
$$
which via Remark \ref{rem:cotangent complex of a vector bundle along the zero section} can also be naturally identified with the pull-back map 
$$
\mbb L_{\sX/S}[1]\ra x_*\mbb L_{T/S}[1].
$$
Finally, we also have a commutative square
$$
\begin{tikzcd}
	\widehat{\mbf V}_{\sX}(\mbb L_{\sX/S}[1]) \arrow[d,"\exp_{\sX}"']&  \widehat{\mbf V}_{T}(\mbb L_{T/S}[1]) \arrow[l]\arrow[d, "\exp_{T}"]\\
\sL_S\sX&\sL_ST\arrow[l] 
\end{tikzcd}
$$
 lying under and above $\sX\leftarrow T$. Recall that essentially by construction, $\exp_{T}$ is an equivalence. Together with the maps above, and taking a limit over all $x\colon T\ra \sX$ we get a commutative square
 $$
 \begin{tikzcd}
 	c^*_\sX \mbb L_{\sL_S\sX/\sX}\arrow[d, "\exp_\sX^*"']\arrow[rr, "\lim \sL(x)^*"] &&\underset{x:T\ra \sX}{\lim}(x_* c_T^* \mbb L_{\sL_ST/T}) \arrow[d, "\lim \exp_T^*"]\\
 	z^*_\sX \mbb L_{\widehat{\mbf V}_{\sX}(\mbb L_{\sX/S}[1])/\sX} \arrow[rr, "\lim \widehat{\mbf V}_{\sX}(x)^*"] && \underset{x:T\ra \sX}{\lim} (x_* z_T^* \mbb L_{\widehat{\mbf V}_{T}(\mbb L_{T/S}[1])/T})
 \end{tikzcd}	
 $$
where the right vertical arrow is an isomorphism since $\exp_T$ is. Moreover, by the above discussion both horizontal arrows are identified with the natural map
$$
\mbb L_{\sX/S}[1]\ra  \underset{x:T\ra \sX}{\lim}  x_*\mbb L_{T/S}[1],
$$
which is an equivalence by Proposition \ref{prop:etale-descen-cot}. Thus $\exp_\sX^*$ is also an equivalence.
\end{proof}

\begin{rem}\label{rem:relative cotangent complex of formal completion}Let $f\colon \sX \ra \sY$ be a map of convergent derived prestacks and assume that $\sY$ admits corepresentable deformation theory in the sense of \cite[Chapter 1, Definition 7.1.2]{GRderalg2}. Then the derived completion $\sY^\wedge_f$ also admits corepresentable deformation theory. Indeed, $\sY^\wedge_f$ is also convergent, and the conditions on admitting cotangent complex (\cite[Chapter 1, Definitions 4.1.4]{GRderalg2}) and being infinitesimally cohesive (\cite[Chapter 1, Definition 6.1.2]{GRderalg2}) reduce to the ones for $\sY$. Moreover, the relative cotangent complex $\mbb L_{\sY^\wedge_f/\sY}$ for the natural map $\sY^\wedge_f\ra \sY$ is 0; in particular,
	$$
	f^*\mbb L_{\sY^\wedge_f} \simeq f^*\mbb L_{\sY}
	$$
(since the cofiber of this map is given by $f^*\mbb L_{\sY^\wedge_f/\sY}$).
\end{rem}	

\begin{rem}\label{rem:cotangent complex of a vector bundle along the zero section}Let $\sX$ be a derived prestack admitting a cotangent complex, let $\sE\in  D_{\mrm{qc}}(\sX)_{\ge 0}$ and consider the corresponding vector bundle $\mbf V_{\sX}(\sE)$, with the zero section and projection maps $z\colon \sX \ra \mbf V_{\sX}(\sE)$ and $p\colon \mbf V_{\sX}(\sE) \ra \sX$. We claim that there is a natural identification
	$$z^*\mbb L_{\mbf V_{\sX}(\sE)/\sX}\simeq \sE .$$
	Indeed, for any map $s\colon \Spec R\ra \sX$ and any $M\in  D(R)_{\ge 0}$ the space of maps $\Spec (R \oplus M) \ra \mbf V_{\sX}(\sE)$ extending $s$ is identified with the space of $R$-algebra maps $\Sym_R (s^*\sE) \ra R\oplus M$ which commute with the natural projections to $R$ on both sides. By the universal property of $\Sym_R$ the latter is naturally identified with $\Hom_R(s^*\sE, M)$, and so $z^*\mbb L_{\mbf V_{\sX}(\sE)}$ and $\sE$ corepresent the same functor. By Remark \ref{rem:relative cotangent complex of formal completion} we also get a formula for the relative cotangent complex of the formal completion
	$$
	z^*\mbb L_{\widehat{\mbf V}_{\sX}(\sE)/\sX}\simeq \sE.
	$$
\end{rem}	

\begin{lem}\label{lem:two completions of loops are the same}
	Let $\sX$ be a derived algebraic stack over an affine $\mbb Q$-scheme $S$. Let $\widehat{\sL}_S\sX$ and $\sL^{\! \curlywedge}_{S}\sX$ be the absolute and thin formal completions of the relative loop stack $\sL_S\sX$ along constant loops. Then the natural map 
	$$
	\widehat{\sL}_S\sX \ra \sL^{\! \curlywedge}_{S}\sX
	$$
	is an isomorphism.
\end{lem}	

\begin{proof}
	This follows from Lemma \ref{lem:two completions coincide}. Indeed, by \cite[tag 04XS]{stacks} the map of classical stacks $\sX^{\mrm{cl}} \ra (\sX\times_S\sX)^{\mrm{cl}}$ is of locally of finite type, and so by \cite[tag 0CMG]{stacks} the second diagonal map $c_\sX^{\mrm{cl}}\colon \sX^{\mrm{cl}} \ra (\sL_S\sX)^{\mrm{cl}}$ is locally of finite presentation.
\end{proof}	

The following lemma shows that the completion at constant loops passes along representable maps.
\begin{lem}[{\cite[Proposition 2.1.20]{Chen_2020}}]\label{lem:completion passes along representable maps} Let $S$ be a derived affine scheme. Let $f\colon \sX\ra \sY$ be a map of derived algebraic $S$-stacks representable in derived algebraic spaces. Then the commutative square
$$
	\begin{tikzcd}
	\widehat{\sL}_S\sX \arrow[r]\arrow[d] &	\sL_S\sX  \arrow[d] \\
\widehat{\sL}_S\sY\arrow[r]& \sL_S\sY
	\end{tikzcd}
$$	
is a fiber square.
\end{lem}	

\begin{proof} To simplify notation we let $\sL\sX:=\sL_S\sX$, and similary for the completion.
By Remark \ref{rem:pull-back of a formal completion} fiber product $\sL \sX\times_{\sL\sY}\widehat{\sL}\sY $ is identified with the formal completion of $\sL\sX$ along the map $\sL \sX\times_{\sL\sY} \sY \ra \sL\sX$. Thus by Remark \ref{rem:properties of completion}\ref{item:lurie's completion} it will be enough to show that the natural map
$$
\sX \ra  \sL \sX\times_{\sL\sY} \sY
$$
is an equivalence when restricted to classical rings. For this we need to check that the map of classical stacks
\begin{equation}\label{eq:map of inertia stacks}
\sX^{\mrm{cl}} \ra \mc I({\sX^{\mrm{cl}}}) \times_{\mc I({\sY^{\mrm{cl}}})} \sY^{\mrm{cl}}
\end{equation}
is an isomorphism. Here, for a classical $S^{\mrm{cl}}$-stack $\sZ$, $\sI\sZ:= \sZ \times_{\sZ\times_{S^{\mrm{cl}}}\sZ}\sZ$ denotes the relative inertia stack of $\sX$. 
The projection $\mc I({\sX^{\mrm{cl}}})\ra \sX^{\mrm{cl}}$ induces a projection $\mc I({\sX^{\mrm{cl}}}) \times_{\mc I({\sY^{\mrm{cl}}})} \sY^{\mrm{cl}} \ra \sX^{\mrm{cl}}$ whose composition with (\ref{eq:map of inertia stacks}) is the identity. It is enough to show that (\ref{eq:map of inertia stacks}) induces an equivalence on fibers over all maps $x\colon \Spec R \ra \sX^{\mrm{cl}}$ from (classical) affine schemes. The map between fibers is identified with the map of algebraic spaces
$$
\Spec R \ra \mrm{Aut}_{\sX^{\mrm{cl}}}(x)\times_{\mrm{Aut}_{\sY^{\mrm{cl}}}(f^{\mrm{cl}}(x)} \Spec R.
$$
where $\mrm{Aut}_{\sX^{\mrm{cl}}}(x):= \Spec R\times_{\sX^{\mrm{cl}}} \Spec R$ (with maps given by $x$) is the automorphism group (algebraic space) of $x$, the map $\Spec R\ra \mrm{Aut}_{\sY^{\mrm{cl}}}(f^{\mrm{cl}}(x))$ is given by the unit element, and the map to the fiber product is induced by the analogous map $\Spec R \ra \mrm{Aut}_{\sX^{\mrm{cl}}}(x)$ and the identity on $\Spec R$.
Note that $f^{\mrm{cl}}\colon \sX^{\mrm{cl}}\ra \sY^{\mrm{cl}}$ is represented by classical algebraic spaces, and so by \cite[tag 04YY]{stacks} the kernel of the map $\mrm{Aut}_{\sX^{\mrm{cl}}}(x)\ra {\mrm{Aut}_{\sY^{\mrm{cl}}}(f^{\mrm{cl}}(x))}$ is trivial. In particular, the preimage of the identity map is just the identity map itself, so the map above is an isomorphism.
\end{proof}	
\section{Derived Deshmukh's theorem}\label{sec:deshmukh}

The goal of this appendix is to recall (and extend to the derived setting) the theorem of Deshmukh.

\begin{thm}[{\cite[Theorem 1.2(1)]{deshmukh}}]\label{thm:deshmukh}
	Let $\sX$ be a quasi-compact quasi-separated derived algebraic stack with separated diagonal. Then there is a smooth-Nisnevich surjection $Y\ra \sX$, with $Y$ being a derived affine scheme.
\end{thm}	

In fact the epimorphism in the construction will be a \textit{smooth-Nisnevich covering} in the following sense:
\begin{defn}[{\cite[Definition 1.1]{deshmukh}}]
	We say that a representable map of derived algebraic stacks  $\sX \ra \sY$ is a \textbf{smooth-Nisnevich covering}, if  \begin{itemize}
		\item it is smooth;
		\item the map 
		$$
		\pi_0(\sX(K)) \ra \pi_0(\sY(K))
		$$
		is surjective for any field $K$.
	\end{itemize}	

\noindent It is clear that smooth-Nisnevich coverings are closed under composition.
\end{defn}	
We have the following lemma:
\begin{lem} Smooth-Nisnevich coverings of derived algebraic stacks are Nisnevich surjections.
\end{lem}	
\begin{proof}
	Let $f\colon \sX \ra \sY$ be a smooth-Nisnevich covering. We need to show that for an $R$-point $y\in \sY(R)$ for $R\in \CAlg^{\mrm{an}}$ there exists a Nisnevich cover $R\ra R'$ and a point $x\in \sX(R')$ lifting $y$. Since both $\sX$ and $\sY$ are convergent it is enough to assume that $R\in \CAlg^{\mrm{an},[0,\infty)}$.
	
	 First, let us assume $R\in \CAlg^{\mrm{cl}}$. Note that \cite[Corollary 2.7]{deshmukh} shows that the map 
	$$
	\sX^{\mrm{cl}}(\sO)\ra \sY^{\mrm{cl}}(\sO)
	$$
	is a surjection on connected components if $\sO$ is a Henselian ring. Since $f$ is smooth it is, in particular, locally finitely presented. Setting $\mc O$ to be the henselizations of points of $\Spec R$, it follows that for any point $y\in \sY^{\mrm{cl}}(R)$ there exists a Nisnevich covering $R\ra R'$ and a point $x\in \sX^{\mrm{cl}}(R')$ lifting $y$.
	
	Now, given $A\in \CAlg^{\mrm{an},[0,\infty)}$ with $\pi_0(A)\simeq R$ and an $A$-point $y\in \sY(A)$ we can lift the corresponding point $\pi_0(y)\in \sY(R)$ to a point $x\in \sX(R')$. Since $R\ra R'$ is \'etale we can uniquely extend it to an \'etale map $A \ra A'$, which is then also a Nisnevich cover. Finally, since $f$ is smooth, the deformation theory for lifting a map along $f$ is unobstructed\footnote{Here we implicitly use that $A'$ (being flat over $A$) also belongs to $\CAlg^{[0,\infty)}$, and that it can be obtained from $R'\simeq \pi_0(A')$ by a finite set of square-zero extensions, as in \cite[Proposition 5.5.3]{GRderalg2}.}, and so $x\in \sX(R')$ extends to an $A'$-point $x'\in \sX(A')$ lifting $y\in \sY(A)$.
\end{proof}	

To construct the required smooth-Nisnevich covering we will need the following:
\begin{constr}
	Let $f\colon \sX\ra \sY$ be a representable morphism of derived algebraic stacks whose diagonal is separated.
	
	Consider the $n$-fold fiber product $(\sX/\sY)^n:=\sX\times_\sY\ldots\times_\sY \sX$. The $S=\Spec R$-point of $(\sX/\sY)^n$ is given by an $S$-point $S\ra \sY$ of $\sY$ and $n$-sections $x_1,\ldots,x_n$ of the induced map $\sX\times_\sY S \ra S$. One defines the \textit{stack of $n$-section} $\mrm{Sec}_n(\sX/\sY)$ as the substack of $(\sX/\sY)^n$ where the classical loci of points $x_1,\ldots,x_n\colon S\ra \sX\times_\sY S$ do not intersect; since the diagonal of $f$ is separated this is in fact a Zariski open substack. The symmetric group $\Sigma_n$ acts naturally on $(\sX/\sY)^n$, as well as $\mrm{Sec}_n(\sX/\sY)\subset (\sX/\sY)^n$. One defines the \textit{stack of finite \'etale substacks of degree $n$} as the quotient 
	$$
	\mrm{Et}_n(\sX/\sY):=[\mrm{Sec}_n(\sX/\sY)/\Sigma_n]
	$$ 
	in \'etale topology. The natural projection $\mrm{Et}_n(\sX/\sY)\ra \sY$ is representable in derived algebraic spaces: indeed, it is clear for $\mrm{Sec}_n(\sX/\sY)\ra \sY$, so $\mrm{Et}_n(\sX/\sY)\ra \sY$ is at least represented in derived Deligne-Mumford stacks. But then the action of $\Sigma_n$ on the fibers of the map $\mrm{Sec}_n(\sX/\sY)^{\mrm{cl}}\ra \sY^\mrm{cl}$ is free, which shows that the quotient by $\Sigma_n$ is in fact represented in derived algebraic spaces. It is also clear that the map $\mrm{Et}_n(\sX/\sY)\ra \sY$ is smooth if $f$ is. 
	
	If $\sX$ is a derived algebraic space, then so is $\mrm{Sec}_n(\sX/\sY)$; consequently $\mrm{Et}_n(\sX/\sY)$ is a derived Deligne-Mumford stack. We warn the reader that the action of $\Sigma_n$ is free on $\mrm{Sec}_n(\sX/\sY)$ only relative to $\sY$, and not globally, so $\mrm{Et}_n(\sX/\sY)$ is not necessarily an algebraic space.\footnote{As an example consider the map $f\colon [\mbb A^1/\mbb G_m] \ra B\mbb G_m$. Then $\mrm{Sec}_2(f)$ is identified with $[(\mbb A^2\setminus \Delta(\mbb A^1))/\mbb G_m]$. The $\Sigma_2$-action on the latter sends $(x,y)$ to $(y,x)$, but note that if $x=-y$, then these points are identified via multiplication by $-1\in \mbb G_m$, and so define the same point in the quotient.} If $\sX$ and $\sY$ are qcqs then both $\mrm{Sec}_n(\sX/\sY)$ and $\mrm{Et}_n(\sX/\sY)$ are also qcqs.
\end{constr}

\textit{Proof of Theorem \ref{thm:deshmukh}:} We will construct a smooth-Nisnevich covering $Y\ra \sX$ by a countable union of qcqs algebraic spaces, by Knutson-Lurie theorem (Theorem \ref{thm:cover}) we then can refine this to a covering by countable union of derived affine schemes. Now let $p\colon U\ra \sX$ be a smooth atlas where $U$ is affine and consider the map $\sqcup_n \mrm{Et}_n(U/\sX) \ra \sX$. Since $p$ is smooth, it is an \'etale surjection, so any map $x\colon \Spec k\ra \sX$ lifts to a map
$$
\begin{tikzcd}
	\Spec \ell \arrow[d]\arrow[r,dashrightarrow]&U\arrow[d]\\
	\Spec k \arrow[r]&\sX
\end{tikzcd}	
$$
from a finite extension $\ell/k$. But a map $\Spec \ell \ra U$ lifting $x$ gives a section $\Spec k \ra \mrm{Et}_n(U/\sX)$ with $n:=\deg(\ell/k)$ lying over $x$. This way we get that $\sqcup_n \mrm{Et}_n(U/\sX) \ra \sX$ is a smooth-Nisnevich surjection. Unfortunately, $ \mrm{Et}_n(U/\sX)$ are not necessarily derived algebraic spaces, but could possibly be derived Deligne-Mumford stacks. However, this is easy to fix, following the argument in \cite[Th\'{e}or\`{e}me 6.1]{LMB} for classical stacks; namely for each $n$ consider the natural embedding $\Sigma_n\hookrightarrow GL_n$ and define $Z_n:=[(\mrm{Sec}_n(U/\sX)\times GL_n)/\Sigma_n]$ where the action of $\Sigma_n$ is diagonal. Each $Z_n$ is a derived algebraic space (since the action of $\Sigma_n$ is now free), and there is a natural map $Z_n\ra \sqcup_n \mrm{Et}_n(U/\sX)$ which is a $GL_n$-torsor. Since $GL_n$-torsors over fields are trivial, the latter map is a smooth-Nisnevich covering. The resulting map 
$$
\sqcup_n Z_n \ra \sX
$$
is then the desired smooth-Nisnevich covering by a countable union of qcqs algebraic spaces.
\section{Motivic filtrations on algebraic spaces}\label{app:mot-algpsc}

\newcommand{\Fil}{\mathrm{Fil}}
In this appendix, we record the following tiny extension of motivic filtrations on the $K$-theory of qcqs derived algebraic spaces. Most of these results are direct consequences of the Knutson-Lurie Theorem~\ref{thm:cover}. We briefly recall our conventions regarding filtrations which follows \cite{e-morrow}. For $\sC$ a stable $\infty$-category, $\infty$-categories of \textbf{filtered objects} is given by $\sC^{\Z^{\op}} := \Fun((\bbZ, \geq)^{\op}, \sC)$, where $(\bbZ, \geq)$ denotes the totally ordered set of the integers and $\bbZ^{\delta}$ is the discrete category of the integers. If $\sC = \Spt$, then we denote this by $\mrm{FiltSpt}$. We write filtered objects of $\sC$ as $\Fil^\star M$, where $M$ is an object of $\sC$ which implicitly means that there is a morphism $\Fil^{-\infty}M:=\colim_{j\to\-\infty}\Fil^jM\to M$ in $\sC$. The filtration is said to be \textbf{exhaustive} when the latter morphism is an equivalence; more strongly we say that it is {\em $\bbN$-indexed} when $\Fil^jM\to M$ is an equivalence for all $j\le 0$. The filtration is said to be \textbf{complete} if $\lim_{j\to\infty}\Fil^jM=0$.

If $\sC$ is presentably symmetric monoida,  then $\sC^{\Z^{\op}}$ admit canonical symmetric monoidal structures given by Day convolution. In particular, we have the $\infty$-category of \emph{filtered $\bbE_{\infty}$-algebras} $\CAlg(\sC^{\Z^{\op}})$. A filtered object with underlying $\bbE_{\infty}$-algebra object $M$ which is compatibly a $\bbE{\infty}$-algebra in filtered objects is called a \textbf{multiplicative filtration}.

\begin{thm}[Motivic filtrations on derived algebraic spaces]\label{thm:deralgspc} There exists a finitary, $\mathbb{N}$-indexed, multiplicative Nisnevich sheaf of filtered spectra
\begin{equation}\label{eq:filt-mot}
\mrm{Fil}^{\star}_{\mrm{mot}}K: (\mrm{dAlgSpc}^{\mrm{qcqs}})^{\op} \rightarrow \mrm{CAlg}(\mrm{FilSpt}).
\end{equation}
with the following properties:
\begin{enumerate}
\item for any $X \in \mrm{dAlgSpc}^{\mrm{qcqs}}$, its graded pieces are given by 
\[
\mrm{gr}^j_{\mrm{mot}}K(X) \simeq R^{\mrm{dAff}}\mathbb{Z}(j)^{\mrm{mot}}(X)[2j]
\]
We set for any $X \in \mrm{dAlgSpc}^{\mrm{qcqs}}$, $H^i_{\mrm{mot}}(X; \mathbb{Z}(j)) := H^i(R^{\mrm{dAff}}\mathbb{Z}(j)^{\mrm{mot}}(X))$, and obtain an Atiyah-Hirzebruch spectral sequence
\[
E_2^{i,j} := H^{i-j}_{\mrm{mot}}(X; \mathbb{Z}(-j)) \Rightarrow K_{-i-j}(X).
\]
\item If $X$ admits a Nisnevich cover by derived affine schemes $\{ X_i \rightarrow X\}$ where the classical locus $X_i^{\mrm{cl}}$ has finite valuative dimension, then the filtration is complete and the Atiyah-Hirzebruch spectral sequence is convergent.
\item the Atiyah-Hirzebruch spectral sequence degenerates rationally and
\[
K_n(X) \otimes_{\mathbb{Z}} \mathbb{Q} \cong \bigoplus_{j \geq 0} H_{\mot}^{2j-n}(X; \mathbb{Z}(j)) \otimes_{\mathbb{Z}} \mathbb{Q}. 
\]
\end{enumerate}
\end{thm}

\begin{proof} Thanks to Proposition~\ref{prop:localizing-nis}, the formation of algebraic $K$-theory on qcqs derived algebraic spaces is right Kan extended from derived affine schemes. Hence, the filtration~\eqref{eq:filt-mot} is defined by taking the right Kan extension of the motivic filtrations built in \cite{bouis-mixed, e-morrow} from derived affine schemes; in particular the underlying object of the filtration remains algebraic $K$-theory because the latter is right Kan extended from derived affine schemes. The claimed properties are then formal consequences of this definition and the equivalence~\eqref{eq:nis-general}.

The only thing to explain concerns the completeness claim of (2). If $X$ is a derived algebraic space, we note that $\mrm{Fil}^{\star}_{\mot}K(X)$ is complete whenever $\lim_{\star \rightarrow \infty} \mrm{Fil}^{\star}_{\mot}K(X) \simeq 0$. By~\eqref{eq:nis-general} we can write $\mrm{Fil}^{\star}_{\mot}K(X)$ as the totalization of $\mrm{Fil}^{\star}_{\mot}K(Y^{\bullet})$ where each $Y^{n}$ has the property that $\lim \mrm{Fil}^{\star}_{\mot}K(Y^{n}) \simeq 0$; see \cite[Proposition 4.50]{bouis-mixed} for the latter completeness result. The claim then follows from the fact that limits commute. 

\end{proof}

\bibliographystyle{alphamod}

\let\mathbb=\mathbf

{\small
\bibliography{references}
}

\parskip 0pt

\end{document}